\documentclass[11pt, a4paper]{amsart}

\usepackage[utf8]{inputenc}
\usepackage[T1]{fontenc}
\usepackage{lmodern} 
\usepackage[left=2.5cm,right=2.5cm,top=2.5cm,bottom=2.5cm, marginpar=2.0cm]{geometry}
\usepackage{amsmath,amsthm, amssymb,amsfonts,nicefrac}
\usepackage{graphicx}
\usepackage{textcomp}
\usepackage[usenames, dvipsnames]{xcolor}
\usepackage{enumerate}
\usepackage[shortlabels]{enumitem}
\usepackage{comment}
\usepackage[colorlinks, linktocpage, citecolor = blue, linkcolor = blue,backref]{hyperref}
\usepackage{framed}
\usepackage{thmtools}
\usepackage{caption}
\usepackage{subcaption}
\usepackage{setspace}
\usepackage{xspace}
\usepackage[normalem]{ulem}

\setlist[enumerate]{leftmargin=7mm,topsep=0pt,itemsep=-1ex,partopsep=1ex,parsep=1ex,label=\rm{(\roman*)}}
\setlist[enumerate,2]{leftmargin=7mm,topsep=0pt,itemsep=-1ex,partopsep=1ex,parsep=1ex,label=\rm{(\alph*)}}
\setlist[enumerate]{leftmargin=7mm,topsep=0pt,itemsep=-1ex,partopsep=1ex,parsep=1ex,label=\rm{(\alph*)}}
\setlist[itemize]{leftmargin=5mm,topsep=0pt,itemsep=-1ex,partopsep=1ex,parsep=1ex,label=\raisebox{0.25ex}{\tiny$\bullet$}}
\setlist[itemize,2]{leftmargin=5mm,topsep=0pt,itemsep=-1ex,partopsep=1ex,parsep=1ex,label=\raisebox{0.25ex}{--}}

\usepackage[cmtip]{xy}
\xyoption{pdf}
\xyoption{color}
\xyoption{all}

 \newcommand{\eq}[1][r]
   {\ar@<-3pt>@{-}[#1]
    \ar@<-1pt>@{}[#1]|<{}="gauche"
    \ar@<+0pt>@{}[#1]|-{}="milieu"
    \ar@<+1pt>@{}[#1]|>{}="droite"
    \ar@/^2pt/@{-}"gauche";"milieu"
    \ar@/_2pt/@{-}"milieu";"droite"}

\usepackage{tikz}
\usetikzlibrary{cd,arrows,positioning, decorations.pathreplacing}
\tikzset{>=stealth}
\tikzcdset{arrow style=tikz}
\tikzset{link/.style={column sep=1.8cm,row sep=0.23cm}}
\tikzset{link2/.style={column sep=0.4cm,row sep=0.1cm}} 
\tikzset{map/.style={row sep=0em, column sep=0em}}
\tikzset{c/.style={every coordinate/.try}}

\renewcommand{\to}{\longrightarrow}
\newcommand{\rat}{\dashrightarrow}
\newcommand{\iso}{\simeq}

\def\dashmapsto{\mapstochar\dashrightarrow}

\newcommand{\I}{\ensuremath{\mathrm{I}}\xspace}
\newcommand{\II}{\ensuremath{\mathrm{II}}\xspace}
\newcommand{\III}{\ensuremath{\mathrm{III}}\xspace}
\newcommand{\IV}{\ensuremath{\mathrm{IV}}\xspace}

\DeclareMathOperator{\Bir}{Bir}
\DeclareMathOperator{\GL}{GL}
\DeclareMathOperator{\SL}{SL}
\DeclareMathOperator{\PGL}{PGL}
\DeclareMathOperator{\PSU}{PSU}

\DeclareMathOperator{\SO}{SO}

\DeclareMathOperator{\Gal}{Gal}
\DeclareMathOperator{\Spec}{Spec}

\DeclareMathOperator{\Autz}{\mathrm{Aut}^{\circ}}

\DeclareMathOperator{\Aut}{\mathrm{Aut}}

\newcommand{\N}{\ensuremath{\mathbb{N}}}
\newcommand{\Z}{\ensuremath{\mathbb{Z}}}
\newcommand{\Q}{\ensuremath{\mathbb{Q}}}
\newcommand{\F}{\ensuremath{\mathbb{F}}}
\newcommand{\R}{\ensuremath{\mathbb{R}}}
\newcommand{\C}{\ensuremath{\mathbb{C}}}
\newcommand{\p}{\ensuremath{\mathbb{P}}}
\renewcommand{\P}{\ensuremath{\mathbb{P}}}
\newcommand{\A}{\ensuremath{\mathbb{A}}}
\newcommand{\G}{\ensuremath{\mathbb{G}}}
\newcommand{\SSS}{\ensuremath{\mathbb{S}^1}}

\renewcommand{\H}{{\mathrm{H}}}
\newcommand{\Id}{{\mathrm{Id}}}
\newcommand{\Pic}{{\mathrm{Pic}}}

\newcommand{\NS}{{\mathrm{NS}}}

\newcommand{\NE}{{\mathrm{NE}}}

\renewcommand{\O}{\mathcal{O}}
\newcommand{\QQ}{\mathcal{Q}}
\newcommand{\RR}{\mathcal{R}}
\newcommand{\TT}{\mathcal{T}}
\newcommand{\V}{\mathcal{V}}
\newcommand{\U}{\mathcal{U}}
\renewcommand{\SS}{\mathcal{S}}
\newcommand{\W}{\mathcal{W}}
\newcommand{\PP}{\mathcal{P}}
\newcommand{\FF}{\mathcal{F}}

\newcommand{\OFa}{\mathcal{O}_{\mathbb{F}_a}}
\newcommand{\s}[1]{s_{#1}}

\renewcommand{\k}{\mathbf{k}}
\newcommand{\K}{\mathbf{K}}

\newcommand{\bk}{\overline{\mathbf{k}}}
\newcommand{\fR}{\mathfrak{R}}

\makeatletter
\@namedef{subjclassname@2020}{%
  \textup{2020} Mathematics Subject Classification}
\makeatother

\newtheorem{theorem}{Theorem}[section]
\newtheorem*{theorem*}{Theorem}
\newtheorem{corollary}[theorem]{Corollary}
\newtheorem{lemma}[theorem]{Lemma}
\newtheorem{proposition}[theorem]{Proposition}

\theoremstyle{definition}
\newtheorem{definition}[theorem]{Definition}
\newtheorem{remark}[theorem]{Remark}
\newtheorem{example}[theorem]{Example}

\author{Ronan Terpereau}
\address{Univ. Lille, CNRS, UMR 8524 - Laboratoire Paul Painlev\'e, F-59000 Lille, France}
\email{ronan.terpereau@univ-lille.fr}

\author{Susanna Zimmermann}
\address{Universit\'e Paris-Saclay, CNRS, UMR 8628 - Laboratoire de math\'ematiques d’Orsay, 91405 Orsay, France}
\email{susanna.zimmermann@universite-paris-saclay.fr}


\thanks{
Both authors were supported by the ANR Project FIBALGA ANR-18-CE40-0003-01.
This work received partial support from the French ``Investissements d\textquoteright Avenir'' program and from project ISITE-BFC (contract ANR-lS-IDEX-OOOB). 
During this project, S.Z. was supported the Project \'Etoiles montantes of the R\'egion Pays de la Loire, the Centre Henri Lebesque, by the ERC StG Saphidir 101076412 and the Institut Universitaire de France.}

\begin{document}
\title[real forms of Mori fiber spaces with many symmetries]{real forms of Mori fiber spaces with many symmetries}

\begin{abstract}
We determine the rational real forms of the complex Mori fiber spaces for which the identity component of the automorphism group is a maximal connected algebraic subgroup of $\mathrm{Bir}(\mathbb{P}_{\mathbb{C}}^{3})$. This yields a list of maximal connected algebraic subgroup of $\mathrm{Bir}(\mathbb{P}_{\mathbb{R}}^{3})$.
We furthermore determine the equivariant Sarkisov links starting from these rational real forms. 
This article is the first step towards classifying all the maximal connected algebraic subgroups of $\mathrm{Bir}(\mathbb{P}_{\mathbb{R}}^{3})$.
\end{abstract}

\maketitle

\tableofcontents

\section{Introduction}

\subsection{What is known}
Consider an algebraic group $G$ acting regularly and faithfully on a variety $Z$. When we replace $Z$ with a variety $X$ that is birationally equivalent to $Z$, the biregular $G$-action on $Z$ becomes a birational $G$-action on $X$. Consequently, $G$ can be viewed as a subgroup of $\Bir(X)$, and we refer to it as an algebraic subgroup of $\Bir(X)$; see  Section~\ref{ss:algebraic groups} for a precise definition. Classifying the algebraic subgroups of $\Bir(X)$ for a given low-dimensional variety $X$ can be quite challenging. In the present article we focus on the case where $X$ is rational of dimension $2$ or $3$.

Enriques classified in \cite{enriques1893sui} the maximal connected algebraic subgroups of $\Bir(\p^2_\C)$, showing that they are conjugate to $\Aut(\p_\C^2)$, $\Autz(\p^1_\C\times\p^1_\C)$, and $\Aut(\F_{n,\C})$ with $n\geq2$ (see \cite{Ume82b} for a modern treatment). In fact, this result holds over any algebraically closed field $\k$. 
Using methods from the Minimal Model Program for smooth projective surfaces, this list follows from the fact that any connected algebraic subgroup of $\Bir(\p^2_\k)$ is conjugate to a subgroup of $\Autz(S)$, where $S$ is a rational minimal smooth projective surface over $\k$. 

A list of the maximal connected algebraic subgroups of $\Bir(\P_{\C}^{3})$ was obtained in the 1980's by Umemura \cite{Ume80,Ume82a,Ume82b,Ume85,Ume88} and Mukai-Umemura \cite{MU83}. More precisely, Umemura proved that any connected algebraic subgroup of $\Bir(\P_{\C}^{3})$ is contained into a maximal one, that any maximal connected algebraic subgroup of $\Bir(\P_{\C}^{3})$ is conjugate to some $\Autz(X)$, where $X$ is a minimal smooth rational projective complex threefold (a smooth projective variety $X$ is called \emph{minimal} if any birational morphism $X \to X'$ with $X'$ smooth is an isomorphism), and he gave a list of all the minimal smooth rational projective complex threefolds $X$ such that $\Autz(X)$ is a maximal connected algebraic subgroup of $\Bir(\P_{\C}^{3})$.

Blanc-Fanelli-Terpereau studied in \cite{BFT23,BFT22} the connected algebraic groups acting on Mori fibrations $X \to Y$, over an algebraically closed field, with $X$ a rational threefold and $Y$ a surface or a curve. For such Mori fiber spaces, they considered the identity components of their automorphism groups and explore their equivariant birational geometry. In the end, this study allowed them to determine the maximal connected algebraic subgroups of $\Bir(\P_{\k}^{3})$, recovering most of the classification results of Umemura when $\k=\C$.
Their main results can be summarized as follows (the notation will be specified in Section  \ref{sec: Notation for the MFS}).

\begin{theorem}{\emph{(\cite[Theorems E]{BFT22}, see also \cite[Theorem (2.1)]{Ume85})}}  
\label{th: list of MFS corresponding to max alg subgroups of Cr3}
Let $Z$ be a rational projective threefold over an algebraically closed field $\k$ of characteristic zero. Then there is an $\Autz(Z)$-equivariant birational map $ Z \dashrightarrow Z$, where $X$ is one of the following Mori fiber spaces:  
\begin{center}
\scalebox{0.9}{
\begin{tabular}{lllrclll}
$\hypertarget{th:D_a}{(a)}$& A decomposable &$\p^1$-bundle & $\FF_a^{b,c}$&\hspace{-0.3cm}$\longrightarrow$& \hspace{-0.2cm}$\F_a$& with $a,b\ge 0$, $a\not=1$, $c\in \Z$, and \\
&&&&&& $(a,b,c)=(0,1,-1)$; or\\
&&&&&& $a=0$, $c \neq 1$, $b \geq 2$, $b\ge \lvert c\rvert $; or\\
&&&&&& $-a<c<a(b-1)$; or\\
&&&&&& $b=c=0$.\\
$\hypertarget{tth:D_b}{(b)}$& A decomposable &$\p^1$-bundle &$\PP_b$&\hspace{-0.3cm}$\longrightarrow$& \hspace{-0.2cm}$\p^2$& for some $b \geq 2$.\\
$\hypertarget{tth:D_c}{(c)}$& An Umemura &$\p^1$-bundle &$\U_a^{b,c}$&\hspace{-0.3cm}$\longrightarrow$& \hspace{-0.2cm}$\F_a$& for some $a,b\ge 1, c\ge 2$ with\\
&&&&&& $c<b$ if $a=1$; and\\
&&&&&& $c-2<ab$  and $c-2 \neq a(b-1)$ if $a \geq 2$.\\
$\hypertarget{tth:D_d}{(d)}$ &A Schwarzenberger\!\! &$\p^1$-bundle &$\SS_b$&\hspace{-0.3cm}$\longrightarrow$& \hspace{-0.2cm}$\p^2$& for some $b=1$ or $b\ge 3$.\\
$\hypertarget{tth:D_e}{(e)}$ &A &$\p^1$-bundle &$\V_{b}$&\hspace{-0.3cm}$\longrightarrow$& \hspace{-0.2cm}$\p^2$& for some $b\ge 3$.\\
$\hypertarget{tth:D_W}{(f)}$& A singular &$\p^1$-fibration &$\W_b$&\hspace{-0.3cm}$\longrightarrow$& \hspace{-0.2cm}$\P(1,1,2)$& for some $b \geq 2$.\\
$\hypertarget{tth:D_RR}{(g)}$ &A decomposable &$\p^2$-bundle &$\RR_{(m,n)}$&\hspace{-0.3cm}$\longrightarrow$& \hspace{-0.2cm}$\p^1$& for some $m \geq n\ge 0$, \\
&&&&&& with $(m,n) \neq (1,0)$ and\\ 
&&&&&&  $m=n$ or $m>2n$.\\
$\hypertarget{tth:D_QQg}{(h)}$ &An Umemura &quadric fibration  &$\QQ_g$&\hspace{-0.3cm}$\longrightarrow$& \hspace{-0.2cm}$\p^1$& for some homogeneous\\ 
&&&&&&    polynomial $g \in \k[u_0,u_1]$ of\\
&&&&&&    even degree with at least \\ 
&&&&&& four roots of odd multiplicity.\\
$\hypertarget{tth:D_bP3}{(i)}$& The &projective space & \multicolumn{3}{c}{$\P^3$.}  &\\
$\hypertarget{tth:D_bQ3}{(j)}$& The smooth & quadric & \multicolumn{3}{c}{$Q_3\subseteq \P^4$.}  &\\
$\hypertarget{tth:D_bP1112}{(k)}$& The weighted &projective space & \multicolumn{3}{c}{$\P(1,1,1,2)$.}  &\\
$\hypertarget{tth:D_bP1123}{(l)}$& The weighted &projective space & \multicolumn{3}{c}{$\P(1,1,2,3)$.} &\\
$\hypertarget{tth:D_Fano}{(m)}$& \multicolumn{6}{l}{A rational $\Q$-factorial Fano threefold of Picard rank $1$ with terminal singularities}\\
&\multicolumn{6}{l}{{not isomorphic to any of the cases \hyperlink{th:D_bP3}{$(i)$}-\hyperlink{th:D_bQ3}{$(j)$}-\hyperlink{th:D_bP1112}{$(k)$}-\hyperlink{th:D_bP1123}{$(l)$}.}}
\end{tabular} } \vspace{-2.7mm}
\end{center}
Moreover, in cases  \hyperlink{th:D_a}{$(a)$}-\hyperlink{th:D_bP1123}{$(l)$}, each $\Autz(X)$ is a maximal connected algebraic subgroup of $\Bir(\P_{\k}^3)$, and the equivariant Sarkisov links between these Mori fibrations are given in \cite[Theorem F]{BFT22}.
\end{theorem}

\begin{remark}\label{rk: case of Y5 and X12}
When $\k=\C$, it follows from the work of Umemura and Mukai that the case \hyperlink{th:D_Fano}{$(m)$} in Theorem \ref{th: list of MFS corresponding to max alg subgroups of Cr3} can be replaced by 
\begin{align*}
\scalebox{0.9}{
\hypertarget{th:D_Fano_mbis}{(n)} \text{ The quintic del Pezzo threefold $Y_{5,\C}$ or the Mukai–Umemura Fano threefold $X_{12,\C}^{\mathrm{MU}}$.} }
\end{align*}
The automorphism groups $\Aut_\C(Y_{5,\C}) \simeq \PGL_{2,\C}$ and $\Aut_\C(X_{12,\C}^{\mathrm{MU}})  \simeq \PGL_{2,\C}$ are furthermore maximal connected algebraic subgroups of $\Bir(\P_{\C}^3)$, and there are no nontrivial equivariant Sarkisov links starting from $Y_{5,\C}$ or $X_{12,\C}^{\mathrm{MU}}$ (see Proposition \ref{prop:Y5-X12-no-links}).
\end{remark}

\subsection{What we do}
This article constitutes the first step towards extending the previous classification of the connected algebraic subgroups of $\Bir(\P_\k^3)$ to the case $\k=\R$. For the sake of completeness, the $2$-dimensional case is treated in Section \ref{sec: algebraic subgroup of Bir(P2)}, employing the classical Minimal Model Program (MMP) approach.
Subsequently, our attention shifts to the $3$-dimensional case, where we adopt a distinct strategy to obtain a (possibly incomplete) list of maximal connected algebraic subgroups of $\Bir(\P_\R^3)$.
Indeed, instead of using the MMP over $\R$ to reduce the study of the connected algebraic subgroups of $\Bir(\P_\R^3)$ to the study of automorphism groups of real rational Mori fiber spaces, we determine the rational real forms of the complex threefolds enumerated in Theorem~\ref{th: list of MFS corresponding to max alg subgroups of Cr3}. 
Remarkably, the identity component of the automorphism group of each such real form yields a maximal connected algebraic subgroup of $\Bir(\P_\R^3)$; see Corollary \ref{cor:maximality-extension}.

The process of determining these rational real forms begins by determining the full automorphism group of the complex threefolds enumerated in Theorem~\ref{th: list of MFS corresponding to max alg subgroups of Cr3} (this is done in Section \ref{sec: Notation for the MFS}). We then compute for each case the first Galois cohomology set with value in the automorphism group, which parametrizes the isomorphism classes of real forms.
What makes this approach possible is the fact that all families enumerated in Theorem~\ref{th: list of MFS corresponding to max alg subgroups of Cr3}, with the exception of family $\hypertarget{tth:D_h}{(h)}$, are actually defined over $\mathbb{Q}$. Therefore, in almost every case, there exists a canonical real structure, corresponding to the trivial form, which enables us to define the first set of Galois cohomology.
The case of family $\hypertarget{tth:D_h}{(h)}$, which is subtler, is addressed in Section \ref{sec: real forms of Umemura quadric fibrations}, where we determine all the real forms in this case.
It then remains to determine which real forms are rational, which we do not fully accomplish for the family $\hypertarget{th:D_h}{(h)}$. All of this is done in Sections \ref{sec: Case of decomposable projective bundles}-\ref{sec: Case of rank 1 Fano threefolds}. The proofs of our main results follow in Section \ref{sec: proof of Theorem list of some real MFS}. To the extent possible, we work over an arbitrary field of characteristic zero; a restriction to the field of real numbers $\R$ is only made when computations are significantly simplified.

Finally, in Section \ref{sec: equiva Sark links in dim 3}, we determine all equivariant birational maps starting from the rational real forms obtained. Notably, we did not find in the existing literature this type of study concerning the classical Fano threefolds $Y_5$ and $X_{12}$, over arbitrary fields of characteristic zero.

\smallskip

This new approach via computation of rational real forms is much simpler than the one via MMP, but it potentially yields an incomplete list of maximal connected algebraic subgroups of $\Bir(\p^3_\R)$. The reason for this is that there might exist connected algebraic subgroups of $\Bir(\p^3_\R)$ whose complexification is not maximal in $\Bir(\p^3_\C)$. Moreover, in the case of the family $\hypertarget{th:D_h}{(h)}$, there are several real forms for which we were not able to determine their rationality (see Section \ref{sec: real forms of Umemura quadric fibrations}).

\subsection{Statement of our main results}\label{ss:main results}
Our first main result, which serves as an appetizer before tackling the $3$-dimensional case over $\R$, consists of handling the $2$-dimensional case over an arbitrary perfect field.
Let us mention that a classification of the infinite algebraic subgroups of the Cremona group of rank two was obtained over $\R$ by Robayo-Zimmermann in \cite{RZ18} and over an arbitrary perfect base fields by Schneider-Zimmermann in \cite{SZ20}.

\begin{proposition}[{see Section  \ref{sec: algebraic subgroup of Bir(P2)} for the proof}] \label{prop:classification in dimension 2}
Let $\k$ be a perfect field of arbitrary characteristic.
Let $G$ be a connected algebraic subgroup of $\Bir(\p^2_\k)$. Then there exists a $G$-equivariant birational map $\psi\colon \p^2_\k\rat X$ such that $\psi G\psi^{-1}\subseteq\Autz(X)$ and $X$ is one of the following surfaces:
\begin{enumerate}[$(i)$]
\item\label{item: i for prop1} the projective plane $\p^2_\k$;
\item\label{item: ii for prop1} a Hirzebruch surface $\F_n$ with $n \in \N_{\geq 0} \setminus \{ 1\}$; 
\item\label{item: iii for prop1} the Weil restriction $\fR_{\K/\k}(\p^1_\K)$, where $\k \hookrightarrow \K$ is a quadratic field extension; or
\item\label{item: iv for prop1} a rational del Pezzo surface of degree $6$ with Picard rank $1$. 
\end{enumerate}
Moreover, all the above groups are maximal connected algebraic subgroups of $\Bir(\p^2_\k)$ and the corresponding conjugacy classes in $\Bir(\p^2_\k)$ are pairwise disjoint. 
\end{proposition}

\begin{remark}\item
\begin{itemize}
\item Cases \ref{item: i for prop1}  and \ref{item: ii for prop1} always occur, while Cases \ref{item: iii for prop1} and \ref{item: iv for prop1} may not occur (depending on the base field $\k$). For instance, if $\k=\bk$, then Cases \ref{item: iii for prop1} and \ref{item: iv for prop1} do not occur.  
\item In Case \ref{item: iii for prop1}, we have $\fR_{\K_1/\k}(\p^1_{\K_1}) \simeq \fR_{\K_2/\k}(\p^1_{\K_2})$ is and only if the fields $\K_1$ and $\K_2$ are $\k$-isomorphic; see \cite[Lemma 3.2 (3)]{SZ20}.
\end{itemize}
\end{remark}

In particular, when the base field is $\R$, Proposition \ref{prop:classification in dimension 2} specializes as follows:

\begin{corollary}[{see Section  \ref{sec: algebraic subgroup of Bir(P2)} for the proof}]\label{cor:classification in dimension 2 for the real numbers}
Let $G$ be a connected algebraic subgroup of $\Bir(\p^2_\R)$. Then there exists a $G$-equivariant birational map $\psi\colon \p^2_\R\rat X$ such that $\psi G\psi^{-1}\subseteq\Autz(X)$ and $X$ is one of the following surfaces:
\begin{enumerate}[$(i)$]
\item\label{mas:1} the projective plane $\p^2_\R$;
\item\label{mas:2} a Hirzebruch surface $\F_n$ with $n \in \N_{\geq 0} \setminus \{ 1\}$; or
\item\label{mas:3} the Weil restriction $\fR_{\C/\R}(\p^1_\C)$, which is the only nontrivial rational real form of $\P_\C^1 \times \P_\C^1$.
\end{enumerate}
Moreover, \ref{mas:1}--\ref{mas:3} are maximal connected algebraic subgroups in $\Bir(\p^2_\R)$ and
the corresponding conjugacy classes in $\Bir(\p^2_\R)$ are pairwise disjoint. 
\end{corollary}

\begin{remark}
It follows from Corollary \ref{cor:classification in dimension 2 for the real numbers} that if $\Autz(X)$ is a maximal connected algebraic subgroup of $\Bir(\p^2_\R)$, then $\Autz(X_\C)$ is a maximal connected algebraic subgroup of $\Bir(\p^2_\C)$. We do not know if this persists in dimension $3$.
\end{remark}

\medskip

Let us now move on to the $3$-dimensional case. 
As previously explained, we are able to provide a (possibly incomplete) list of maximal connected algebraic subgroups of $\Bir(\P_\R^3)$ simply by computing the rational real forms of the complex Mori fiber spaces obtained by Blanc-Fanelli-Terpereau (see Theorem \ref{th: list of MFS corresponding to max alg subgroups of Cr3}). 
For the sake of completeness, we actually compute the rational real forms of all complex Mori fiber spaces enumerated in Theorem \ref{th: list of MFS corresponding to max alg subgroups of Cr3}, considering all parameters rather than just those appearing in \emph{loc.cit.} Our main results in the $3$-dimensional case are Theorems \ref{th: first cases, over k, dim 3}, \ref{th: second cases, over R, dim 3}, and \ref{th: third cases Qg, over R, dim 3} and Corollary \ref{cor: max subrgoups Cr3 real}.

\begin{theorem}[{see Section  \ref{sec: proof of Theorem list of some real MFS} for the proof}]\label{th: first cases, over k, dim 3}
Let $\k$ be a field of characteristic zero.
Let $X$ be a Mori threefold over $\bk$ that belongs to one of the families $\hypertarget{tth:D_a}{(a)}$ with parameter $a \geq 1$, $\hypertarget{tth:D_b}{(b)}$, $\hypertarget{tth:D_c}{(c)}$, $\hypertarget{tth:D_e}{(e)}$, $\hypertarget{tth:D_f}{(f)}$, $\hypertarget{tth:D_g}{(g)}$, or $\hypertarget{tth:D_i}{(i)}$ from Theorem \ref{th: list of MFS corresponding to max alg subgroups of Cr3}. 
Then the trivial $\k$-form is the only rational $\k$-form of $X$. 
Moreover, for each of those families except  $\hypertarget{tth:D_f}{(f)}$, the trivial $\k$-form is the only $\k$-form of $X$ with a $\k$-point.  
\end{theorem}

\begin{remark}
A description of the automorphism groups for all the families considered in Theorem \ref{th: first cases, over k, dim 3}, over any base field $\k$ of characteristic zero, can be found in Section \ref{sec: Notation for the MFS}.
\end{remark}

In the following result, we set $\k=\mathbb{R}$ and determine all real forms, as well as the identity component of their automorphism groups, for the complex Mori threefolds within families $\hypertarget{tth:D_a}{(a)}$ with parameter $a=0$, $\hypertarget{tth:D_d}{(d)}$, $\hypertarget{tth:D_j}{(j)}$, $\hypertarget{tth:D_k}{(k)}$, $\hypertarget{tth:D_l}{(l)}$, and $\hypertarget{tth:D_n}{(n)}$. This encompasses all remaining cases except for family \hypertarget{tth:D_h}{(h)}, which will be addressed separately in Theorem \ref{th: third cases Qg, over R, dim 3}.

\begin{theorem}[{see Section  \ref{sec: proof of Theorem list of some real MFS} for the proof}]\label{th: second cases, over R, dim 3}\item
\begin{enumerate}[$(i)$]
\item\label{item 0 th2} If $X=(\P_\C^1)^3$ $($family $\hypertarget{tth:D_a}{(a)}$ with $a=b=c=0)$, then $X$ has six non-isomorphic real forms, two rational real forms and four real forms without real points, which are
\[
Z_{p,q,r}:=(\fR_{\C/\R}(\P_\C^1))^p \times (\P_\R^1)^q \times C^r, \ \text{with}\ p,q,r \in \N_{\geq 0} \ \text{satisfying}\ 2p+q+r=3.
\]
$($Here we denote by $C$ the nontrivial real form of $\P_\C^1$.$)$
Moreover, either $r=0$ and then $Z_{p,q,r}$ is rational or $r>0$ and then $Z_{p,q,r}$ has no real points. 
The identity component of the automorphism group of $Z_{p,q,r}$ is given by
\[
\Autz(Z_{p,q,r}) \simeq (\fR_{\C/\R}(\PGL_{2,\C}))^p \times (\PGL_{2,\R})^q \times (\SO_{3,\R})^r.
\]

\smallskip

\item\label{item 1 th2}  Let $X=\FF_{0,\C}^{b,c}$ with $b \geq 0$ and $(b,c)\neq(0,0)$ $($family $\hypertarget{tth:D_a}{(a)})$.
\begin{itemize}
\item If $b \neq |c|$, then $\FF_{0,\C}^{b,c}$ has no nontrivial real form with a real point.
\item If $b=|c|$, then $\FF_{0,\C}^{b,c}$ has exactly two non-isomorphic real forms with a real point, namely the trivial real form $\FF_{0,\R}^{b,c}$  and the real form $\mathcal{G}_b$ if $b=-c$ resp. $\mathcal{H}_b$ if $b=c$ $($with the notation of Definition \ref{def: Gb and Hb}$)$, corresponding to the real structures given by \eqref{eq: real structures for Gb and Hb} in Proposition \ref{prop: k-forms of Fabc with a=0}.
\end{itemize}
Moreover, if $Z$ is a real form of $\FF_{0,\C}^{b,c}$, then either $Z$ is rational or $Z$ has no real points, and the identity components of the automorphism groups of the nontrivial real forms $\mathcal{G}_b$ and $\mathcal{H}_b$ are described in Remark \ref{rk: auto group of Gb and Hb}.

\smallskip

\item\label{item 2 th2} If $X=\SS_{1,\C}$ $($family $\hypertarget{tth:D_d}{(d)}$ with $b=1)$, then $X$ has three non-isomorphic real forms:
\begin{itemize}
\item The trivial real form $\SS_{1,\R}$, which is rational, and such that $\Autz(\SS_{1,\R}) \simeq \PGL_{3,\R}$.
\item A nontrivial real form $\widetilde{\SS}_{1,\R}$, which is also rational, and such that  $\Autz(\widetilde{\SS}_{1,\R}) \simeq \PSU(1,2)$.
\item A nontrivial real form $\widehat{\SS}_{1,\R}$, which has no real points, and such that  $\Autz(\widehat{\SS}_{1,\R}) \simeq \PSU(3)$.
\end{itemize}
\vspace{-1mm}
These three real forms are described explicitly, in terms of real structures, in Example \ref{ex: real structures for S1}.

\smallskip

\item\label{item 3 th2} If $X=\SS_{b,\C}$ with $b \geq 2$ $($family $\hypertarget{tth:D_d}{(d)})$, then $X$ has two non-isomorphic real forms: 
\begin{itemize}
\item The trivial real form $\SS_{b,\R}$, which is rational, and such that $\Aut(\SS_{b,\R}) \simeq \PGL_{2,\R}$.
\item A nontrivial real form $\widetilde{\SS}_{b,\R}$, which is rational if $b$ is odd and has no real points if $b$ is even, and such that  $\Aut(\widetilde{\SS}_{b,\R}) \simeq \SO_{3,\R}$.
\end{itemize}
\vspace{-1mm}
The real structure corresponding to the real form $\widetilde{\SS}_{b,\R}$ is given in Example \ref{ex: real structures for Sb}. In particular, for $b$ odd, it is shown that $\widetilde{\SS}_{b,\R}$ is the total space of a $\P^1$-bundle over $\P_\R^2$.

\smallskip

\item\label{item 4 th2} If $X =Q_3$ $($family $\hypertarget{tth:D_j}{(j)})$, then $X$ has three non-isomorphic real forms: $Q^{3,2}$ and $Q^{4,1}$, which are both rational and with automorphism groups the indefinite special orthogonal groups $\SO(3,2)$ and $\SO(4,1)$ respectively, and $Q^{5,0}$, which has no real points. $($Here we denote by $Q^{r,s}$ a real quadric hypersurface in $\P_\R^4$ given by a real quadratic form of signature $(r,s)$.$)$ 

\smallskip

\item\label{item 5 th2}  If $X=\P(1,1,1,2)_\C$ $($family $\hypertarget{tth:D_k}{(k)})$ or $X=\P(1,1,2,3)_\C$ $($family $\hypertarget{tth:D_l}{(l)})$, then the trivial real form is the unique real form of $X$, up to isomorphism.

\smallskip

\item\label{item 6 th2} If $X$ is one of the two complex Fano threefolds $Y_{5,\C}$ or $X_{12,\C}^{\mathrm{MU}}$ $($family $\hypertarget{tth:D_n}{(n)})$, then $X$ has exactly two non-isomorphic real forms and both are rational with automorphism groups isomorphic to $\PGL_{2,\R}$, for the trivial real form, and $\SO_{3,\R}$, for the nontrivial real form.
\end{enumerate}

\smallskip

In particular, in each case, a real form of $X$ is either rational or has no real points.
\end{theorem}

\begin{remark}\item
\begin{itemize}
\item A description of the automorphism groups of the trivial real forms of the families considered in Theorem \ref{th: second cases, over R, dim 3} can be found in Section \ref{sec: Notation for the MFS}.
\item  A description of the automorphism groups of weighted projective spaces can be found in \cite[Section 8]{AA89}.
\item In general, the group schemes $\Aut(Z)$ and $\Autz(Z)$ say nothing about the rationality of $Z$ or the existence of real points; see for instance the case  of the nontrivial real form $Z=\widetilde{\SS}_{b,\R}$ of $\mathcal S_{b,\C}$ in Theorem \ref{th: second cases, over R, dim 3}.
\item It turns out that all the rational real forms obtained in Theorem \ref{th: second cases, over R, dim 3} are real Mori fiber spaces.
\end{itemize}
\end{remark}

It remains to determine the (rational) real forms of the family \hypertarget{tth:D_h}{(h)}, which stands as the only family among those listed in Theorem \ref{th: list of MFS corresponding to max alg subgroups of Cr3} whose members are not necessarily defined over $\Q$ in all cases (depending on the choice of $g \in \C[u_0,u_1]$). This is done in the next result:

\begin{theorem}[{see Section  \ref{sec: proof of Theorem list of some real MFS} for the proof}]\label{th: third cases Qg, over R, dim 3}
Let $g \in \C[u_0,u_1]$ be a homogeneous polynomial, which is not a square and is of degree $2n$ for some $n \in \N_{\geq 1}$.
\smallskip
\begin{enumerate}[$(i)$]
\item\label{QQgR:1} The complex threefold $\QQ_{g,\C}$ admits a real form if and only if there exist $\lambda \in \C^*$ and $\varphi=\begin{bmatrix}
a & b \\ c & d
\end{bmatrix} \in \SL_{2}(\C)$ such that $g':=\lambda (g \circ \varphi) \in \R[u_0,u_1]$, in which case $\QQ_{g,\C} \simeq \QQ_{g',\C}$.

\smallskip

\item\label{QQgR:2}  The algebraic group $\Aut(\QQ_g)$  fits into a short exact sequence
\begin{equation*}
1 \to \Aut(\QQ_g)_{\P^1} \to \Aut(\QQ_g) \to F \to 1,
\end{equation*} 
where $F$ is an algebraic subgroup of $\PGL_{2,\C}$ and $\Aut(\QQ_g)_{\P^1} \simeq \PGL_{2} \times \Z/2\Z$.
Moreover, if $g$ has at least three distinct roots, then $F$ is finite. If $g$ has exactly two distinct roots, then $F$ is isomorphic to either $\G_{m,\C} \rtimes \Z/2\Z$ $($when the two roots have the same multiplicities$)$ or $\G_{m,\C}$ $($when the two roots have different multiplicities$)$.

\smallskip

\item\label{QQgR:3}  The number of non-isomorphic real forms of $\QQ_{g,\C}$ is given in the following table:
\smallskip
{\small
\[
\begin{array}{|l|c|c|c|c|c|c|}
\hline 
\text{Subgroup } F \subseteq \PGL_{2,\C}   &  \multicolumn{3}{c|}{n \text{ even\ $(n \geq 2)$\ \ \ }} & \multicolumn{3}{c|}{n \text{ odd\ $(n \geq 3)$\ \ \ }} \\
\hline 
& \text{rational}& \text{\ \ \ ?\ \ \ }& \text{w/o real points}&\text{rational}& \text{\ \ \ ?\ \ \ }& \text{w/o real points}\\

\hline 
\hline 
  \G_{m,\C}      & 1&1& &1&1&   \\
                    \hline
                       \G_{m,\C} \rtimes \Z/2\Z      & 2&2& &2& 2&    \\
                    \hline
A_l,\ l\ge 1,\ l \text{ odd}    & 2&2&&2&2& \\
  \hline
 A_l,\ l\ge 2,\ l \text{ even}  & 4&4&&2&2& \\
   \hline
D_l,\ l\ge 3,\ l \text{ odd}   & 4&4&&2&2&\\
  \hline
 D_l,\ l\ge 2,\ l \text{ even}  & 6&6&4&3&3& 4 \\
   \hline
        E_6       & 2&2&4&1&1&4    \\
          \hline
              E_7    &4 &4&4&2&2&4   \\
                \hline
                    E_8      & 2&2&4&1&1&4   \\
                    \hline
\end{array} 
\]}

\smallskip

\noindent The column labeled ``rational'' indicates the number of real forms known to be rational $($for any choice of $g)$, the column labeled  ``w/o real points'' indicates the number of real forms known to be without real points $($for any choice of $g)$, and the column labeled  ``?'' indicates the number of real forms for which conclusions regarding rationality or the existence of real points are not known for an arbitrary $g$.

\smallskip

\item\label{QQgR:4} Let $Z$ be a real form of $\QQ_{g,\C}$, then $\Autz(Z)$ is isomorphic to\\[-19pt]
\begin{itemize}
\item $\PGL_{2,\R} \times \G_{m,\R}$ if $F$ is infinite and $Z$ is one of the real forms $\QQ_{g,\C}$ labeled  ``rational''; 
\item  $\SO_{3,\R} \times \G_{m,\R}$ if $F$ is infinite and $Z$ is one of the real forms $\QQ_{g,\C}$  labeled  ``?'';
\item $\PGL_{2,\R}$ if $F$ is finite and $Z$ is one of the real forms $\QQ_{g,\C}$ labeled  ``rational''; and
\item  $\SO_{3,\R}$ if $F$ is finite and $Z$ is one of the real forms $\QQ_{g,\C}$  labeled  ``?''.
\end{itemize}
\end{enumerate}
\end{theorem}

\begin{remark}\item
\begin{itemize}
\item When $F$ is finite, explicit equations for the real forms of $\QQ_{g,\C}$ labeled  ``rational'' or ``?'' can be obtained through the classical invariant theory of the finite subgroups of $\SL_{2,\C}$; see Section \ref{sec: gi with invariant theory} for details.
\item When $F$ is infinite, explicit equations for the real forms of $\QQ_{g,\C}$ can be found in Section \ref{ss:g has two roots}.
\item It turns out that all the real forms obtained in Theorem \ref{th: third cases Qg, over R, dim 3} are real Mori fiber spaces.
\end{itemize}
\end{remark}

\begin{corollary}[{see Section  \ref{sec: proof of Theorem list of some real MFS} for the proof}]\label{cor: max subrgoups Cr3 real}
Let $Z$ be a rational real projective threefold. If $\Autz(Z_\C)$ is a maximal connected algebraic subgroup of $\Bir(\P_\C^3)$, then $\Autz(Z)$ is a maximal connected algebraic subgroup of $\Bir(\P_\R^3)$. 
Moreover, $Z_\C$ is one of the complex Mori fiber spaces listed in Theorem \ref{th: list of MFS corresponding to max alg subgroups of Cr3} if and only if $Z$ is one of the rational real Mori fiber spaces listed in Theorems \ref{th: first cases, over k, dim 3}, \ref{th: second cases, over R, dim 3}, and \ref{th: third cases Qg, over R, dim 3} $($with parameters as in Theorem \ref{th: list of MFS corresponding to max alg subgroups of Cr3}$)$. 
\end{corollary}

\begin{remark}
We do not know if there exist rational real Mori fiber spaces $Z$ such that $\Autz(Z)$ is a maximal connected algebraic subgroup of $\Bir(\P_\R^3)$ but $\Autz(Z_\C)$ is not a maximal connected algebraic subgroup of $\Bir(\P_\C^3)$.
Answering this question will be the subject of a next article.
\end{remark}

Finally, we determine the equivariant Sarkisov links starting from one of the rational real Mori fiber spaces listed in Theorems \ref{th: first cases, over k, dim 3}, \ref{th: second cases, over R, dim 3}, and \ref{th: third cases Qg, over R, dim 3}. Let us note that the equivariant Sarkisov links starting from the trivial real forms of the complex Mori fiber spaces enumerated in Theorem \ref{th: list of MFS corresponding to max alg subgroups of Cr3} are described in \cite[Theorem F]{BFT22}. (The base field in \textit{loc.~cit.} is assumed to be of characteristic zero and algebraically closed; however, one can check that \cite[Theorem F]{BFT22} is indeed valid over any characteristic zero field.) Consequently, we are left with determining the equivariant Sarkisov links starting from the nontrivial rational real Mori fiber spaces listed in Theorems \ref{th: second cases, over R, dim 3} and \ref{th: third cases Qg, over R, dim 3}.

\begin{theorem}[{see Section  \ref{sec: proof of Theorem list of some real MFS} for the proof}]\label{th: eq Sarkisov links between real MFS}\item 
\begin{enumerate}[$(i)$]
\item\label{th: eq Sarkisov links between real MFS i} Let $X$ be one of the rank $1$ Fano threefolds $\widetilde{\SS}_{1,\R}$, $Q^{3,2}$, $Q^{4,1}$, or a real form of $Y_{5,\C}$ or $X_{12,\C}^{\rm{MU}}$. Then there are no nontrivial $\Autz(X)$-equivariant Sarkisov links starting from $X$.

\item\label{th: eq Sarkisov links between real MFS ii} There is one $\Autz(X)$-equivariant Sarkisov link of type IV starting from $X=\fR_{\C/\R}(\P_\C^1) \times \P_\R^1$ $($induced by the projections on both factors$)$, and there are two $\Autz(X)$-equivariant Sarkisov links of type IV starting from $X=\P_\R^1 \times \P_\R^1 \times \P_\R^1$ $($induced by the three natural projections on $\P_\R^1 \times \P_\R^1)$.

\item\label{th: eq Sarkisov links between real MFS iii} The only $\Aut(\mathcal G_1)$-equivariant link starting from $\mathcal G_1$ is a divisorial contraction $\psi \colon \mathcal G_1\to  Z$, where $ Z$ is isomorphic to the singular quadric hypersurface $Q^{1,3}$ in $\P_\R^4$. 
The only nontrivial Sarkisov link starting from $Z \simeq Q^{1,3}$ is the blow-up of the singular point $\psi\colon \mathcal G_1\to Z$. Moreover, $\psi  \Aut(\mathcal G_1) \psi^{-1} =\Aut(Q^{1,3})$.
\item\label{th: eq Sarkisov links between real MFS iv} If $b\geq2$, then there are no nontrivial $\Autz(\mathcal G_b)$-equivariant Sarkisov links starting from $\mathcal G_b$.
\item\label{th: eq Sarkisov links between real MFS v}  There is an $\Autz(\mathcal{H}_1)$-equivariant birational morphism $\delta\colon \mathcal{H}_1 \to \P_\R^3$ such that $\delta \Autz(\mathcal{H}_1) \delta^{-1}\newline \subsetneq \Autz(\P_\R^3)$.
\item\label{th: eq Sarkisov links between real MFS vi} If $b \geq 2$, then there are no nontrivial $\Autz(\mathcal{H}_b)$-equivariant Sarkisov links starting from $\mathcal{H}_b$.

\item\label{th: eq Sarkisov links between real MFS vii} Let $b\geq3$ be an odd integer. 
There is a birational involution $\varphi\colon \tilde{\mathcal S}_{b,\R} \dashrightarrow \tilde{\mathcal S}_{b,\R}$, which is a type II equivariant link such that $\varphi \Autz(\tilde{\mathcal S}_{b,\R}) \varphi^{-1}=\Autz(\tilde{\mathcal S}_{b,\R})$. Moreover, $\varphi$ is the unique equivariant Sarkisov link starting from $\tilde{\mathcal S}_{b,\R}$.

\item\label{th: eq Sarkisov links between real MFS viii} Let $g \in \R[u_0,u_1]$ be a homogeneous polynomial, which is not a square and is of degree $2n$ for some $n \in \N_{\geq 1}$. Let $h\in\R[u_0,u_1]$ be a polynomial, irreducible over $\R$.
Let $\pi\colon U_g \to \P_\R^1$ be a real form of $\pi_\C\colon \QQ_{g,\C} \to \P_\C^1$.
Then the birational map defined by
\[
\psi_h\colon U_g\rat U_{gh^2},\ [x_0:x_1:x_2:x_3;u_0:u_1]\mapsto[hx_0:hx_1:hx_2:x_3;u_0:u_1]
\] 
is a Sarkisov link of type II.
Moreover, $\psi_h^{-1} \Autz(U_{gh^2}) \psi_h \subseteq \Autz(U_g)$ with equality if and only if either $g$ has at least three $($complex$)$ roots or $gh^2$ has exactly two $($complex$)$ roots.
If furthermore $g$ has at least three distinct $($complex$)$ roots, then there are no other equivariant Sarkisov links starting from $U_g$.
\end{enumerate}
\end{theorem}

\begin{remark}\item
\begin{itemize}
\item The singular quadric threefold $Q^{1,3}$ is a real Mori fiber space, but its complexification $Q_{\C}^{1,3}$ is not a complex Mori fiber space (see Remark \ref{rk: case of Q13}). It provides an example of a real Mori fiber space $Z$ for which $\Autz(Z)$ is a maximal connected algebraic subgroup of $\Bir(\P_\R^3)$, yet $Z_{\C}$ does not appear among the complex Mori fiber spaces listed in Theorem \ref{th: list of MFS corresponding to max alg subgroups of Cr3} (even though $\Autz(Z_{\C})$ is a maximal connected algebraic subgroup of $\Bir(\P_\C^3)$).

\item Assume that $g$ (as in Theorem \ref{th: eq Sarkisov links between real MFS} \ref{th: eq Sarkisov links between real MFS viii}) has exactly two distinct $($complex$)$ roots.  Let $\pi\colon U_g \to \P_\R^1$ be a real form of $\pi_\C\colon \QQ_{g,\C} \to \P_\C^1$ with $U_g$ rational. Then $\Autz(U_g)$ is not a maximal connected algebraic subgroup of $\Bir(\P_\R^3)$; see Corollary \ref{cor: exactly two roots}.
\end{itemize}
\end{remark}

\subsection*{Acknowledgments}
We would like to thank 
Olivier Benoist, Michel Brion, Jean-Louis Colliot-Th\'el\`ene, Adrien Dubouloz, Alexander Duncan, Andrea Fanelli, Enrica Floris, Lena Ji, Fr\'ed\'eric Mangolte, Christian Urech, Isabel Vogt and Egor Yasinsky for interesting discussions related to this work.

\section{Preliminaries}\label{sec: preliminaries}

\subsection{Notation}
In this article, we work over a fixed field $\k$ of characteristic zero (unless stated otherwise) with algebraic closure $\bk$.
We denote by $\Gamma=\Gal(\bk/\k)$ the absolute Galois group of $\k$.
To the extent possible, we do not assume that $\k=\R$, except when this assumption considerably simplifies the situation. We refer to \cite{Ser02} for a general reference on Galois cohomology.

A \emph{variety} is a separated scheme of finite type over $\k$ that is geometrically integral. 
An \emph{algebraic group} is a group scheme of finite type over $\k$. 
By an algebraic subgroup, we always mean a closed subgroup scheme. 
The identity component of an algebraic group $G$ is the connected component containing the identity element, denoted as $G^0$; this is a normal algebraic subgroup of $G$, and the quotient $G/G^0$ is a finite group scheme. 
We denote by $\G_{m,\k}$ (or simply by $\G_m$ when no confusion is possible regarding the base field) the multiplicative group scheme over $\k$.

Let $\K/\k$ be a field extension. We denote by $X_{\K}=X\times_{\Spec(\k)}\Spec(\K)$ the variety over $\K$ obtained from the variety $X$ over $\k$ via the base change $\Spec(\K) \to \Spec(\k)$.
Moreover, if $X$ is a group scheme over $\k$, then $X_{\K}$ is a group scheme over $\K$. 
If $Y$ is a variety resp. a group scheme, over $\K$ and $X$ is a variety (resp. a group scheme) over $\k$ such that
$X_{\K} \simeq Y$ as varieties resp. as group schemes, over $\K$, then $X$ is called \emph{$\k$-form} of $Y$. 
Recall that, if the variety $Y$ is quasi-projective and the extension $\K/\k$ is Galois, then there is a one-to-one correspondence between the isomorphism classes of $\k$-forms of $Y$ and the equivalence classes of descent data on $Y$ \cite[Proposition 2.6]{BS64} (see also \cite[\href{https://stacks.math.columbia.edu/tag/0CDQ}{Tag 0CDQ}]{SP24}).
In the particular case where the field extension $\k \hookrightarrow \K$ is $\R \hookrightarrow \C$, giving a descent datum on $Y$ is equivalent to giving a \emph{real structure} (i.e.~an antiregular involution) on $Y$.

Assume now that the field extension $\k \hookrightarrow \K$ is finite.
If $Z$ is a quasi-projective algebraic variety over $\K$ resp. an algebraic group over $\K$, then we denote by $\fR_{\K/\k}(Z)$ the quasi-projective algebraic variety over $\k$ resp. the algebraic group over $\k$, that represents the \emph{Weil restriction functor} with respect to the field extension $\k \hookrightarrow \K$ (see \cite[Section  7.6]{BLR90} for a detailed treatment of the Weil restriction functor).

All along the text, a \emph{$\P^n$-bundle} over a scheme is always assumed to be locally trivial for the Zariski topology; in particular, it is the projectivization of a rank $n+1$ vector bundle when working over a regular Noetherian scheme.

\subsection{Mori fibrations and Sarkisov links}\label{sec: Mori fibrations and Sarkisov links}
Throughout Section  \ref{sec: preliminaries} we work over a base field $\k$ of characteristic zero, unless explicitly stated otherwise.

\begin{definition}\label{def:MoriFibration} 
Let $\pi\colon X \to Y$ be a surjective morphism of normal projective varieties. Then $\pi$ is called a \emph{Mori fibration}, and the variety $X$ a \emph{Mori fiber space}, if the following conditions are satisfied:
\begin{enumerate}
\item\label{MoriFibrationDefa} $\pi_*(\O_X)=\O_Y$ and $\dim(Y) < \dim(X)$;
\item\label{MoriFibrationDefb} $X$ is $\Q$-factorial with terminal singularities;
\item\label{MoriFibrationDefc} the anticanonical sheaf $\omega_X^{\vee}=\O_X(-\mathrm{K}_X)$ is $\pi$-ample; and
\item\label{MoriFibrationDefd} the relative Picard rank $\rho(X/Y):= \dim_\R(N_1(X/Y))$ is one, where $N_1(X/Y)$ is the real vector space of $1$-cycles of $X$ that are contracted by $\pi$ modulo numerical equivalence.
\end{enumerate}
\end{definition}

\begin{remark}\label{rk: MMP in arbitrary dim}
If $X$ is a smooth projective variety of arbitrary dimension whose canonical divisor $\mathrm{K}_X$ is not pseudo-effective (i.e.~not contained in the closure of the cone spanned by classes of effective divisors), then we can run an MMP; see \cite[Corollary 1.3.3]{BCHM10} when $\k=\bk$ and \cite[Example 2.18]{KM98} for the extension to non-algebraically closed fields. This applies in particular when $X$ is uniruled (see  \cite[Theorem 0.2]{BDPP13}), in which case the output is always a Mori fibration.
\end{remark}

\begin{proposition}\label{prop: MMP is G-eq}
Let $G$ be a connected algebraic group.
Let $X_0$ be a smooth projective $G$-variety on which we can run an MMP. 
Then each step of the MMP is $G$-equivariant, and if the output is a Mori fibration $\pi\colon X \to Y$, then $G$ acts on $Y$ and $\pi$ is $G$-equivariant. 
\end{proposition}

\begin{proof}
As the induced action of $G$ on the cone of effective curves $\NE(X_0)$ is trivial (since $G$ is connected), any contraction morphism $\phi\colon X_0 \to Y_0$ associated with an extremal ray of $\overline{\NE}(X_0)_{K_{X_0}<0}$ is $G$-equivariant. 
Moreover, the finite type $\O_{Y_0}$-algebra $\mathcal{A}:=\bigoplus_{m \geq 0}\phi_* \O_{X_0}(mK_{X_0})$ is canonically a \emph{$G$-equivariant sheaf} (see \cite[\href{https://stacks.math.columbia.edu/tag/03LE}{Tag 03LE}]{SP24} for the notion of equivariant sheaf), hence the variety $X^+:=\mathrm{Proj}(\mathcal{A})$ is endowed with a $G$-action and the birational map $X^+ \dashrightarrow X_0$ is $G$-equivariant.
\end{proof}

\smallskip

In \cite{Cor95, HM13}, a \emph{Sarkisov link} $\chi\colon X_1\rat X_2$ between two Mori fibrations $X_1/B_1$ and $X_2/B_2$ is one of the four types of birational maps between Mori fiber spaces as shown in Figure~\ref{fig:SarkisovTypes1}. 
Each morphism has relative Picard rank $1$, a morphism denoted by $div$ is a divisorial contraction, a morphism denoted by $fib$ is a Mori fiber space, and the horizontal dotted arrows are isomorphisms in codimension $1$. 

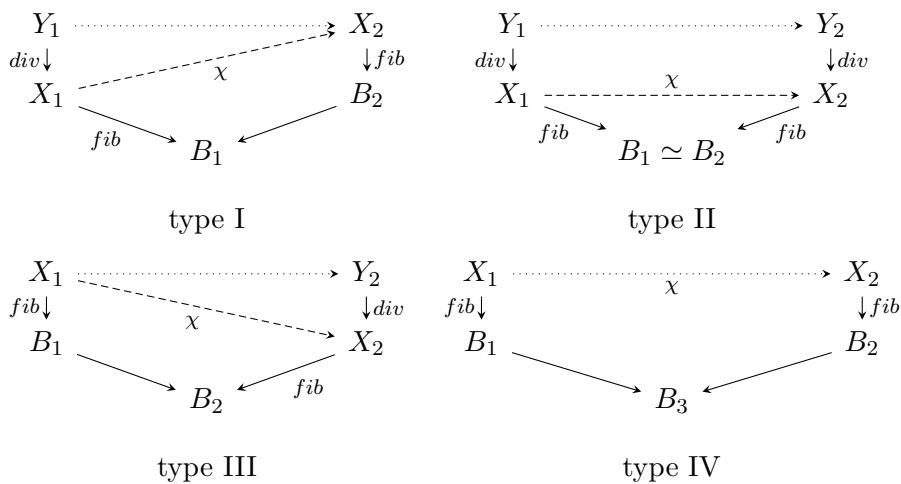
\begin{figure}[ht]
\[
{
\def\arraystretch{2.2}
\begin{array}{cc}
\begin{tikzcd}[ampersand replacement=\&,column sep=1.3cm,row sep=0.16cm]
Y_1\ar[dd,"div",swap]\ar[rr,dotted] \&\& X_2 \ar[dd,"fib"] \\ \\
X_1\ar[uurr,"\chi",dashed,swap]  \ar[dr,"fib",swap] \&  \& B_2 \ar[dl] \\
\& B_1 \&
\end{tikzcd}
&
\begin{tikzcd}[ampersand replacement=\&,column sep=.8cm,row sep=0.16cm]
 Y_1\ar[dd,"div",swap]\ar[rr,dotted]\& \& Y_2\ar[dd,"div"]\\ \\
X_1 \ar[rr,"\chi",dashed] \ar[dr,"fib",swap] \&  \& X_2 \ar[dl,"fib"] \\
\&  B_1\simeq B_2 \&
\end{tikzcd}
\\
\text{type }\I& \text{type }\II 
\\
\begin{tikzcd}[ampersand replacement=\&,column sep=1.3cm,row sep=0.16cm]
X_1 \ar[ddrr,"\chi",dashed,swap] \ar[dd,"fib",swap] \ar[rr,dotted]\&\&Y_2\ar[dd,"div"] \\ \\
B_1 \ar[dr] \& \& X_2 \ar[dl,"fib"] \\
\&B_2 \&
\end{tikzcd}
&
\begin{tikzcd}[ampersand replacement=\&,column sep=1.7cm,row sep=0.16cm]
X_1 \ar[rr,"\chi", dotted,swap] \ar[dd,"fib",swap]  \&\& X_2 \ar[dd,"fib"] \\ \\
B_1 \ar[dr] \& \& B_2 \ar[dl] \\
\& B_3 \&
\end{tikzcd}
\\
\text{type }\III & \text{type }\IV 
\end{array}
}
\]
\caption{The four types of Sarkisov links.}\label{fig:Sarkisov links}
\label{fig:SarkisovTypes1}
\end{figure}

\begin{remark}\label{rmk:sarkisov-2-rays-game}
A Sarkisov link is represented by a directed commutative diagram, a so-called \emph{Sarkisov diagram}. This diagram corresponds to the $2$-rays game starting from a variety at the top of the diagram. For instance, a link of type I or II satisfies $\rho(Y_1/B_1)=2$, indicating that there are at most two extremal rays above $B_1$. Each of these rays corresponds to a contraction, one of which corresponds to the divisorial contraction $Y_1\to X_1$. The other contraction is either a small contraction, thereby inducing a flip, flop, or antiflip (as the start of the dotted arrow), or it is a divisorial contraction or a fibration. In the latter two cases, the dotted arrow is an isomorphism and the contraction is either $Y_2\to X_2$ or $X_2\to B_2$ (in the case of links of type II or I, respectively).
We can see the $2$-rays game on other diagrams in a similar way.
\end{remark}

\begin{remark}\label{Rem:UpToIso}\item
\begin{itemize}
\item The composition of a Sarkisov link with an isomorphism of Mori fibrations is again a Sarkisov link. In what follows we will always identify two such links. 
\item In dimension $2$,  isomorphisms in codimension $1$ are isomorphisms and the horizontal dotted arrows in the four diagrams are therefore isomorphisms.
\end{itemize}
\end{remark}

Our main reason for introducing the notion of Sarkisov link is the following statement.

\begin{theorem}\label{thm:SP}\emph{(factorization of birational maps into Sarkisov links)}
\begin{enumerate}
\item\label{SP:1} Assume that $\k$ is a perfect field. Then any birational map between $2$-dimensional Mori fiber spaces is a composition of Sarkisov links.
\item\label{SP:2} Assume that $\k$ is a field of characteristic zero. Then any birational map between Mori fiber spaces is a composition of Sarkisov links. 
\end{enumerate} 
\end{theorem}
\begin{proof}
\ref{SP:1}: See \cite[Appendix]{Cor95} or \cite[Theorem 2.5]{Isk96}.\\
 \ref{SP:2}: 
If $\k$ is algebraically closed, the statement is \cite[Theorem 3.7]{Cor95} for the $3$-dimensional case and \cite[Theorem 1.1]{HM13} for the general case. Let $\k$ be a field of characteristic zero and $\Gamma:=\Gal(\bk/\k)$. Any MMP over $\k$ starting from a projective $\k$-variety $X$ corresponds to a $\Gamma$-equivariant MMP starting from $X_{\bk}$, and vice-versa. Using this, we can follow the proofs of \cite[Theorem 1.1]{HM13} and \cite[Theorem 3.1]{LZ20} to obtain the $\Gamma$-equivariant Sarkisov program; in fact \cite[Theorem 3.1]{LZ20} adapts the proof of \cite[Theorem 1.1]{HM13} for dimension two, but the adaption it can be done for any dimension. 
\end{proof}

Since we are mostly interested in Mori fiber spaces endowed with actions of connected algebraic groups, we will rather use the following statement, which is a refinement of Theorem \ref{thm:SP}.

\begin{theorem}\label{thm:G-equiv SP}\emph{(factorization of equivariant birational maps into equivariant Sarkisov links)}\\
Let $G$ be a connected algebraic group. Let $X_1/B_1$ and $X_2/B_2$ be two Mori fiber spaces endowed with a $G$-action. 
\begin{enumerate}
\item\label{G-equiv SP:1} Assume that $\k$ is a perfect field and that $\dim X_1=2$.
Then any $G$-equivariant birational map between $X_1$ and $X_2$ is a composition of $G$-equivariant Sarkisov links.
\item\label{G-equiv SP:2} Assume that $\k$ is a field of characteristic zero and that $\dim X_1\geq 2$. 
Then any $G$-equivariant birational map between $X_1$ and $X_2$ is a composition of $G$-equivariant Sarkisov links. 
\end{enumerate}
\end{theorem}
\begin{proof}
\ref{G-equiv SP:1}: 
We follow the proof of \cite[Theorem 1.2]{Flo20}, which uses Theorem~\ref{thm:SP}\ref{SP:2}, but instead of using \textit{loc.cit.} (which holds over algebraically closed fields), we use \cite[Theorem 3.1]{LZ20} (which adapts the proof of Theorem~\ref{thm:SP}\ref{SP:2} for surfaces over perfect fields).\\
\ref{G-equiv SP:2}: One can adapt the proof of \cite[Theorem 1.2]{Flo20} by using Theroem~\ref{thm:SP}\ref{SP:2} instead of \cite[Theorem 1.1]{HM13}.
\end{proof}

The next result will be useful in the proof of Proposition \ref{prop:classification in dimension 2} in Section  \ref{sec: algebraic subgroup of Bir(P2)}.

\begin{lemma}[{\cite[p.~250]{Cor95}}]\label{lem:SP-dP}
Let $Z$ be a surface on the top row of a Sarkisov diagram and $B$ the variety at the bottom row of the Sarkisov diagram. Then $-K_Z$ is relatively ample over $B$. 
In particular, if $B=\Spec(\k)$, then $Z$ is a del Pezzo surface. 
\end{lemma}
\begin{proof}
It follows from the definition of Sarkisov link that $\rho(Z/B)=2$ and that there are exactly two extremal contractions from $Z$ over $B$ (divisorial contractions and/or Mori fiber spaces), and so $\Pic(Z/B)_{\Q}=\Q f_1\oplus \Q f_2$, where $f_1,f_2$ are extremal curves in $Z$.  
Then $-K_{Z}\cdot f_i>0$ for $i=1,2$, and Kleinman's criterion implies that $-K_{Z}$ is relatively ample over $B$. 
\end{proof}

We finish this section with two observations concerning the equivariant Sarkisov program that will be useful in Section \ref{sec: equiva Sark links in dim 3}.

\begin{lemma}\label{lem:no links over bk}
The base field $\k$ is assumed to be of characteristic zero.
Let $\pi\colon X\to B$ be a Mori fibration such that $\pi_{\bk}\colon X_{\bk}\to B_{\bk}$ is a Mori fibration. If there are no $\Autz(X_{\bk})$-equivariant Sarkisov links of type I, II and III starting from $X_{\bk}\to B_{\bk}$, then there are no $\Autz(X)$-equivariant Sarkisov links of type I, II and III starting from $X\to B$. 
\end{lemma}

\begin{proof}
Let $\chi\colon X\rat X'$ be an $\Autz(X)$-equivariant Sarkisov link to a Mori fiber space $X'\to B'$. The following is summarized in the two diagrams below. 
We run an MMP $\Psi\colon X_{\bk}'\rat Z$ over $B_{\bk}'$ to a $\bk$-Mori fiber space $Z\to B'_{\bk}$. 
Then $\Theta:=\Psi\circ\chi_{\bk}\colon X_{\bk}\rat Z$ is an $\Aut(X_{\bk})$-equivariant birational map between $\bk$-Mori fiber spaces. 
It is the composition of $\Aut(X_{\bk})$-equivariant Sarkisov links by Theorem~\ref{thm:G-equiv SP}\ref{G-equiv SP:2}. 
By assumption, there are no $\Aut(X_{\bk})$-equivariant Sarkisov links of type I, II and III starting from $X_{\bk}$. Therefore, $\Theta$ is an isomorphism. 
In particular, if $\mu\colon \Gamma \times X_{\bk} \to X_{\bk},\ (\gamma,x) \mapsto \mu_\gamma(x)$ is a descent datum, then $\Gamma \times Z \to Z,\ (\gamma,z) \mapsto (\Theta \circ \mu_\gamma \circ \Theta^{-1})(z)$ is a descent datum on $Z$ corresponding to a $\k$-variety $\hat Z$.  
Denote by $\theta\colon X\to \hat Z$ the isomorphism (defined over $\k$) such that $\theta_{\bk}=\Theta$. 
Then $\psi:=\theta\circ\chi^{-1}$ is a birational map (defined over $\k$) and $\psi_{\bk}=\Psi$.
\[
{
\def\arraystretch{2.2}
\begin{array}{cc}
\begin{tikzcd}[ampersand replacement=\&,column sep=1.3cm,row sep=0.16cm]
X_{\bk}\ar[r,"\chi_{\bk}",dashed,swap]\ar[dd,"\pi_{\bk}"]\ar[rrr,bend left=20,"\Theta",dashed]    \&  X'_{\bk} \ar[dd]\ar[rr,dashed,"\Psi"] \&\& Z\ar[lldd] \\ \\
B_{\bk}\& B_{\bk}'\&\&
\end{tikzcd}
&\qquad
\begin{tikzcd}[ampersand replacement=\&,column sep=1.3cm,row sep=0.16cm]
X\ar[r,"\chi",dashed,swap]\ar[dd,"\pi"]\ar[rrr,bend left=20,"\theta",dashed]   \&  X' \ar[dd]\ar[rr,dashed,"\psi"] \&\& \hat Z\ar[lldd] \\ \\
B\& B'\&\&
\end{tikzcd}
\end{array}
}
\]
Note that $\psi$ is a birational contraction (its inverse does not contract any divisor), because $\Psi$ is a birational contraction. Then $1=\rho(X'/B')\leq \rho(\hat Z/B')=1$ implies that $\psi$ is an isomorphism in codimension $1$.  
Since $Z\to B_{\bk}'$ is a Mori fiber space, also $X'_{\bk}\to B'_{\bk}$ is a Mori fiber space. It now follows from our assumptions that $\chi_{\bk}$ and hence $\chi$ are isomorphisms in codimension $1$. 
\end{proof}

\begin{lemma}\label{lem:contractions}
The base field $\k$ is assumed to be of characteristic zero.
Let $\eta\colon Y\to X$ be a $K_Y$-extremal divisorial contraction between two terminal $\Q$-factorial $\k$-varieties contracting a divisor $E$ onto a curve $C\subset X$. Suppose that $E_{\bk}=E_1+\cdots +E_n$, where $E_1,\cdots, E_n$ are hypersurfaces in $Y_{\bk}$.
Then there exist birational morphisms $\eta_1\colon Y_{\bk}\to Z_1$ contracting $E_1$ and $\eta_{i+1}\colon Z_i\to Z_{i+1}$ contracting $\eta_i\cdots\eta_1(E_{i+1})$ for $i=1,\dots,n-1$ with $Z_n=X_{\bk}$ such that $\eta_{\bk}=\eta_n\circ\cdots\circ\eta_1$. 
Moreover, if $C$ is smooth, then each $\eta_i$ contracts $E_{i+1}$ onto a smooth curve isomorphic to $C_{\bk}$. 
\end{lemma}

\begin{proof}
Since $\rho(Y/X)=1$, the hypersurface $E$ is irreducible, and $\Gamma$ permutes $E_1,\dots,E_n$. 
Let $f_1\subset E_1$ be any curve contracted onto a point by $\eta_{\bk}$ and such that the $\Gamma$-orbit of $f_1$ consists of curves $f_1,\dots,f_n$ with $f_i\subset E_i$, $i=1,\dots,n$.  
Then 
\[
0>K_{Y_{\bk}}\cdot(f_1+\cdots+f_n)=nK_{Y_{\bk}}\cdot f_i,\qquad \text{for}\ i=1,\dots,n.
\] 
In particular, there is a birational contraction $\eta_1\colon Y_{\bk}\to Z_1$ contracting $f_1$. Let $C_i:=\eta_1(f_i)$, $i=2,\dots,n$. Then $\eta_1^*C_i=f_1+f_i$. Then
\[
K_{Z_1}\cdot C_i=\eta_1^*K_{Z_1}\cdot \eta_1^*C_i=K_{Y_{\bk}}\cdot(f_1+f_i)=2K_{Y_{\bk}}\cdot f_i<0.
\]
In particular, there is a birational contraction $\eta_2\colon Z_1\to Z_2$ contracting $C_2$, and $Z_2$ is again terminal and $\Q$-factorial. 
Inductively, we show the existence of the contractions $\eta_3,\dots,\eta_n$ in the same way.   
Then the birational map $\eta_n\circ\cdots\circ\eta_1\colon Y_{\bk}\to Z_n$ contracts $E_{\bk}$ and hence corresponds to a birational map $\eta'\colon Y\to Y'$ contracting $E$. By the uniqueness of contraction of extremal rays, there is an isomorphism $\alpha\colon Y'\to X$ with $\alpha\circ\eta'=\eta$. 
The last claim follows from the construction of the $\eta_i$. 
\end{proof}

\subsection{Algebraic subgroups of \texorpdfstring{$\Bir(X)$}{Bir(X)} and regularization}\label{ss:algebraic groups}

\begin{definition}\label{def: rational action}
Let $X$ be a variety, and let $G$ be an algebraic group.
\begin{itemize}
\item We say that $G$ acts \emph{rationally} on $X$ if there exists a birational map 
\[\eta \colon G\times X\rat G\times X,\ (g,x)\dashmapsto(g,\eta_g(x))\] 
that restricts to an isomorphism $U\to V$ on dense open subsets $U,V\subseteq G\times X$, whose projections onto $G$ are surjective, and such that $\eta_{gh}=\eta_g\circ\eta_h$ for any $g,h\in G$. If moreover the kernel of the induced homomorphism $G(\k) \to \Bir(X),\ g \mapsto \eta_g$ is trivial (i.e.~if $G$ acts \emph{faithfully} on $X$), then $(G,\eta)$ is called an \emph{algebraic subgroup of $\Bir(X)$}.
\item If, in addition, $\eta$ is an automorphism of $G \times X$, then we way that $G$ acts \emph{regularly} on $X$ and that $(G,\eta)$ is an \emph{algebraic subgroup of $\Aut(X)$}.
\item
An algebraic subgroup $(G,\eta)$ of $\Bir(X)$ is called \emph{maximal} if it is maximal with respect to the inclusion among the algebraic subgroups of $\Bir(X)$, i.e.~for every commutative diagram
\[\xymatrix@R=4mm@C=2cm{
    G' \times X \ar@{-->}[r]^{\eta'}   & G' \times X  \\
    G \times X\ar@{-->}[r]^{\eta} \ar@{^{(}->}[u] & G \times X \ar@{^{(}->}[u]
  }\]  
induced by an inclusion of algebraic groups $\iota\colon G \hookrightarrow G'$, we have $G=G'$ and $\iota=\Id$.
\end{itemize}
\end{definition}

\begin{remark}\label{rmk:algebraic subgroup}\item
\begin{itemize}
\item When $\eta$ is clear from the context, we will often drop it and simply write that $G$ is an algebraic subgroup of $\Bir(X)$ or of $\Aut(X)$.
\item  Let $\Gamma:=\Gal(\bk/\k)$. The canonical $\Gamma$-action on $X_{\bk}$ yields a canonical $\Gamma$-action on $\Bir(X_{\bk})$ by conjugation, and there is a canonical bijection $\Bir(X) \simeq \Bir(X_{\bk})^{\Gamma}$. We will therefore identify $\Bir(X)$ with the subgroup $\Bir(X_{\bk})^{\Gamma}$ of $\Bir(X_{\bk})$ in the rest of this article.
\item If $G$ is an algebraic subgroup of $\Bir(X)$, then $G_{\bk}$ is an algebraic subgroup of $\Bir(X_{\bk})$ and 
$G(\k)=\Bir(X)\cap G_{\bk}(\bk)$.
\end{itemize}
\end{remark}

We recall that if $X$ is a projective variety, then the contravariant functor 
\[
\begin{array}{ccccc}
\underline{\Aut}_\k(X)& \colon &\{\k\text{-schemes}\} &\to &\{\text{Groups}\} \\
& &  S & \mapsto & \Aut_S(X \times_\k S)
\end{array} 
\]
is represented by a group scheme $\Aut(X)$ locally of finite type over $\k$ (see \cite[Theorem~3.7]{MO67}).
Also, the identity component $\Autz(X)$ is a geometrically irreducible algebraic group over $\k$ (see \cite[Theorem~2.4.1]{Bri17}).
In particular, $\Autz(X)$ is a connected algebraic subgroup of $\Bir(X)$.

\begin{theorem}  \label{th alg subg of Cr3 are aut of Mori fib}
Let $n \in \N_{\geq 1}$. For each connected algebraic subgroup $G$ of $\Bir(\P_\k^{n})$, there exists an $n$-dimensional rational Mori fiber space $X$ and a birational map $\varphi\colon\p^n_\k\rat X$ such that $\varphi G\varphi^{-1}\subseteq\Autz(X)$.
\end{theorem}

\begin{proof}
The proof of this result is based on the proof of \cite[Theorem 2.4.4]{BFT22}, where the base field was assumed to be algebraically closed of characteristic zero.

Note that the connected algebraic subgroup $G_{\overline{\k}}$ of ${\Bir(\P_{\overline{\k}}^n)}$ is linear (\cite[Lemma 2.4.2]{BFT22}), and thus $G$ is linear (by \cite[Proposition 14.51~(6)]{GW10}).
According to the regularization theorem of Weil (valid over any field; see \cite[p.~375, Theorem]{Wei55}), the group $G$ is conjugated to an algebraic subgroup of $\Autz(X')$, where $X'$ is a rational variety. Replacing $X'$ by its smooth locus, we may moreover assume that $X'$ is smooth. Then, by a result of Sumihiro (\cite[Theorem 3.8]{Sum75}), $X'$ is covered by $G$-stable quasi-projective open subsets. Replacing $X'$ by one of these $G$-stable quasi-projective open subsets, we may assume that $X'$ is quasi-projective.
Taking a $G$-equivariant compactification (\cite[Theorem 2.5]{Sum75}), and then a $G$-equivariant resolution of singularities (see \cite[Proposition 3.9.1 and Theorem 3.36]{Kol07} or \cite[Theorem 8.11]{Vil14}), we may assume that the rational variety $X'$ is smooth and projective. 
Then we can run an MMP on $X'$ (see Remark \ref{rk: MMP in arbitrary dim}), which is $G$-equivariant as $G$ is connected, see Proposition \ref{prop: MMP is G-eq}. 
We obtain a Mori fiber space $X$, birational to $\P_{\k}^{n}$, on which $G$ acts faithfully and regularly.
Hence, there exists a birational map $\varphi\colon \P_\k^n \dashrightarrow X$ such that $\varphi G \varphi^{-1} \subseteq \Autz(X)$.
\end{proof}

\begin{remark}\label{rk: alg subg of Cr3 are aut of Mori fib in dim 2}
If $n=2$, then Theorem \ref{th alg subg of Cr3 are aut of Mori fib} holds over any perfect base field (not necessarily of characteristic zero). Indeed, the only place in the proof of Theorem \ref{th alg subg of Cr3 are aut of Mori fib} where we use the fact that $\mathrm{char}(\k)=0$ is for the existence of an equivariant resolution of singularities, but such a resolution always exist for a surface: it suffices to repeatedly alternate normalizing the surface with blowing-up singular points (see \cite{Lip78} or \cite{Art86}).
\end{remark}

\subsection{Algebraic tori of the Cremona group}\label{sec: algebraic tori of Bir(Pn)}
We are now more particularly interested in the case of algebraic tori of $\Bir(\P_\k^n)$. According to \cite[p.~508, Corollaire 4 (p. 522) and Corollaire 1 (p. 524)]{Dem77}, there are no tori of dimension greater than $n$ in $\Bir(\P_\k^n)$, and any $d$-dimensional split torus of $\Bir(\P_\k^n)$, with $d \in \{n-1,n\}$, is conjugate to a torus of $\Aut(\P_\k^n)$. More generally, for $n$-dimensional tori of $\Bir(\P_\k^n)$, we have the following result.

\begin{proposition}\label{prop:conjugate tori in Bir(Pn)}
Let $T_1$ and $T_2$ be two $n$-dimensional tori of $\Bir(\P_\k^n)$. 
Then $T_1 \simeq T_2$, as algebraic groups, if and only if $T_1$ and $T_2$ are conjugate in $\Bir(\P_\k^n)$. 
\end{proposition}

\begin{proof}
$(\Leftarrow)$: It is clear that if $T_1$ and $T_2$ are conjugate in $\Bir(\P_\k^n)$, then they are isomorphic as algebraic groups. 

$(\Rightarrow)$: We now suppose that $T_1$ and $T_2$ are isomorphic as algebraic groups.
According to Weil's regularization theorem, there exists a rational threefold $X_i$ endowed with a faithful $T_i$-action which makes it a toric variety. Let $U_i$ be the dense open $T_i$-orbit of $X_i$ (which is necessarily isomorphic to $T_i$ itself as $X_i$ is rational). Then the equivariant isomorphism $U_1 \simeq U_2$, induced by the isomorphism $T_1 \simeq T_2$, yields a birational equivariant map $\psi\colon X_1  \dashrightarrow X_2$ such that $\psi T_1 \psi^{-1}=T_2$, and the result follows.
\end{proof}

Let us now focus on the case $\k=\R$. We first recall the classification of the real tori (Lemma \ref{lem: real tori}) and then prove that an $n$-dimensional torus of  of $\Bir(\P_\R^n)$ is never a maximal connected algebraic subgroup (Corollary \ref{cor: n-dim torus is not max over R}).

\begin{lemma}\label{lem: real tori}
Any real torus of dimension $d \in \N_{\geq 0}$ is isomorphic to 
\[
(\fR_{\C/\R}(\G_{m,\C}))^p \times (\SSS)^q \times \G_{m,\R}^r\  \text{ with } 2p+q+r=d,
\] 
where $\SSS:=\Spec(\R[x,y]/(x^2+y^2-1))$ is the unit circle and $\fR_{\C/\R}(\G_{m,\C})$ is the Weil restriction of $\G_{m,\C}$ with respect to the field extension $\R \hookrightarrow \C$. 
Moreover, since all tori of dimension at most two are rational, any real torus is rational.
\end{lemma}

\begin{proof}
See \cite[Chapter 4, Section  10.1]{Vos98} or \cite[Lemma 1.5]{MJT18} for the first statement, and \cite[Theorem 4.74]{Vos77} for the second statement.
\end{proof}

\begin{remark}
It follows from the results recalled in this section that each isomorphism class of real tori of $\Bir(\P_\R^2)$ non-isomorphic to $\SSS$ yields a single conjugacy class in $\Bir(\P_\R^2)$. However, we will see in Section  \ref{sec: algebraic subgroup of Bir(P2)} that it is not true for real tori isomorphic to $\SSS$ (see Proposition \ref{prop: conjugacy classes of S1 in Bir(P2)}). 
\end{remark}

\begin{corollary} \label{cor: n-dim torus is not max over R}
Let $T \simeq \fR_{\C/\R}(\G_{m,\C})^p \times (\SSS)^q \times \G_{m,\R}^r$, with $2p+q+r=n \geq 1$, be a torus of $\Bir(\P_\R^n)$.  
Then  $T$ is conjugate to a torus of the real algebraic subgroup
\[
\Autz((\fR_{\C/\R}(\P_\C^1))^p \times (\P_\R^1)^{q+r}) \simeq (\fR_{\C/\R}(\PGL_{2,\C}))^p \times (\PGL_{2,\R})^{q+r}
\]
in $\Bir(\P_\R^n)$. In particular, an $n$-dimensional torus of $\Bir(\P_\R^n)$ is never a maximal connected algebraic subgroup of $\Bir(\P_\R^n)$.
\end{corollary}

\begin{proof}
The $n$-dimensional variety $(\fR_{\C/\R}(\P_\C^1))^p \times (\P_\R^1)^{q+r}$ is rational (since the real locus of the smooth projective real surface $\fR_{\C/\R}(\P_\C^1)$ is non-empty and connected; see \cite[Theorem 4.4.8]{Man20}), $\Aut(\P_\R^1) \simeq \PGL_{2,\R}$ contains strict algebraic subgroups isomorphic to $\G_{m,\R}$ and $\SSS$, and $\Autz(\fR_{\C/\R}(\P_\C^1)) \simeq \fR_{\C/\R}(\PGL_{2,\C})$ contains strict algebraic subgroups isomorphic to $\fR_{\C/\R}(\G_{m,\C})$. On the other hand, by Proposition \ref{prop:conjugate tori in Bir(Pn)}, two $n$-dimensional tori of $\Bir(\P_\R^n)$ are conjugate if and only if they are isomorphic as real algebraic groups. Hence, $T$ is conjugate to a torus of $\Autz((\fR_{\C/\R}(\P_\C^1))^p \times (\P_\R^1)^{q+r})$.
The last sentence follows from Lemma \ref{lem: real tori} and from the fact that $T$ is a strict algebraic subgroup of $\Autz((\fR_{\C/\R}(\P_\C^1))^p \times (\P_\R^1)^{q+r})$.
\end{proof}

\begin{remark}
It follows from Corollary \ref{cor:classification in dimension 2 for the real numbers} that tori are not maximal algebraic subgroups of $\Bir(\P_\R^2)$. However, for $n \geq 3$ and $1 \leq d <n$, we do not know if a $d$-dimensional torus can be a maximal algebraic subgroup of $\Bir(\P_\R^n)$.
\end{remark}

\subsection{Maximal connected algebraic subgroups of \texorpdfstring{$\Bir(X)$}{Bir(X)} and base change}

\begin{proposition}\label{prop:aut_groups_and_base_change}
Let $\k_1,\k_2$ be two fields with $\k_1 \hookrightarrow  \k_2$.
Let $X$ be a projective variety over $\k_1$. 
Then $(\Aut_{\k_1}(X))_{\k_2} \iso \Aut_{\k_2}(X_{\k_2})$ as group schemes over $\k_2$. 
In particular, $(\Aut_{\k_1}^{\circ}(X))_{\k_2} \iso \Aut_{\k_2}^{\circ}(X_{\k_2})$ as connected group schemes over $\k_2$. 
\end{proposition}

\begin{proof}
Let us start with an observation. Consider the functor 
\[H \colon \{\k_2\text{-schemes}\} \to \{\k_1\text{-schemes}\},\ (T \to \Spec(\k_2)) \mapsto  (T \to \Spec(\k_2)\to \Spec(\k_1)),\] where the morphism $\Spec(\k_2)\to \Spec(\k_1)$ is the one induced by the field inclusion $\k_1 \hookrightarrow \k_2$. It follows from the universal property of the fiber product (see \cite[\href{https://stacks.math.columbia.edu/tag/01JO}{Tag 01JO}]{SP24}) that if a functor $F\colon \{\k_1\text{-schemes}\} \to \{\text{Sets}\}$ is represented by a $\k_1$-scheme $S$, then the functor $F \circ H\colon \{\k_2\text{-schemes}\} \to \{\text{Sets}\}$ is represented by the $\k_2$-scheme $S_{\k_2}=S\times_{\k_1} \Spec(\k_2)$.

For $i \in \{1,2\}$, we denote by $F_i=\underline{\Aut}_{\k_i}(X)\colon \{\k_i\text{-schemes}\} \to \{\text{Groups}\}$ the functor represented by $\Aut_{\k_i}(X)$. 
By the observation above, the group scheme $(\Aut_{\k_1}(X))_{\k_2}$ represents the functor $F_1 \circ H$. Hence, $(\Aut_{\k_1}(X))_{\k_2} \iso \Aut_{\k_2}(X_{\k_2})$ as group schemes over $\k_2$ if and only if the group functors $F_2$ and $F_1 \circ H$ are isomorphic. 

Let $T \to \Spec(\k_2)$ be a $\k_2$-scheme. Then we have the following commutative diagram
\[\xymatrix@R=4mm@C=0.8cm{
    X_{\k_2} \times_{\k_2} T \ar[r] \ar[d]  & X_{\k_2} \ar[r] \ar[d] & X \ar[d] \\
    T \ar[r] & \Spec(\k_2)\ar[r]  & \Spec(\k_1)
  }
\]
whose two squares are Cartesian. In particular, by natural associativity of the fiber product, there is an isomorphism of 
schemes $X_{\k_2} \times_{\k_2} T  \simeq X \times_{\k_1} H(T)$, and so for each element $f \in \Aut_T(X_{\k_2} \times_{\k_2} T)$ we have a commutative diagram
\[
\xymatrix@R=2mm@C=0.8cm{
    &&&X_{\k_2} \times_{\k_2} T \eq[r] \ar[dd]_{\simeq}^f \ar[dl] & X \times_{\k_1} H(T)\ar[dd]_{\simeq} \\
   \Spec(\k_1) &\Spec(\k_2) \ar[l] &T \ar[l] && \\
   &&&X_{\k_2} \times_{\k_2} T \ar[lu] \eq[r] & X \times_{\k_1} H(T)
  }.
\] 
As a consequence, there is a group isomorphism
\[
\eta_T\colon \ F_2(T)=\Aut_T(X_{\k_2} \times_{\k_2} T) \simeq \Aut_{H(T)}(X \times_{\k_1} H(T))=(F_1 \circ H)(T). 
\] 
Moreover, one can check that, for every morphism of $\k_2$-schemes $f\colon T_1 \to T_2$, we have $\eta_{T_2} \circ F_2(f)=(F_1 \circ H)(f) \circ \eta_{T_1}$. Therefore the group functors $F_2$ and $F_1 \circ H$ are isomorphic, and so $(\Aut_{\k_1}(X))_{\k_2} \iso \Aut_{\k_2}(X_{\k_2})$ as group schemes over $\k_2$.  

Finally, the last statement follows from \cite[Theorem~2.4.1]{Bri17}. 
\end{proof}

\begin{remark}\label{rk:Gamma respects aut-action}
Let $X$ be a $\k$-variety. 
Then the Galois group $\Gamma=\Gal(\bk/\k)$ permutes the $\Autz(X_{\bk})$-orbits. 
Indeed, the $\Gamma$-action on $X_{\bk}$ induces a $\Gamma$-action on $\Autz(X_{\bk}) \simeq \Autz(X)_{\bk}$ such that, for every $\gamma \in \Gamma$, we have $\mu_\gamma(\Autz(X_{\bk}) \cdot p)=\Autz(X_{\bk})\cdot \mu_\gamma(p)$ for all $p\in X_{\bk}$. 
\end{remark}

The next result will be the key-ingredient to prove Corollary \ref{cor: max subrgoups Cr3 real}.

\begin{corollary}\label{cor:maximality-extension}
Let $\k_1,\k_2$ be two fields with $\k_1 \hookrightarrow  \k_2$.
Let $X$ be a projective variety over $\k_1$. 
If $\Autz_{\k_2}(X_{\k_2})$ is maximal in $\Bir(X_{\k_2})$, then $\Autz_{\k_1}(X)$ is maximal in $\Bir(X)$.  
\end{corollary}

\begin{proof}
We prove this statement by contraposition. If $\Autz_{\k_1}(X)$ is not maximal in $\Bir(X)$, then by Theorem \ref{th alg subg of Cr3 are aut of Mori fib} there exists a projective variety $Y$ over $\k_1$ and a birational map $\varphi\colon\ X \rat Y$ such that $\varphi \Autz_{\k_1}(X) \varphi^{-1} \subsetneq \Autz_{\k_1}(Y)$. 
Then Proposition~\ref{prop:aut_groups_and_base_change} yields
\[
\varphi_{\k_2} \Autz_{\k_2}(X_{\k_2}) \varphi_{\k_2}^{-1}  \simeq (\varphi \Autz_{\k_1}(X) \varphi^{-1})_{\k_2} \subsetneq (\Autz_{\k_1}(Y))_{\k_2} \simeq \Autz_{\k_2}(Y_{\k_2}),
\]
where $\varphi_{\k_2}=\varphi \times_{\k_1} \Id_{\Spec(\k_2)} \colon\ X_{\k_2} \rat Y_{\k_2}$ is the birational map induced by $\varphi$.
Thus $\Autz_{\k_2}(X_{\k_2})$ is not maximal in $\Bir(X_{\k_2})$, which finishes the proof.
\end{proof}

\begin{remark}
The converse of Corollary \ref{cor:maximality-extension} is not true in general. 
For instance, let $\k$ be a field for which there is a rational del Pezzo surface $X$ of degree $6$ and of Picard rank $1$. Then, according to Proposition \ref{prop:classification in dimension 2}, $\Autz_{\k}(X)$ is maximal in $\Bir(\P_{\k}^{2})$, while $\Autz_{\bk}(X_{\bk})\simeq\G_{m,\bk}^{2}$ is not maximal in $\Bir(\P_{\bk}^{2})$; see \cite[Section  4]{SZ20} for details on automorphism groups of del Pezzo surfaces of degree $6$.
\end{remark}

\section{Maximal connected algebraic subgroups of \texorpdfstring{$\Bir(\p^2)$}{Bir(P2)}} \label{sec: algebraic subgroup of Bir(P2)}
In this section we prove Proposition \ref{prop:classification in dimension 2} and Corollary \ref{cor:classification in dimension 2 for the real numbers}, and we then determine the conjugacy classes of the real algebraic tori of $\Bir(\P_\R^2)$. Recall that, given a perfect field $\k$, we denote its Galois group by $\Gamma:=\Gal(\bk/\k)$.

\smallskip

The following result is classic and will be useful to us several times later.

\begin{proposition}{\emph{(Ch\^atelet's theorem, see e.g. \cite[Proposition 4.5.10]{Poo17})}} \label{prop:chatelet}
Let $\k$ be an arbitrary field.
Let $X$ be a $\k$-form of $\P_{\bk}^n$ with $n \geq 1$. Then $X(\k)\neq \varnothing$ if and only if $X \simeq \P_{\k}^n$. In particular, the only rational $\k$-form of $\P_{\bk}^n$ is the trivial $\k$-form $\P_{\k}^n$ (up to isomorphism).
\end{proposition}

Let us now prove Proposition \ref{prop:classification in dimension 2} and then Corollary \ref{cor:classification in dimension 2 for the real numbers}.

\begin{proof}[Proof of Proposition \ref{prop:classification in dimension 2}]
Let $\k$ be a perfect field of arbitrary characteristic.
Let $G$ be a connected algebraic subgroup of $\Bir(\P_{\k}^{2})$.
By Theorem~\ref{th alg subg of Cr3 are aut of Mori fib} and Remark \ref{rk: alg subg of Cr3 are aut of Mori fib in dim 2}, $G$ is conjugate to an algebraic subgroup of $\Autz(X)$, where $X$ is a $2$-dimensional rational Mori fiber space. Hence, to determine the conjugacy classes of connected algebraic subgroups of $\Bir(\P_{\k}^{2})$, it suffices to list the Mori fiber spaces $X/B$ up to $\Autz(X)$-equivariant birational maps. 
Moreover, since every equivariant birational map between Mori fiber spaces can be decomposed into equivariant Sarkisov links (see Theorem \ref{thm:G-equiv SP}), we only need to check for the existence of equivariant Sarkisov links starting from $X$.

\smallskip

\textbf{Case 1.} Suppose first that the basis $B$ of the Mori fibration is $\Spec(\k)$. Then $X$ is a smooth rational del Pezzo surface with $\rho(X)=1$ and \cite[Theorem 2.6]{Isk96} implies that $\mathrm{K}_X^2\geq 5$ (see also \cite[Proposition 2.9]{SZ20}). We now consider the different possibilities for $\mathrm{K}_X^2$.
\begin{itemize}
\item If $\mathrm{K}_X^2=5$, then $X$ has a finite automorphism group since $\Aut(X_{\bk}) \simeq \mathfrak{S}_5$ is the symmetric group of degree $5$ (see \cite[Theorem 8.5.6]{Dol12}). Hence $\mathrm{K}_X^2=5$ cannot occur.
\item If $\mathrm{K}_X^2=6$, then $\Autz(X)$ is a $2$-dimensional torus and $X_{\bk}$ has exactly six $(-1)$-curves that form a hexagon, all preserved by $\Autz(X)$ (see \cite[Section  4]{SZ20}).  
Moreover, the only $\Autz(X_{\bk})$-fixed points in $X_{\bk}$ are the six intersection points of the $(-1)$-curves, and so there are no $\Autz(X)$-equivariant Sarkisov links starting from $X$ by Lemma~\ref{lem:SP-dP}.
This yields Case \ref{item: iv for prop1} of Proposition \ref{prop:classification in dimension 2}.
\item If $\mathrm{K}_X^2=7$, then $X$ has Picard rank at least $2$; indeed, $X_{\bk}$ has three $(-1)$-curves, two of which are disjoint and the third one intersects the other two, and so $\NS(X)=\NS(X_{\bk})^\Gamma$ has rank at least $2$. Hence $\mathrm{K}_X^2=7$ cannot occur.
\item If $\mathrm{K}_X^2=8$, then $X\simeq \fR_{\K/\k}(\P_\K^1)$ for some quadratic field extension $\k \hookrightarrow \K$ (see \cite[Lemma 3.2]{SZ20}), and so $\Autz(X) \simeq \fR_{\K/\k}(\PGL_{2,\K})$ acts on $X$ with no fixed closed point. Consequently, there are no $\Autz(X)$-equivariant Sarkisov links starting from $X$. This yields Case \ref{item: iii for prop1} of Proposition \ref{prop:classification in dimension 2}.
\item If $\mathrm{K}_X^2=9$, then $X\simeq\p^2_\k$ by Ch\^atelet's theorem (Proposition \ref{prop:chatelet}), and so there are no $\Autz(X)$-equivariant Sarkisov links starting from $X$ since $\Autz(X) \simeq \PGL_{3,\k}$ acts on $X$ with no fixed closed point. This yields Case \ref{item: i for prop1} of Proposition \ref{prop:classification in dimension 2}.
\end{itemize}

\smallskip

\textbf{Case 2.} Suppose now that $B \simeq \p^1_\k$. Then $X\to B$ is a Mori conic fibration. According to \cite[Lemma 6.13]{Sch19}, we have the following possibilities for $X$.
\begin{itemize}
\item $X$ is a Hirzebruch surface $\F_n$ for some $n \in \N_{\geq 0}$. 
If $n=0$, then $\Autz(\F_0)$ acts transitively on $\F_0$. 
For $n\geq1$, we have $\Autz(\F_n)=\Aut(\F_n)$, and $\Aut(\F_n)$ acts on $\F_n$ with exactly two orbits: the $(-n)$-section and its complement. In particular, there are no $\Autz(\F_n)$-equivariant birational maps to another Mori fiber space, except when $n=1$, in which case the contraction of the $(-1)$-curve conjugates $\Aut(\F_1)$ to a strict algebraic subgroup of $\Autz(\P_\k^2)=\Aut(\P_\k^2)$.
This yields Case \ref{item: ii for prop1} of Proposition \ref{prop:classification in dimension 2}.
\item $X$ is a del Pezzo surface of degree $6$ obtained by blowing-up a point of degree $2$ on $X'=\fR_{\K/\k}(\P_\K^1)$, where $\k \hookrightarrow \K$ is some quadratic field extension. But then $\Autz(X)$ is conjugate to the strict subgroup of $\Autz(X')$ preserving the point of degree $2$ of $X'$ that was blown-up.
\item $X$ is a del Pezzo surface of degree $5$ obtained by blowing-up a point of degree $4$ on $\p^2_\k$. 
But this case is excluded for us since then $\Autz(X)$ is trivial.
\end{itemize}
This concludes the proof of Proposition \ref{prop:classification in dimension 2}.
\end{proof}

\begin{proof}[Proof of Corollary \ref{cor:classification in dimension 2 for the real numbers}]
As the only quadratic field extension of $\R$ is $\R \hookrightarrow \C$ (up to complex conjugation), and as there are no real del Pezzo surfaces of degree $6$ with Picard rank $1$ (see \cite[Lemma 3.1]{RZ18}), the result follows readily from Proposition \ref{prop:classification in dimension 2}.
\end{proof}

\smallskip

We now focus on algebraic tori of $\Bir(\P_\R^2)$.
According to Lemma \ref{lem: real tori}, a torus of $\Bir(\P_\R^2)$ is isomorphic to $\G_{m,\R}$, $\SSS$, $\G_{m,\R}^2$, $(\SSS)^2$, $\SSS\times\G_{m,\R}$ or $\fR_{\C/\R}(\G_{m,\C})$. For each type, there is a single conjugacy class (see Section  \ref{sec: algebraic tori of Bir(Pn)}) in $\Bir(\P_\R^2)$, except for $\SSS$, for which the following lemma gives the list of conjugacy classes.

\begin{proposition}\label{prop: conjugacy classes of S1 in Bir(P2)}
The unit circle $\SSS$ has the following two distinct conjugacy classes in $\Bir(\p_\R^2)$:
\begin{enumerate}
\item\label{S in dim 2:1} the conjugacy class of the natural embedding of $\SSS$ into $\PGL_{3,\R}=\Aut(\P_\R^2)$ fixing the line $(z=0)$ -- here we write $[x:y:z]$ for the coordinates on $\p_\R^2$ -- and the point $[0:0:1]$; and
\item\label{S in dim 2:2} the conjugacy class obtained by embedding $\SSS$ into $\PGL_{2,\R}\times\PGL_{2,\R}=\Autz(\p_\R^1\times\p_\R^1)$ diagonally.
\end{enumerate}
\end{proposition}

\begin{proof}
Let $G$ be a connected algebraic subgroup of $\Bir(\p_\R^2)$ isomorphic to $\SSS$.
By Corollary \ref{cor:classification in dimension 2 for the real numbers}, we have the following possibilities for $G$.
\begin{itemize}
\item $X=\F_n$, with $n\geq 2$, and $G$ is conjugate to an algebraic subgroup of $\Aut(X)\simeq V_n\rtimes \GL_{2,\R}/\mu_n$\footnote{This is an old classical result, but we did not find an original reference. It is reproved in \cite[Remark 2.4.4]{BFT23} with details.}, where $V_n$ is a unipotent group of dimension $n+1$ and $\mu_n=\{ \lambda I_2, \ \lambda^n=1\} \subseteq \GL_{2,\R}$.
Conjugating $G$ by an element of $\Aut(X)$ if necessary, we can furthermore assume that $G$ is contained in $\GL_{2,\R}/\mu_n$.
In particular, $G$ preserves the two sections $s_{-n}\colon [x_0:x_1] \mapsto [0:1\ ;\ x_0:x_1]$ and $s_n \colon [x_0:x_1] \mapsto [1:0\ ;\ x_0:x_1]$ of the structure morphism $\F_n \to \P_\R^1$ (see Section~\ref{subsubsec: decomposable P1 bundles} for the notation). 
On the other hand, $G$ is not contained in $\Autz(X)_{\P_{\R}^{1}} \simeq V_n\rtimes \G_{m,\R}$, and so $G$ acts naturally on $\P_{\R}^{1}$ and  preserves therefore two $\Gamma$-conjugate fibers of the structure morphism $\F_n \to \P_\R^1$.
We pick two $\Gamma$-conjugate points on these fibers that are contained in the the section $s_{n}$. The elementary transformation $\F_n\rat\F_{n-2}$ not defined at these two points is $G$-equivariant. 
Continuing like this, we find a $G$-equivariant birational map $\F_n\rat\F_m$ over $\p_\R^1$ with $m\in\{0,1\}$. If $m=0$, then $G$ is conjugate to an algebraic subgroup of $\Autz(\P_\R^1 \times \P_\R^1)$, and if $m=1$, then $G$ is conjugate to an algebraic subgroup of $\Aut(\P_\R^2)$.

\item $X=\F_0$ and $G$ is conjugate to an algebraic subgroup of $\Autz(X)=\PGL_{2,\R} \times \PGL_{2,\R}$. Then $G$ preserves the two fibrations $p_1,p_2\colon \F_0 \to \P_\R^1$ (by Proposition \ref{prop: MMP is G-eq}) and acts nontrivially on both or one of them.
In particular, $G$ has no fixed real point, and so any equivariant Sarkisov link starting from $\F_0$ goes to $\F_2$.
Finally, if $G$ acts trivially on one of the two fibrations, then $G$ is contained in $\PGL_{2,\R} \times \{1\}$ or $\{1\} \times \PGL_{2,\R}$, and otherwise $G$ is conjugate to an algebraic subgroup of $\{(g,g)\ |\   g \in \PGL_{2,\R}\} \simeq \PGL_{2,\R}$. 
However, denoting by $\tau\colon \P_\R^1 \times \P_\R^1 \to \P_\R^1 \times \P_\R^1, \ (x,y) \mapsto (y,x)$, we see that $\tau \circ (g,1) \circ \tau=(1,g)$, hence the actions of $\SSS$ from the left and from the right are conjugate in $\Bir(\P_\R^2)$. 
Finally, consider the birational map
\[
\theta\colon\p^1_\R\times\p^1_\R\rat \p^1_\R\times\p^1_\R,\quad  ([x_1:x_2],[y_1:y_2])\dashmapsto ([x_1y_1+x_2y_2: x_1y_2-x_2y_1],[y_1:y_2]).
\] 
Then $\theta\circ(g,g)=(id, g)\circ \theta$ for any $g\in\SSS$.
This yields Case \ref{S in dim 2:2}.

\item $X=\fR_{\C/\R}(\p^1_\C)$ and $G$ is conjugate to an algebraic subgroup of $\Autz(X)=\fR_{\C/\R}(\PGL_{2,\C})$. 
The maximal tori of $\fR_{\C/\R}(\PGL_{2,\C})$ are all conjugate, and each of them is isomorphic to $\fR_{\C/\R}(\G_{m,\C})$. Since $\fR_{\C/\R}(\G_{m,\C})$ contains a unique subgroup isomorphic to $\SSS$, we deduce that $\Autz(X)$ contains a unique conjugacy class of tori isomorphic to $\SSS$. 
The surface $X$ has three $\SSS$-fixed points: one point whose residue field is $\C$ (corresponding to two $\Gamma$-conjugate points in $X_\C$) and two real points.
Blowing-up one of the two $\SSS$-fixed real points and then contracting the two disjoint $(-1)$-curves swapped by the $\Gamma$-action on $\p^1_\C \times \p^1_\C$ yields an $\SSS$-equivariant Sarkisov link to $\p_\R^2$.

\item $X=\P_\R^2$ and $G$ is conjugate to an algebraic subgroup of $\Aut(X)=\PGL_{3,\R}$. Let $g\in G\setminus\{1\}$; its characteristic polynomial in $\p(\R[t])$ is of degree $3$ and thus has a real root of multiplicity one (as $g \neq 1$). 
So, $g$ has a fixed real point $p\in\p^2_\R$, say $p=[0:0:1]$. Moreover, since $G$ is commutative and connected, it fixes $p$. 
Blowing up $p$, the group $G$ acts on $\F_1$ as a subgroup  of $\Aut(\F_1)=V_1\rtimes \GL_{2,\R}$ conjugate to a subgroup of $\GL_{2,\R}$. Conjugating $G$ by an element of $\Aut(X,p)$ if necessary, we can assume that $G$ is  a subgroup of $\GL_{2,\R}$. In particular, $G$ preserves the section $s_1$ (in the notation from above) of the structure morphism $\F_1 \to \P_{\R}^{1}$, and hence the line $\ell=[*:*:0]$ in $\p^2_\R$. The group $G$ is therefore conjugate to the image of the natural embedding of $\SSS$ into $\PGL_{3,\R}=\Aut(\P_\R^2)$ preserving the line $\ell$ and the point $p$.
Lastly, the only $G$-equivariant Sarkisov links starting from $\p^2_\R$ are the blow-up $\F_1\to\p^2_\R$ of the real point $p$ and the blow-up of the unique $G$-fixed point on $\ell$ with residue field $\C$ followed by the contraction of the strict transform of $\ell$, which gives an equivariant Sarkisov link to $\fR_{\C/\R}(\p^1_\C)$. This yields Case \ref{S in dim 2:1} of the proposition.
\end{itemize}
This concludes the proof of the proposition.
\end{proof}

\section{Mori fiber spaces of Theorem \ref{th: list of MFS corresponding to max alg subgroups of Cr3} and their automorphism groups}\label{sec: Notation for the MFS}

Until the end of this article, we work over an arbitrary base field $\k$ of characteristic zero, unless explicitly mentioned otherwise. We define the Mori fiber spaces enumerated in Theorem \ref{th: list of MFS corresponding to max alg subgroups of Cr3}  and determine their full automorphism groups (not just the identity component). This will be used in the next sections to compute their rational $\k$-forms through Galois cohomology.

\subsection{Family (a)} \label{subsubsec: decomposable P1 bundles}
Let $a,b,c\in\Z$.
The $a$-th Hirzebruch surface $\F_a$ can be defined as the quotient of $(\A^2\setminus \{0\})^2$ by the action of $\G_{m}^{2}$ given by
\[\begin{array}{ccc}
\G_{m}^{2} \times (\A^2\setminus \{0\})^2 & \to & (\A^2\setminus \{0\})^2\\
((\mu,\rho), (y_0,y_1,z_0,z_1))&\mapsto& (\mu\rho^{-a} y_0,\mu y_1,\rho z_0,\rho z_1)\end{array}\]
The class of $(y_0,y_1,z_0,z_1)$ will be written $[y_0:y_1;z_0:z_1]$. The projection 
\[\F_a\to\p^1, \ \ [y_0:y_1;z_0:z_1]\mapsto [z_0:z_1]\]
identifies $\F_a$ with $\P(\O_{\P^1}(a) \oplus \O_{\P^1})$ as a $\P^1$-bundle over $\P^1$.
 Also, we have $\F_a \simeq \F_{-a}$ via $[y_0:y_1;z_0:z_1] \mapsto [y_1:y_0;z_0:z_1]$, and most of the time so we will choose $a \geq 0$. 
The disjoint sections $\s{-a},\s{a}\subset \F_a$ given by $y_0=0$ and $y_1=0$ have self-intersection $-a$ and $a$ respectively. The fibers $f\subset \F_a$ given by $z_0=0$ and $z_1=0$ are linearly equivalent and of self-intersection $0$. We moreover have $\Pic(\F_a)=\Z f\bigoplus \Z \s{-a}= \Z f\bigoplus \Z \s{a}$, since $\s{a}\sim\s{-a}+af$.

We now define $\FF_a^{b,c}$ to be the quotient of $(\A^2\setminus \{0\})^3$ by the action of $\G_{m}^{3}$ given by
\[\begin{array}{ccc}
\G_{m}^{3}\times (\A^2\setminus \{0\})^3 & \to & (\A^2\setminus \{0\})^3\\
((\lambda,\mu,\rho), (x_0,x_1,y_0,y_1,z_0,z_1))&\mapsto& 
(\lambda\mu^{-b} x_0, \lambda\rho^{-c} x_1,\mu\rho^{-a} y_0,\mu y_1,\rho z_0,\rho z_1)\end{array}\]
The class of $(x_0,x_1,y_0,y_1,z_0,z_1)$ will be written $[x_0:x_1;y_0:y_1;z_0:z_1]$. The projection 
\[q\colon \FF_a^{b,c}\to \F_a, \ \ [x_0:x_1;y_0:y_1;z_0:z_1]\mapsto [y_0:y_1;z_0:z_1]\]
identifies $\FF_{a}^{b,c}$ with $\P(\OFa(b \s{a}) \oplus \OFa(c f))=\P(\OFa \oplus \OFa(-b\s{a}+cf))$
as a $\P^1$-bundle over $\F_a$.

Like for Hirzebruch surfaces, we have $\FF_{-a}^{b,c}\simeq\FF_a^{b,c}$. Moreover, the exchange of $x_0$ and $x_1$ yields an isomorphism $\FF_a^{b,c}\simeq\FF_a^{-b,-c}$. Hence, we may assume that $a,b\geq0$. Note that $\FF_a^{0,0}\simeq\P^1\times\F_a$. 
Let us mention that the algebraic group $\Autz(\FF_{a}^{b,c})$ is described in \cite[Section 3.1]{BFT23}; in particular, $\Autz(\FF_a^{b,c})$ maps onto $\Autz(\F_a)$.
We recall the following result as it will be useful for us to prove Propositions \ref{prop: aut group of Fabc}, \ref{prop: k-forms of Fabc with a>1}, and \ref{prop: k-forms of Fabc with a=0}.

\begin{lemma}
\label{lem: XabPic}
Let $X=\FF_{a}^{b,c}$ with $a,b \geq 0$ and $c \in \Z$.
\begin{enumerate}
\item\label{XabPicN1}
The group $\NS_\Q(X)$ is generated by $H_{x_0}:=(x_0=0) \simeq \F_a$,  $H_{y_0}:=(y_0=0) \simeq \F_{|c|}$, and $H_{z_0}:=(z_0=0) \simeq \F_b$. 
\item\label{XabPicN2}
The group $N_1^\Q(X) $ is generated by the curves $\ell_1:=H_{y_0} \cap H_{x_0}$, $\ell_2:=H_{z_0}\cap H_{x_0}$, and $\ell_3:=H_{z_0}\cap H_{y_0}$.
\item\label{XabPicN3}
The intersection form on $X$ is given by
\[\begin{array}{|c|ccc|}
\hline
& H_{z_0} & H_{y_0} & H_{x_0} \\
\hline
H_{z_0} & 0 & \ell_3 & \ell_2 \\
H_{y_0} & \ell_3 & -a\ell_3& \ell_1\\
H_{x_0} & \ell_2 & \ell_1& -b\ell_1+(c-ab)\ell_2\\
\hline\end{array}\ \ 
\begin{array}{|c|ccc|}
\hline
& H_{z_0} & H_{y_0} & H_{x_0}\\
\hline
\ell_1 & 1 & -a& c\\
\ell_2 & 0 & 1 & -b \\
\ell_3 & 0 & 0 & 1\\
\hline\end{array}\]
\item\label{XabPicN4}
Let $\ell_4=(x_1=y_0=0)$ that satisfies $\ell_4=\ell_1-c\ell_3$ in $N_1^\Q(X) $. 
The cone of effective curves $\NE(X)$ is generated by $\ell_1$, $\ell_2$, and $\ell_3$ if $c \leq 0$ and by $\ell_4$, $\ell_2$, and $\ell_3$ if $c>0$.
\item\label{XabPicN5}
The canonical divisor of $X$ is $\mathrm{K}_X=-(a(b+1)+2-c)H_{z_0}-(b+2)H_{y_0}-2 H_{x_0}$, thus $K_{X} \cdot \ell_1=a-c-2$,  $\mathrm{K}_X \cdot \ell_2=b-2$, $\mathrm{K}_X\cdot\ell_3=-2$, and $\mathrm{K}_X \cdot \ell_4=a+c-2$.
\item\label{XabPicN6} 
The structure morphism $q\colon\FF_a^{b,c}\to\F_a$ corresponds to the contraction of $\ell_3$.
\end{enumerate}
\end{lemma}

\begin{proof}
The proof is the same as in the case over algebraically closed fields; see \cite[Lemma 6.5.1]{BFT22} for details.
\end{proof}

\begin{lemma}\label{lem: number of connected components}
Consider a short exact sequence of algebraic groups $1 \to G_1 \to G_2 \to G_3 \to 1$ over an arbitrary base field.
If $G_1$ is connected, then $G_2$ and $G_3$ have the same number of connected components.
\end{lemma}

\begin{proof}
Since $G_2 \to G_3$ is onto, the number of connected components of $G_2$ is at least equal to the number of connected components of $G_3$.
Now assume that there exist two connected components $C$ and $C'$ of $G_2$ that are mapped onto the same connected component $H$ of $G_3$.
Multiplying $C$ and $C'$ by the inverse of an element of $C$ if necessary, we can assume that $C=G_2^\circ$ and $H=G_{3}^{\circ}$. But since $G_1$ is assumed to be connected, it is contained in $G_2^\circ=C$, which contradicts the fact that an element of $C'$ is mapped to $1_{G_3}$. Hence, only one connected component of $G_2$ is mapped onto a given connected component of $G_3$, which finishes the proof.
\end{proof}

\begin{proposition}\label{prop: aut group of Fabc}
Let $X=\FF_{a}^{b,c}$ with $a,b \geq 0$ and $c \in \Z$.
\begin{enumerate}
\item If $a=b=c=0$, then $X \simeq (\P^1)^3$ and $\Aut(X) \simeq (\PGL_2)^3 \rtimes S_3$, where $S_3$ is the symmetric group of order $3$ that acts naturally on the three copies of $\P^1$.
\item If $a=0$ and $b=|c| \neq 0$ or if $b=0$ and $a=|c| \neq 0$, then $\Aut(X) \simeq \Autz(X)  \rtimes \Z/2\Z$, where the generator of $\Z/2\Z$ can be taken as the involution $\sigma$ of $X$ defined by
\begin{itemize}
\item  $[x_0:x_1;y_0:y_1;z_0:z_1] \mapsto [x_1:x_0; z_0:z_1; y_0: y_1]$ if $a=0$ and $b=c$
\item  $[x_0:x_1;y_0:y_1;z_0:z_1] \mapsto [x_0:x_1; z_0:z_1; y_0: y_1]$ if $a=0$ and $b=-c$
\item  $[x_0:x_1;y_0:y_1;z_0:z_1] \mapsto [y_0:y_1; x_1:x_0; z_0: z_1]$ if $b=0$ and $a=c$
\item  $[x_0:x_1;y_0:y_1;z_0:z_1] \mapsto [y_0:y_1; x_0:x_1; z_0: z_1]$ if $b=0$ and $a=-c$
\end{itemize}
\smallskip
\item In all other cases, $\Aut(X)$ is connected.
\end{enumerate}
\end{proposition}

\begin{proof}
Assume first that $a=c=0$. If $b=0$, then $X \simeq (\P^1)^3$ and $\Aut(X) \simeq (\PGL_2)^3 \rtimes S_3$ (this can be checked for instance by looking at the induced action of $\Aut(X)$ on $\NS_\Q(X)$ and $N_1^\Q(X) $).
If $b \neq 0$, then $X \simeq \F_b \times \P^1$ and $\Aut(X) \simeq \Aut(\F_b) \times \PGL_2$ is connected.  

Assume now that $b=0$. If $a= 0$ and $c \neq 0$, then $X \simeq \F_{c} \times \P^1$ and $\Aut(X) \simeq \Aut(\F_{c}) \times \PGL_2$ is connected. If $c= 0$ and $a \neq 0$, then $X \simeq \F_a \times \P^1$ and $\Aut(X) \simeq \Aut(\F_a) \times \PGL_2$ is connected. 
If $a=|c| \neq 0$, then Lemma \ref{lem: XabPic}\ref{XabPicN4}\&\ref{XabPicN5} implies that the involution $\sigma$ defined in the statement swaps the two extremal rays generated by $\ell_2$ and $\ell_3$ in $\NE(X)$.
Therefore, for every automorphism $\varphi \in \Aut(X)$, either $\varphi$ or $\varphi \circ \sigma$ stabilizes the two extremal rays generated by $\ell_2$ and $\ell_3$. On the other hand, using the global description of $X=\FF_{a}^{b,c}$ above and the description of $\Autz(X)$ given in \cite[Section 3.1]{BFT23}, one can check that an automorphism $\psi \in \Aut(X)$ that preserves the two fibrations $q_2\colon X \to \F_{c}$ and $q_3\colon X \to \F_a$ (corresponding to the contraction of the extremal rays generated by $\ell_2$ and $\ell_3$ respectively) belongs to $\Autz(X)$. We deduce that either $\varphi$ or $\varphi \circ \sigma$ belongs to $\Autz(X)$, and the result follows.

From now on, we assume that $b \neq 0$ and ($a \neq 0$ or $c \neq 0$), in which case the structure morphism $q\colon X \to \F_a$ is $\Aut(X)$-equivariant; indeed, it corresponds to the contraction of the extremal ray generated by $\ell_3$, and the class of $\ell_3$ is  $\Aut(X)$-invariant by Lemma \ref{lem: XabPic}\ref{XabPicN4}\&\ref{XabPicN5}. Hence, we have an exact sequence of algebraic groups
\[
1 \to \Aut(X)_{\F_a} \to \Aut(X) \to H \to 1,
\]
where $\Aut(X)_{\F_a}$ is connected by \cite[Remark 3.1.6]{BFT23} and $\Autz(\F_a) \subseteq H \subseteq \Aut(\F_a)$ by \cite[Lemma 3.1.5]{BFT23}. If $a \geq 1$, then $H =\Aut(\F_a)$ is connected, and so $\Aut(X)$ is also connected by Lemma \ref{lem: number of connected components}.
We now assume that $a=0$ (and $bc \neq 0$), in which case $\Aut(\F_a) \simeq (\PGL_2)^2 \rtimes \Z/2\Z$. 
If $|b|=|c|$, then the involution $\sigma$ defined above maps to the generator of $\Z/2\Z$ (given by $[y_0:y_1;z_0:z_1] \mapsto [z_0:z_1;y_0:y_1]$).
Hence $H=\Aut(\F_a)$, which implies that $\Aut(X)$ has two connected components (by Lemma \ref{lem: number of connected components}), and so $\Aut(X) \simeq \Autz(X)  \rtimes \Z/2\Z$. 
If $|b| \neq |c|$, then the fibers of $\mathrm{pr}_i\circ q \colon \FF_{0}^{b,c}\to \P^1$ are the non-isomorphic Hirzebruch surfaces $\F_b$ and $\F_{c}$ respectively, where $\mathrm{pr}_i\colon \P^1\times \P^1\to \P^1$ is the $i$-th projection (for $i=1,2$). 
Thus, every automorphism $ \psi \in \Aut(X)$ maps to $\Autz(\F_a)$, i.e.~$H=\Autz(\F_a)$ and $\Aut(X)$ is connected (by Lemma \ref{lem: number of connected components}).
\end{proof}

\subsection{Family (b)} \label{sec 4.2.2}
Let $b\in\Z$.
We denote by $\PP_b$ the quotient of $(\A^2\setminus \{0\})\times (\A^3\setminus \{0\})$ by the action of $\G_{m}^{2}$ defined by
\[\begin{array}{ccc}
\G_{m}^{2}\times (\A^2\setminus \{0\})\times (\A^3\setminus \{0\}) & \to & (\A^2\setminus \{0\})\times (\A^3\setminus \{0\})\\
((\mu,\rho), (y_0,y_1;z_0,z_1,z_2))&\mapsto& (\mu\rho^{-b} y_0,\mu y_1;\rho z_0,\rho z_1,\rho z_2)\end{array},\]
and we write $[y_0:y_1\ ;\ z_0:z_1:z_2]$ for the class of $(y_0,y_1,z_0,z_1,z_2)$. 
The projection 
\[q\colon \PP_b\to \P^2, \ \ [y_0:y_1\ ;\ z_0:z_1:z_2]\mapsto [z_0:z_1:z_2]\]
identifies $\PP_b$ with $\P(\O_{\P^2}(b) \oplus \O_{\P^2})$ as a $\P^1$-bundle over $\P^2$.
In particular, $\PP_{-b}\simeq\PP_b$ and so we may assume that $b\geq0$. 
Let us mention that $\Autz(\PP_b)$ is described in \cite[Section 4.1]{BFT23}; in particular, $\Autz(\PP_b)$ maps onto $\Aut(\P^2)$.

\begin{proposition}\label{prop: Aut of Pb}
Let $b \geq 0$. Then $\Aut(\PP_b)$ is connected.
\end{proposition}

\begin{proof}
If $b=0$, then $\PP_0 \simeq \P^1 \times \P^2$ and $\Aut(\PP_0) \simeq \PGL_2 \times \PGL_3$ is connected. 
If $b=1$, then $\PP_1$ is isomorphic to $\P^3$ blown-up at the point $p:=[0:0:0:1]$ (see \cite[Lemma 5.3.5]{BFT22}), and so $\Aut(\PP_1) \simeq \Aut(\P^3,p)$ is connected.
Assume now that $b \geq 2$.
A rather direct computation (see e.g.~the proof of \cite[Lemma 6.5.1]{BFT22} for a similar computation) gives us that the cone of effective curves of $\PP_b$ is generated by two curves $f$ (a fiber of $q$) and $\ell$ (a line contained in the section defined by $y_0=0$), and that $K_{\PP_b}.f=-2$ and $K_{\PP_b}.\ell=b-3$. Hence, $\Aut(\PP_b)$ stabilizes the two extremal rays of $\NE(\PP_b)$ generated by $f$ and $\ell$, from which it follows that the structure morphism $q\colon \PP_b \to \P^2$ is $\Aut(\PP_b)$-equivariant.
Moreover, by \cite[Lemma 4.1.2]{BFT23}, the induced homomorphism $\Aut(\PP_b) \to \Aut(\P^2)$ is onto.
Hence, we have an exact sequence of algebraic groups
\[
1 \to \Aut(\PP_b)_{\P^2} \to \Aut(\PP_b) \to \Aut(\P^2) \to 1,
\]
where $\Aut(\PP_b)_{\P^2}$ is connected by \cite[Remark 4.1.3]{BFT23}. 
By Lemma \ref{lem: number of connected components}, the algebraic group $\Aut(\PP_b)$ is therefore connected. 
\end{proof}

\subsection{Family (c)}
Let $a,b\ge 1$ and $c\ge 2$ such that $c=ak+2$ with $0\le k\le b$. 
We call \emph{Umemura $\p^1$-bundle} the $\P^1$-bundle $\U_{a}^{b,c} \to \F_a$ 
obtained by the gluing of two copies of $\F_b\times\A^1$  along $\F_b\times(\A^1 \setminus \{0\})$ by the automorphism $\nu\in \Aut(\F_b\times(\A^1 \setminus \{0\}))$ defined by
\[\begin{array}{ccl}\nu\colon([x_0:x_1\ ;\ y_0:y_1],z) &\mapsto &\left([x_0:x_1z^{c}+x_0 y_0^ky_1^{b-k}z^{c-1}\ ;\ y_0z^a:y_1],\frac{1}{z}\right),\\
&=&\left([x_0:x_1z^{c-ab}+x_0 y_0^ky_1^{b-k}z^{c-ab-1}\ ;\ y_0:y_1z^{-a}],\frac{1}{z}\right)\end{array}.\] 
The structure morphism  $q\colon \U_{a}^{b,c} \to \F_a$ sends $([x_0:x_1\ ;\ y_0:y_1],z)\in \F_b\times\A^1$ onto respectively 
$[y_0:y_1\ ;\ 1:z]\in \F_a$ and $[y_0:y_1\ ;\ z:1]\in \F_a$ on the two charts.
Let us mention that the algebraic group $\Autz(\U_{a}^{b,c})$ is described in \cite[Section 3.6]{BFT23}; in particular, $\Autz(\U_{a}^{b,c})$ maps onto $\Aut(\F_a)$.
In fact, the Umemura $\P^1$-bundles are the only $\P^1$-bundles over Hirzebruch surfaces, besides the decomposable $\P^1$-bundles introduced in Section  \ref{subsubsec: decomposable P1 bundles}, whose identity component of the automorphism group maps onto the identity component of the automorphism group of the basis (see \cite[Proposition 3.7.4]{BFT23}).

\begin{proposition}\label{prop: Aut of Uabc}
Let $a,b\ge 1$ and $c\ge 2$ such that $c=ak+2$ with $0\le k\le b$.
Then $\Aut(\U_{a}^{b,c})$ is connected.
\end{proposition}

\begin{proof}
The proof is analogous to that of Proposition \ref{prop: Aut of Pb}.
The cone of effective curves and the intersection form on $X=\U_{a}^{b,c}$ are described in \cite[Lemma 6.4.1]{BFT22}: 
the cone $\NE(X)$ is generated by three curves $f,s,\ell$ if $c>2$ and by three curves $f,s,r$ if $c=2$, where the structure morphism $q\colon  X \to \F_a$ contracts the class of $f$. 
The intersection of these $1$-cycles with $\mathrm{K}_X$ is given by
\[
\mathrm{K}_X\cdot f=-2,\quad \mathrm{K}_X \cdot s=b-2,\quad \mathrm{K}_X\cdot \ell=a(k+1)\text{ if $c>2$},\quad \mathrm{K}_X\cdot r=a-2\text{ if $c=2$.}
\]
Consequently, $\Aut(X)$ preserves the extremal ray of $\NE(X)$ generated by $f$, and so $q$ is $\Aut(X)$-equivariant.
Since $a\geq1$,  the group $\Aut(\F_a)$ is connected and so the surjectivity of the induced homomorphism $\Aut(X) \to \Aut(\F_a)$ follows from \cite[Lemma 3.6.3]{BFT23}. 
The connectedness of $\Aut(X)_{\F_a}$ is given by \cite[Remark 3.6.4]{BFT23}.
The connectedness of $\Aut(X)$ then follows from Lemma \ref{lem: number of connected components}.
\end{proof}

\subsection{Family (d)} \label{sec: Family (d)}
Let $b\ge -1$ and let $\kappa$ be the $(2:1)$-cover defined by
\[
\begin{array}{rccc}
\kappa\colon~& \P^1\times\P^1 & \to &\P^2\\
& ([y_0:y_1],[z_0:z_1]) &\mapsto &[y_0 z_0:y_0 z_1+y_1 z_0:y_1z_1],\end{array}
\] 
whose ramification locus is the diagonal $\Delta\subseteq\P^1\times\P^1$ and whose branch locus is the smooth conic $C_0=\{ [X:Y:Z] \mid Y^2=4XZ\}\subseteq\P^2.$
The \emph{$b$-th Schwarzenberger $\P^1$-bundle} $\SS_b\to \P^2$ is the $\P^1$-bundle defined by \[q\colon \SS_b=\P(\kappa_* \O_{\P^1\times\P^1}(-b-1,0))\to \P^2.\]
Note that $\SS_b$ is the projectivization of the classical rank $2$ vector bundle $\kappa_* \O_{\P^1\times\P^1}(-b-1,0)$ introduced by Schwarzenberger in \cite{Sch61}. 
The shift by $-1$ in our notation, and therefore the fact that we take $b \geq -1$, comes from the fact that the preimage of a tangent line to $C_0$ by $\SS_b \to \P^2$ is isomorphic to $\F_b$ when $b \geq 1$ (see \cite[Lemma 4.2.5]{BFT23}).

\begin{remark}
From \cite[Proposition 7]{Sch61} (see also \cite[Corollary 4.2.2]{BFT23} and \cite[Example 2.1.8]{IP99}), 
we have 
\[\SS_{-1}\simeq\PP_1\simeq\p(\O_{\p^2}(1)\oplus\O_{\p^2}),\quad \SS_0\simeq\PP_0\simeq\p^1\times\p^2, \quad \text{and} \quad \SS_1 \simeq \P(T_{\P^2}) \simeq \PGL_3/B,
\] 
where $B$ is a Borel subgroup of $\PGL_3$. 
Recall that we have already determined $\Aut(\SS_{-1})$ and $\Aut(\SS_0)$ in Proposition~\ref{prop: Aut of Pb}.
\end{remark}

\begin{proposition}\label{prop: aut group of Sb}
Let $b \geq 1$.
\begin{enumerate}
\item If $b=1$, then $\Aut(\SS_1) \simeq \PGL_3 \rtimes \Z/2\Z$.
\item If $b \geq 2$, then $\Aut(\SS_b) \simeq \Aut(\P^2,C_0):=\{g \in \Aut(\P^2)\ |\ g(C_0)=(C_0)\} \simeq \PGL_2$. 
\end{enumerate}
\end{proposition}

\begin{proof}
If $b=1$, then $\SS_1 \simeq \PGL_3/B$ is a flag variety and \cite[Theorem 1]{Dem77} yields $\Aut(\SS_1)  \simeq \Aut_{\mathrm{gr}}(\PGL_3) \simeq \PGL_3 \rtimes \Z/2\Z$.
If $b\geq2$, then by \cite[Lemma 5.3.9]{BFT22} the effective cone $\NE(\SS_b)$ is generated by a fiber $f$ of the structure morphism $q\colon \SS_b\to\p^2$ and a certain curve $s_1$, and they satisfy $K_{\SS_b}.f=-2$ and$K_{\SS_b}.s_1=b-3$.
It follows that $q$ is $\Aut(\SS_b)$-equivariant, and so we have an exact sequence of algebraic groups
\[
1 \to \Aut(\SS_b)_{\P^2} \to \Aut(\SS_b) \to H \to 1,
\]
where $H \subseteq \Aut(\P^2)$ is the image of the homomorphism $\Aut(\SS_b) \to \Aut(\P^2)$ induced by the $\Aut(\SS_b)$-action on $\P^2$.
By \cite[Lemma 4.2.5]{BFT23}, the group $\Aut(\SS_b)_{\P^2}$ is trivial, $\Autz(\SS_b) \simeq \Aut(\P^2,C_0) \simeq \PGL_2$ (since $b\geq2$), and $H$ must preserve $C_0$.
Therefore, we have $\Aut(\SS_b) \simeq  \Aut(\P^2,C_0) \simeq \PGL_2$. 
\end{proof}

\subsection{Family (e)}
Let $b \geq 1$. 
We denote by $q\colon\V_b \to \P^2$ the $\P^1$-bundle obtained from $\U_{1}^{b,2} \to \F_1$ by contracting the $(-1)$-section of $\F_1$. The existence of the $\P^1$-bundle $\V_b \to \P^2$ follows from the descent Lemma obtained in \cite[Lemma~2.3.2]{BFT23}; see \cite[Lemma 5.5.1]{BFT23} for details. 
Let us note that $\V_1 \simeq \SS_1$ (see the proof of \textit{loc.~cit.}), and so we may assume $b \geq 2$ since we already determined $\Aut(\SS_1)$ in Proposition \ref{prop: aut group of Sb}.

\begin{corollary}[of Proposition \ref{prop: Aut of Uabc}]\label{cor:Aut-Vb}
Let $b \geq 2$ and let $\psi\colon \U_{1}^{b,2} \to \V_b$ be the blow-up of the smooth rational curve $q^{-1}([0:1:0])$.
Then $\Aut(\V_b) = \psi\Aut(\U_{1}^{b,2})\psi^{-1}$ is connected.
\end{corollary}

\begin{proof}
Let $\eta\colon \F_1\to\p^2$ be the blow-up of $p_0:=[0:1:0]$. By \cite[Lemma 5.5.1]{BFT23}, there is a commutative diagram
\[\xymatrix@R=4mm@C=2cm{
    \U_{1}^{b,2} \ar[r]^{\psi} \ar[d]  & \V_b \ar[d]\\
    \F_1 \ar[r]^{\eta} & \P^2
  }\]  
where the vertical maps are the structure morphisms. By \cite[Lemma 6.4.2]{BFT22}, the structure morphism $\V_b \to \P^2$ is $\Aut(\V_b)$-equivariant. Moreover, by \cite[Lemma 5.5.1(2)]{BFT23}, $\Aut(\V_b)$ fixes the point $p_0$ in $\P^2$, and so $\Aut(\V_b)$ stabilizes the fiber $q^{-1}(p_0)$ and $\psi$ is therefore $\Aut(\V_b)$-equivariant, i.e.~$\psi^{-1} \Aut(\V_b) \psi \subseteq \Aut(\U_{1}^{b,2})$. Thus
\[
\psi^{-1} \Aut(\V_b) \psi \subseteq \Aut(\U_{1}^{b,2}) = \Autz(\U_{1}^{b,2}) =\psi \Autz(\V_b) \psi^{-1},
\]
where the connectedness of $\Aut(\U_{1}^{b,2})$ is Proposition \ref{prop: Aut of Uabc} and the last equality is \cite[Lemma 5.5.1(4)]{BFT23}.
Hence, $\Aut(\V_b) = \psi\Aut(\U_{1}^{b,2})\psi^{-1}$ is connected.
\end{proof}

\subsection{Family (f)}
Let $b \geq 2$. The variety $\W_b$ is the toric threefold defined as the quotient of $(\A^2\setminus \{0\})\times (\A^3\setminus \{0\})$ by the action of $\G_{m}^{2}$ given by
\[\begin{array}{ccc}
\G_{m}^{2}\times (\A^2\setminus \{0\})\times (\A^3\setminus \{0\}) & \to & (\A^2\setminus \{0\})\times (\A^3\setminus \{0\})\\
((\mu,\rho), (y_0,y_1\ ;\ z_0,z_1,z_2))&\mapsto& (\mu\rho^{1-2b} y_0,\mu y_1\ ;\ \rho z_0,\rho z_1,\rho^2 z_2)\end{array},\]
and we write $[y_0:y_1;z_0:z_1:z_2]$ for the class of $(y_0,y_1,z_0,z_1,z_2)$. 
The projection 
\[q\colon \W_b\to \p(1,1,2), \ \ [y_0:y_1\ ;\ z_0:z_1:z_2]\mapsto [z_0:z_1:z_2]\]
yields a Mori $\P^1$-fibration over $\P(1,1,2)$ which is a $\P^1$-bundle over $\P(1,1,2) \setminus [0:0:1]$. 
Moreover, $\W_b$ has exactly two singular points, namely $[1:0\ ;\ 0:0:1]$ and $[0:1\ ;\ 0:0:1]$, both located in the fiber over $[0:0:1]$. 
Let us recall that $\Aut(\P(1,1,2)) \simeq \Aut(\F_2)$ is connected.

\begin{remark} \label{rk: description auto of P1-fibration of Wb}
Let  $b \geq 2$. Then $\Aut(\W_b)_{\P(1,1,2)}$ is the connected algebraic group
 \[
\left \{ \begin{bmatrix} 1 &  0 \\ 
            p & \lambda \end{bmatrix} \in \GL_2(\k[z_0,z_1,z_2])\ \middle| \ \lambda \in \k^* \text{ and } p \in \k[z_0,z_1,z_2]_{2b-1} \right \},
            \]
and the action of $\Aut(\W_b)_{\P(1,1,2)}$ on $\W_b$ is as follows:
\[ 
(\lambda,p) \cdot [y_0:y_1; z_0:z_1:z_2]=[y_0: \lambda y_1+y_0 p(z_0,z_1,z_2); z_0:z_1:z_2].
\]
Moreover, the natural homomorphism $\Autz(\W_b) \to \Aut(\P(1,1,2))$ is onto. 
All these facts can be seen  directly from the global description of $\W_b$ above.
\end{remark}

\begin{lemma} \label{lem:WbPicN}
Let $b \geq 2$.
\begin{enumerate}
\item\label{WbPicN1}
The group of numerical equivalence classes of $1$-cocycles $\NS_\Q(\W_b)$ is generated by $F=(z_2=0)$ and $S=(y_0=0)$.
\item\label{WbPicN2}
 The group of numerical equivalence classes of $1$-cycles $N_1^\Q(\W_b)$ is generated by $f=(z_1=z_2=0)$ and $\ell=F \cap S=(y_0=z_2=0)$.  
\item\label{WbPicN3}
The intersection form on $\W_b$ satisfies
\[ \begin{array}{|c|cc|}
\hline
& F & S\\
\hline
\ell & 1 & \frac{1}{2}(1-2b) \\
f&  0 &  1\\
\hline\end{array} \]
\item\label{WbPicN4}
The cone of effective curves $\NE(\W_b)$ is generated by $f$ and $\ell$.
\item\label{WbPicN5}
The canonical divisor of $\W_b$ is $K_{\W_b}=-2S-\frac{1}{2}(2b+3)F$, thus
$ K_{\W_b}\cdot f=-2$ and $K_{\W_b} \cdot \ell =b-\frac{5}{2}.$
\end{enumerate}
\end{lemma}

\begin{proof}
The proof of the statements \ref{WbPicN1}-\ref{WbPicN4} is analogous to that of \cite[Lemma 6.3.1]{BFT22} using the global description of $\W_b$ given above. For \ref{WbPicN5}, we cannot apply the adjunction formula (because $\W_b$ is singular), but since $\W_b$ is toric, we know that an anticanonical divisor of $\W_b$ is obtained by summing the five torus invariant prime divisors of $\W_b$ (see e.g. \cite[Theorem 8.2.3]{CLS11}), and the result follows. 
\end{proof}

\begin{proposition}
Let $b \geq 2$. Then $\Aut(\W_b)$ is connected.
\end{proposition}
\begin{proof}
It follows from Lemma \ref{lem:WbPicN}~\ref{WbPicN5} that the structure morphism $q\colon \W_b \to \P(1,1,2)$ is $\Aut(\W_b)$-equivariant. Moreover, the induced homomorphism $\Aut(\W_b) \to \Aut(\P(1,1,2))$ is onto and $\Aut(\W_b)_{\P(1,1,2)}$ is connected (see Remark \ref{rk: description auto of P1-fibration of Wb}).
Hence, the result follows from Lemma \ref{lem: number of connected components}.
\end{proof}

\subsection{Family (g)}\label{subsubsec: Rmn}
Let $m \geq n \geq 0$. We consider the $\P^2$-bundle over $\P^1$ defined by
\[\RR_{(m,n)}:= \P(\O_{\P^1}(-m) \oplus \O_{\P^1}(-n) \oplus \O_{\P^1});\]
it identifies with the quotient of $(\A^3\setminus \{0\})\times (\A^2\setminus \{0\})$ by the action of $\G_{m}^{2}$ given by
\[\begin{array}{ccc}
\G_{m}^{2}\times (\A^3\setminus \{0\})\times (\A^2\setminus \{0\}) & \to & (\A^3\setminus \{0\})\times (\A^2\setminus \{0\}) \\
((\lambda,\mu), (x_0,x_1,x_2,y_0,y_1))&\mapsto& 
(\lambda \mu^{-m} x_0, \lambda \mu^{-n} x_1, \lambda x_2,\mu y_0,\mu y_1)\end{array},\]
and we write $[x_0:x_1:x_2\ ;\ y_0:y_1]$ for the class of $(x_0,x_1,x_2,y_0,y_1)$. The structure morphism $q\colon \RR_{(m,n)} \to \P^1$ identifies with the projection $[x_0:x_1:x_2\ ;\ y_0:y_1]\mapsto [y_0:y_1]$.
The algebraic group $\Autz(\RR_{(m,n)})$ is described in \cite[Lemma 5.3.1]{BFT22}.

\begin{lemma}\label{lem: vertical part of Aut(Rmn)}
For each $i \in \Z$ we denote by $\k[y_0,y_1]_{i}\subseteq \k[y_0,y_1]$ the vector subspace of homogeneous  polynomials of degree $i$ (which is $\{0\}$ if $i<0$). 
Then $\Aut(\RR_{(m,n)})_{\P^2}$ is the connected algebraic group defined by
\[ \left \{ \begin{bmatrix}
 p_{1,0}& p_{2,n-m} & p_{3,-m}\\ p_{4,m-n} & p_{5,0} & p_{6,-n} \\ 
 p_{7,m} & p_{8,n} & p_{9,0}  \end{bmatrix} \in \PGL_3(\k[y_0,y_1]) \ \middle| \ \begin{array}{l}
 p_{k,i}\in \k[y_0,y_1]_{i}\\
 \mbox{for }k=1,\dots,9.\end{array} \right \}.\]
Moreover, the action of $\Aut(\RR_{(m,n)})_{\P^2}$ on $\RR_{(m,n)}$ is as follows:
\[ [x_0:x_1:x_2 \ ;\ y_0:y_1] \mapsto \left[  \begin{bmatrix}
 p_{1,0}& p_{2,n-m} & p_{3,-m}\\ p_{4,m-n} & p_{5,0} & p_{6,-n} \\ 
 p_{7,m} & p_{8,n} & p_{9,0}  \end{bmatrix} \begin{bmatrix} x_0 \\ x_1 \\ x_2 \end{bmatrix}\ ;\ y_0:y_1\right].\]
\end{lemma}

\begin{proof}
This can be seen  directly from the global description of $\RR_{(m,n)}$ above, and by using tri\-vializations on open subsets of $\P^1$ isomorphic to $\A^1$.  
\end{proof}

\begin{proposition}\label{prop:Aut-R_mn is connected}
Let $m \geq n \geq 0$. Then $\Aut(\RR_{(m,n)})$ is connected.
\end{proposition}

\begin{proof}
The proof is analogous to that of Proposition \ref{prop: Aut of Pb}.
By Lemma \cite[Lemma 6.3.1]{BFT22}, the cone of effective curves of $X=\RR_{(m,n)}$ is generated by two curves $f$ and $\ell$, and they satisfy $\mathrm{K}_X\cdot f=-3$ and $\mathrm{K}_X \cdot \ell=m+n-2\geq-2$.
It follows that the structure morphism $q\colon \RR_{(m,n)} \to \P^1$ is $\Aut(\RR_{(m,n)})$-equivariant. 
The surjectivity of the induced homomorphism $\Aut(\RR_{(m,n)}) \to \Aut(\P^1)$ is given by \cite[Lemma 5.3.1]{BFT22} and the connectedness of $\Aut(\RR_{(m,n)})_{\P^1}$ by Lemma \ref{lem: vertical part of Aut(Rmn)}. 
The connectedness of $\Aut(\RR_{(m,n)})$ then follows from Lemma \ref{lem: number of connected components}.
\end{proof}

\subsection{Family (h)}
Assume that $\k=\bk$.
Let $g\in \k[u_0,u_1]$ be a homogeneous polynomial of degree $2n$, for some $n \geq 0$. 
We denote by $\QQ_g$ the projective threefold defined by 
\[\QQ_g:=\{[x_0:x_1:x_2:x_3;u_0:u_1]\in \P(\O_{\P^1}^{\oplus 3}\oplus \O_{\P^1}(n))\mid x_0^2-x_1x_2-g(u_0,u_1)x_3^2=0\},\]
where $\P(\O_{\P^1}^{\oplus 3}\oplus \O_{\P^1}(n))$ is the quotient of $(\A^4\setminus \{0\} )\times (\A^2\setminus \{0\})$ by the action of $\G_{m}^2$ given by \[((\mu,\rho)\cdot (x_0,x_1,x_2,x_3,u_0,u_1))\mapsto (\mu x_0,\mu x_1,\mu x_2,\rho^{-n}\mu x_3,\rho u_0,\rho u_1),\]
and we denote by $\pi_g\colon\  \QQ_g\to \P^1$ the projection $ [x_0:x_1:x_2:x_3;u_0:u_1]\mapsto  [u_0:u_1].$
If $g$ is not a square, then $\pi_g\colon\  \QQ_g\to \P^1$ is a Mori quadric fibration (see \cite[\S~4.4]{BFT22} for details). In the following, we will call such a fibration an \emph{Umemura quadric fibration}. 

\begin{remark}\label{rmk:forms of Q_g if g constant}
Suppose that $n=0$. If $g=0$, then $\QQ_0\simeq \p^2\times\p^1\simeq\PP_0$ (see Section  \ref{sec 4.2.2}). If $g\neq0$, then the equation $x_0^2-x_1x_2-gx_3^2=0$ defines a smooth quadric in $\P^3$ and $\QQ_g\simeq(\p^1)^3\simeq\FF_0^{0,0}$ (see Section  \ref{subsubsec: decomposable P1 bundles}). 
Hence, from now on, we will always assume that $n \geq 1$.
\end{remark}

When $g$ is not a square and $n \geq 1$, the algebraic group $\Autz(\QQ_g)$ is described in \cite[Corollary 4.4.7]{BFT22}: it is isomorphic to $\PGL_{2} \times \G_{m}$ if $g$ has only two roots, and isomorphic to $\PGL_{2}$ if $g$ has at least three roots. 
We now compute the whole automorphism group $\Aut(\QQ_g)$.

\begin{proposition}\label{prop: aut group of Qg}
Assume that $\k=\bk$.
Let $g\in\k[u_0,u_1]$ be a homogeneous polynomial, which is not a square and is of degree  $\deg(g)=2n$ for some $n \in \N_{\geq 1}$.
Then $\Aut(\QQ_g)$ acts on $\p^1$ and $\pi_g\colon \QQ_g\to\p^1$ is $\Aut(\QQ_g)$-equivariant.
Then $\Aut(\QQ_g)$ fits into a short exact sequence
\begin{equation}\label{eq: ES for Aut(Qg)}
1 \to \Aut(\QQ_g)_{\P^1} \to \Aut(\QQ_g) \to F \to 1,
\end{equation} 
where $F$ is an algebraic subgroup of $\PGL_{2}$ and $\Aut(\QQ_g)_{\P^1} \simeq \PGL_{2} \times \Z/2\Z $ acts on $\QQ_g$ as follows:
\begin{itemize}
\item for every $\begin{bmatrix} a & b \\ c & d \end{bmatrix} \in \PGL_{2}(\k)$, we have 
\[ \begin{bmatrix} a & b \\ c & d \end{bmatrix} \cdot [x_0:x_1:x_2:x_3;u_0:u_1] = \begin{array}{l}
[(ad+bc)x_0+acx_1+bdx_2:
2abx_0+a^2x_1+b^2x_2:\\ \ 2cdx_0+c^2x_1+d^2x_2:(ad-bc)x_3;u_0:u_1]
\end{array};\]
\item the nontrivial element of $\Z/2\Z$ acts on $\QQ_{g}$ as the biregular involution  
\[[x_0:x_1:x_2:x_3;u_0:u_1]\mapsto [x_0:x_1:x_2:-x_3;u_0:u_1].\]
\end{itemize}
Moreover, the following hold:
\begin{itemize}
\item If $g$ has at least three distinct roots, then $F$ is finite. If furthermore $n$ is even, then the short exact sequence \eqref{eq: ES for Aut(Qg)} splits and $\Aut(\QQ_g) \simeq \Aut(\QQ_g)_{\P^1} \times F$.
\item If $g$ has exactly two distinct roots with the same multiplicities, then $F\simeq \mathbb{G}_m\rtimes\Z/2\Z$, where the generator of $\Z/2\Z$ exchanges the two roots of $g$; otherwise, when the roots of $g$ have different multiplicities, $F \simeq \mathbb{G}_m$.
\end{itemize}
\end{proposition}

\begin{proof}
By \cite[Lemma 4.4.3]{BFT22} and its proof, the cone of effective curves on $\QQ_g$ is generated by two curves $f$ and $h$ that satisfy 
$K_{Q_g}\cdot f=-2$ and $K_{Q_g}\cdot h=n-2\geq -1$.
Since the contraction of the extremal ray generated by $f$ yields the structure morphism $\pi_g\colon Q_g\to\p^1$, we obtain that $\Aut(\QQ_g)$ acts on $\p^1$ and $\pi_g\colon \QQ_g\to\p^1$ is $\Aut(\QQ_g)$-equivariant.
The fact that $\Aut(\QQ_g)_{\P^1} \simeq \PGL_{2} \times \Z/2\Z$ follows from \cite[Lemmas 4.4.4 and 4.4.5]{BFT22}. 

Let $F:=\textrm{Im}(\Aut(\QQ_g) \to \PGL_{2})$. The identity component $F^0$ must fix each root of $g$. 
Assume first that $g$ has at least three distinct roots. Then $F^0$ must be trivial and hence $F$ is a finite subgroup of $\PGL_{2}$.
If $n$ is furthermore even, then the natural $\SL_{2}$-action on $\P(\O_{\P^1}^{\oplus 3}\oplus \O_{\P^1}(n))$ defined by 
\[
\begin{bmatrix}
a & b \\ c & d
\end{bmatrix} \cdot [x_0:x_1:x_2:x_3; u_0:u_1]:=[x_0:x_1:x_2:x_3; a u_0+bu_1:cu_0+d u_1]
\]
induces a $\PGL_{2}$-action on $\P(\O_{\P^1}^{\oplus 3}\oplus \O_{\P^1}(n))$ that restricts to an $F$-action on $\QQ_g$. This $F$-action commutes with the $\Aut(\QQ_g)_{\P^1}$-action, and thus $\Aut(\QQ_g) \simeq \Aut(\QQ_g)_{\P^1} \times F$.

Assume now that $g$ has exactly two distinct roots. Up to a linear change of coordinates, we can assume that $g=u_0^au_1^b$ for some odd $a,b \geq 1$ (because $g$ is not a square). Then a direct computation (see \cite[Example 4.4.6]{BFT22}) yields that either $F\simeq \G_m\rtimes\Z/2\Z$ when $a=b$, where the generator of $\Z/2\Z$ exchanges the two roots of $g$, or $F\simeq \G_m$ when $a \neq b$. 
\end{proof}

\begin{remark}\label{rk: every F occurs}
Every finite subgroup of $\PGL_{2}$ can occur as $F$ in Proposition \ref{prop: aut group of Qg}.
Indeed, fix $H$ a finite subgroup of $\PGL_{2}$ and denote by $\widetilde{H}$ the inverse image of $H$ in $\SL_{2}$. It suffices then to take $g \in \k[u_0,u_1]^{\widetilde{H}}$, and $g$ not contained in the invariant algebra of a finite subgroup of $\SL_{2}$ containing $\widetilde{H}$ as a strict subgroup, to have $F=H$ in Proposition \ref{prop: aut group of Qg}. 
\end{remark}

\section{Forms of decomposable projective bundles}\label{sec: Case of decomposable projective bundles}

From now on, we will indicate with a subscript the base field when there is a possible ambiguity. As before we denote by $\Gamma=\Gal(\bk/\k)$ the absolute Galois group of $\k$.

\smallskip

In this section we determine the rational  $\k$-forms of the projective bundles \hyperlink{th:D_a}{(a)}, \hyperlink{th:D_b}{(b)}, and \hyperlink{th:D_g}{(g)} of Theorem \ref{th: list of MFS corresponding to max alg subgroups of Cr3}.

\begin{proposition}\label{prop: k-forms of PPb}
The $\bk$-variety $\PP_{b,\bk}$, with $b \geq 0$, has no nontrivial $\k$-form with a $\k$-point. 
In particular, the trivial $\k$-form $\PP_{b,\k}$ is the only rational $\k$-form of $\PP_{b,\bk}$.
\end{proposition}

\begin{proof}
Let us denote $X=\PP_{b,\bk}$.
If $b=0$, then $X \simeq \P_{\bk}^2 \times \P_{\bk}^1$ and it follows from Ch\^atelet's theorem that $\P_\k^2 \times \P_\k^1$ is the only $\k$-form of $X$ with a $\k$-point. We now assume that $b \geq 1$. 
Computing the cone of effective curves of $X$ (see the proof of Proposition \ref{prop: Aut of Pb}), we see that $\Gamma=\Gal(\bk/\k)$ fixes the two extremal ray $\left\langle f\right\rangle$ and $\left\langle \ell\right\rangle$ of $\NE(X)$; indeed, $K_X.f=-2$ and $K_X.\ell=b-3$, so the result is clear for $b \geq 2$, and for $b=1$ the contraction of $\left\langle f\right\rangle$ yields the structure morphism $X \to \P_{\bk}^1$ while the contraction of $\left\langle \ell\right\rangle$ yields a divisorial contraction to $\P_{\bk}^3$, and so these two contractions cannot be swapped by the $\Gamma$-action.
The $\Gamma$-action on $X$ induces therefore a $\Gamma$-action $\P_{\bk}^1$ such that the structure morphism $X \to \P_{\bk}^1$ is $\Gamma$-equivariant.

The group $\Aut(X)$ is connected by Proposition~\ref{prop: Aut of Pb}.
Then the homomorphism $\Aut(X)\to \Aut\left(\P_{\bk}^2\right)$  is onto by \cite[Lemma 4.1.2]{BFT23} and $\Aut(X)_{\P_{\bk}^2} \simeq U \rtimes \G_{m,\bk}$ by \cite[Remark 4.1.3]{BFT23}, where $U:=\bk[z_0,z_1,z_2]_b$ is a unipotent group. 
Hence, we have a $\Gamma$-equivariant short exact sequence
\[
1 \to U \rtimes \G_{m,\bk} \to \Aut(X) \to \Aut\left(\P_{\bk}^2\right) \to 1
\]
which induces an exact sequence in Galois cohomology
\[
\H^1\left(\Gamma,U \rtimes \G_{m,\bk}\right) \to \H^1\left(\Gamma,\Aut(X)\right) \to \H^1\left(\Gamma,\Aut\left(\P_{\bk}^2\right)\right).
\]

Let us note that, if $X \to Y$ is a  morphism between $\k$-varieties such that $X(\k) \neq \varnothing$, then $Y(\k) \neq \varnothing$. Thus, if moreover $Y_{\bk} \simeq \P_{\bk}^n$, we must have $Y \simeq \P_{\k}^n$ (by Ch\^atelet's theorem). Therefore, the $\k$-forms of $X$ with a $\k$-point correspond to elements of $\H^1\left(\Gamma,\Aut\left(X\right)\right)$ contained in the fiber of $\H^1\left(\Gamma,\Aut\left(X\right)\right) \to \H^1\left(\Gamma,\Aut\left(\P_{\bk}^2\right)\right)$ over the marked point of $\H^1\left(\Gamma,\Aut\left(\P_{\bk}^2\right)\right)$ (i.e.~over the point corresponding to the trivial $\k$-form of $\P_{\bk}^2$); this fiber coincides with the image of the map $\H^1\left(\Gamma,U \rtimes \G_{m,\bk}\right) \to \H^1\left(\Gamma,\Aut\left(X\right)\right)$. Hence, to prove the result, it suffices to show that $\H^1\left(\Gamma,U \rtimes \G_{m,\bk}\right)$ is a singleton.  

Using the description of the $\Aut(X)_{\P^2}$-action on $X$ given in \cite[Remark 4.1.3]{BFT23}, we see that $U$ is $\Gamma$-stable and that each element $ \gamma \in \Gamma$ acts on $\G_{m,\bk}(\bk)=\bk^*$ via the corresponding field automorphism of $\bk$. As the two corresponding sets $\H^1\left(\Gamma,U\right)$ and $\H^1\left(\Gamma,\G_{m,\bk}\right)$ are singletons (by \cite[Chp.~III, Section  2.1, Proposition 6]{Ser02} for $U$ and by Hilbert's Theorem 90 for $\G_{m,\bk}$), we obtain that $\H^1\left(\Gamma,U \rtimes \G_{m,\bk}\right)$ is also a singleton. Thus, the unique $\k$-form of $X$ with a $\k$-point is $\PP_{b,\k}$ (up to isomorphism).
\end{proof}

\begin{remark}\label{rk: real forms of Pb}
Assume furthermore that $\k=\R$. Then $\PP_{b,\C}$ admits a unique real form, up to isomorphism, if $b \geq 1$ and exactly two non-isomorphic real forms if $b=0$ (namely $\P_\R^2 \times \P_\R^1$ and $\P_\R^2 \times C$, where $C$ is the conic with no real points). Indeed, this is an easy consequence of the fact that $\P_\R^2$ is the unique real form of $\P_\C^2$, up to isomorphism.
\end{remark}

\begin{proposition}\label{prop: k forms of Rmn}
Let $m \geq n \geq 0$.
The $\bk$-variety $\RR_{(m,n),\bk}$ has no nontrivial $\k$-form with a $\k$-point. 
In particular,  the trivial $\k$-form $\RR_{(m,n),\k}$ is the only rational $\k$-form of $\RR_{(m,n),\bk}$.
\end{proposition}

\begin{proof}
Let us denote $X=\RR_{(m,n),\bk}$.
If $m=n=0$, then $X \simeq \P_{\bk}^2 \times \P_{\bk}^1 \simeq \PP_0$ and the result follows from Proposition \ref{prop: k-forms of PPb}.
From now on we assume that $(m,n) \neq (0,0)$. The rest of the proof is analogous to that of Proposition \ref{prop: k-forms of PPb}.  
By Lemma \cite[Lemma 6.3.1]{BFT22}, the cone of effective curves of $X$ is generated by two curves $f$ and $\ell$, and they satisfy $\mathrm{K}_X\cdot f=-3$ and $\mathrm{K}_X \cdot \ell=m+n-2\geq-2$. 
Moreover, the natural homomorphism $\Aut(X)\to\Aut(\p^1_{\bk})$  is onto by \cite[Lemma 5.3.1]{BFT23}.
Hence, we have a $\Gamma$-equivariant short exact sequence
\[
1\to \Aut(X)_{\P^1_{\bk}}\to \Aut(X)\to\Aut(\p^1_{\bk})\to 1
\]
which induces an exact sequence in Galois cohomology
\[
\H^1(\Gamma,\Aut(X)_{\P^1}) \to \H^1(\Gamma,\Aut(X)) \to \H^1(\Gamma,\Aut(\P_{\bk}^1 )).
\]
As explained in the proof of Proposition \ref{prop: k-forms of PPb}, 
it suffices then to show that $\H^1(\Gamma,\Aut(X)_{\P^1_{\bk}})$ is a singleton to prove the proposition.
Let us note that an element $\gamma \in \Gamma$ acts on $\Aut(X)_{\P^1} \subseteq \PGL_3(\bk[y_0,y_1])$ (described in Lemma \ref{lem: vertical part of Aut(Rmn)}) as the corresponding field automorphism of $\bk$ on each coefficient. We distinguish two cases:
\begin{itemize}
\item \emph{Case $m=n \geq 1$ or $m>n=0$.} In both cases $\Aut(X)_{\P^1} \simeq U \rtimes \GL_{2,\bk}$, where $U$ is a unipotent group stabilized by the $\Gamma$-action, and $\Gamma$ acts naturally on $\GL_{2,\bk}$. Since $\H^1(\Gamma,U)$ and $\H^1(\Gamma, \GL_{2,\bk})$ are singletons (by \cite[Chp.~III, Section  2.1, Proposition 6]{Ser02} for $U$ and by \cite[Chp.~X, Section 1, Proposition 3]{Ser04} for $\GL_{2,\bk}$), we obtain that $\H^1(\Gamma,\Aut(X)_{\P^1})$ is also a singleton.
\item \emph{Case $m>n>0$.} Then $\Aut(X)_{\P^1} \simeq U \rtimes \G_{m,\bk}^{2}$, where $U$ is a unipotent group stabilized by the $\Gamma$-action, and $\Gamma$ acts naturally on $\G_{m,\bk}^{2}$. But there again, $\H^1(\Gamma,U)$ and $\H^1(\Gamma, \G_{m,\bk}^{2})$ are singletons (by Hilbert's Theorem 90 for $\G_{m,\bk}^{2}$), and so $\H^1(\Gamma,\Aut(X)_{\P^1})$ is a singleton.
\end{itemize}
Thus, in both cases we find that $\H^1(\Gamma,\Aut(X)_{\P^1})$ is a singleton, which concludes the proof.
\end{proof}

\begin{remark}\label{rk: real forms of Rmn}
Assume furthermore that $\k=\R$. Then, using the fact that $\P_\C^1$ admits two non-isomorphic real forms, namely $\P_\R^1$ and the conic with no real points $C$, we can adapt the proof of Proposition \ref{prop: k forms of Rmn} and use the global description of $\RR_{(m,n)}$ recalled in Section  \ref{subsubsec: Rmn} to show that $\RR_{(m,n),\C}$ admits a unique real form (up to isomorphism) if $m$ or $n$ is odd, and two non-isomorphic real forms if $m$ and $n$ are even.  In the second case, the nontrivial real form has no real points and corresponds to the real structure (i.e. the antiregular involution) defined by 
\[
\RR_{(m,n),\C} \to \RR_{(m,n),\C},\ [x_0:x_1:x_2\ ; \ y_0:y_1] \mapsto [\overline{x_0}:\overline{x_1}:\overline{x_2}\ ; \ -\overline{y_1}:\overline{y_0}].
\]
\end{remark}

\begin{lemma}\label{lem: k-forms of Fa}
Let $a \geq 1$. 
The $\bk$-variety $\F_{a,\bk}$ has no nontrivial $\k$-form with a $\k$-point. 
In particular, the trivial $\k$-form $\F_{a,\k}$ is the only rational $\k$-form of $\F_{a,\bk}$.
\end{lemma}
\begin{proof}
The proof is analogous to that of Proposition \ref{prop: k-forms of PPb}.
\end{proof}

\begin{remark}\label{rk: real forms of Fa}
Assume furthermore that $\k=\R$. Then $\F_{a,\C}$ admits a unique real form (up to isomorphism) when $a$ is odd and exactly two non-isomorphic real forms when $a>0$ is even (see e.g.~\cite[Th\'eor\`eme 4.3]{Ben16}). 
\end{remark}

\begin{proposition}\label{prop: k-forms of Fabc with a>1}
Let $a \geq 1$ and let $b,c \in \Z$ such that $b > 0$ or $a \neq |c|$. 
The $\bk$-variety $\FF_{a,\bk}^{b,c}$ has no nontrivial $\k$-form with a $\k$-point. 
In particular,  the trivial $\k$-form $\FF_{a,\k}^{b,c}$ is the only rational $\k$-form of $\FF_{a,\bk}^{b,c}$.
\end{proposition}

\begin{proof}
Let us denote $X=\FF_{a,\bk}^{b,c}$.
Let us first assume that $b=c=0$. Then $X \simeq \P_{\bk}^1 \times \F_a$ and the result follows from Ch\^atelet's theorem combined with Lemma \ref{lem: k-forms of Fa}.

We now assume that $b > 0$ or $c \neq 0$.
The proof is then analogous to that of Proposition \ref{prop: k-forms of PPb}. Indeed, since $a \neq |c|$, Lemma \ref{lem: XabPic} implies that $\Gamma$ fixes the extremal rays of $\NE(X)$, and so 
we have a $\Gamma$-equivariant short exact sequence
\[
1 \to \Aut(X)_{\F_a} \to \Aut(X) \to \Aut\left(\F_{a,\bk}\right) \to 1
\]
which induces an exact sequence in Galois cohomology
\[
\H^1\left(\Gamma,\Aut(X)_{\F_a}\right) \to \H^1\left(\Gamma,\Aut(X)\right) \to \H^1\left(\Gamma,\Aut\left(\F_{a,\bk}\right)\right).
\]
Hence, since $\F_{a,\bk}$ has no nontrivial $\k$-form with a $\k$-point (by Lemma \ref{lem: k-forms of Fa}),
it suffices to show that $\H^1(\Gamma,\Aut(X)_{\F_a})$ is a singleton to prove the proposition. By \cite[Remark 3.1.6]{BFT23}, we have that $\Aut(X)_{\F_a} \simeq U \rtimes \G_{m,\bk}$, where $U$ is a $\Gamma$-stable unipotent subgroup and $\Gamma$ acts naturally on $\G_{m,\bk}(\bk)=\bk^*$, so $\H^1(\Gamma,\Aut(X)_{\F_a})$ is a singleton, which concludes the proof.
\end{proof}

\begin{remark}\label{rk: non rational k-form for FFabc}
Assume furthermore that $\k=\R$.  If $a$ or $c$ is odd, then the trivial real form $\FF_{a,\R}^{b,c}$ is the unique real form of $\FF_{a,\C}^{b,c}$ (up to isomorphism). If $a\,(\geq 1)$ and $c$ are even, then $\FF_{a,\C}^{b,c}$ has exactly two non-isomorphic real forms: the trivial real form $\FF_{a,\R}^{b,c}$ and a second one (with no real points) corresponding to the following real structure: 
\[
\FF_{a,\C}^{b,c} \to \FF_{a,\C}^{b,c},\ [x_0:x_1;y_0:y_1;z_0:z_1] \mapsto [\overline{x_0}:\overline{x_1};\overline{y_0}: \overline{y_1};-\overline{z_1}:\overline{z_0}].
\]
This follows from Remark \ref{rk: real forms of Fa} and the global description of the $\FF_{a}^{b,c}$ given in Section  \ref{subsubsec: decomposable P1 bundles}.
\end{remark}

\smallskip

It remains to determine the rational $\k$-forms of the $\bk$-variety $\FF_{a,\bk}^{b,c}$ when $a=0$. 
Until the end of Section  \ref{sec: Case of decomposable projective bundles}, it is assumed for simplicity that $\k=\R$.

\begin{lemma}{\emph{(see \cite[Proposition 1.2]{Rus02} for the case $n=2$)}} \label{lem: real forms of (P1)n}\\
Denote by $C$ the real conic with no real points (i.e.~the nontrivial real form of $\P_\C^1$).
Up to isomorphism, the real forms of the complex variety $(\P_\C^1)^n$, with $n \geq 1$, are the following: 
\[
Z_{p,q,r}:=(\fR_{\C/\R}(\P_\C^1))^p \times (\P_\R^1)^q \times C^r, \ \text{with}\ p,q,r \in \N_{\geq 0} \ \text{satisfying}\ 2p+q+r=n.
\]
Moreover, either $r=0$ and then $Z_{p,q,r}$ is rational or $r>0$ and then $Z_{p,q,r}$ has no real points.
The identity component of the automorphism group of $Z_{p,q,r}$ is given by
\[
\Autz(Z_{p,q,r}) \simeq (\fR_{\C/\R}(\PGL_{2,\C}))^p \times (\PGL_{2,\R})^q \times (\SO_{3,\R})^r.
\]
\end{lemma}

\begin{proof}
Since $\Aut((\P_\C^1)^n) \simeq \Aut_{\mathrm{gr}}((\SL_{2,\C})^n)$ as $\Gamma$-groups (endowed with the canonical $\Gamma$-action on both sides), there is a bijection between $\H^1(\Gamma,\Aut((\P_\C^1)^n)$ and $\H^1(\Gamma,\Aut_{\mathrm{gr}}((\SL_2)^n)$. The latter set, which parametrizes the real forms (up to isomorphism) of $(\SL_2)^n$ in the category of real algebraic groups is described in \cite[Lemma 1.7 and Example 1.8]{MJT18}; it corresponds to the set of non-isomorphic real forms $Z_{p,q,r}$ defined in the statement of the lemma. 
The claim concerning the rationality and the real points of $Z_{p,q,r}$ follows from the facts that the surface $\fR_{\C/\R}(\P_\C^1)$ is rational while $C$ has no real points. The description of $\Autz(Z_{p,q,r})$ is straightforward since $Z_{p,q,r}$ is a product of projective varieties (see \cite[Corollary 4.2.7]{BSU13}) and we have $\Autz(\fR_{\C/\R}(\P_\C^1)) \simeq \fR_{\C/\R}(\PGL_{2,\C})$ and $\Aut(C) \simeq \SO_{3,\R}$.   
\end{proof}

\begin{lemma}\label{lem: two real forms for F0bc}
Let $\Gamma=\Gal(\C/\R)$ act on $\P_\C^1 \times \P_\C^1$ via $([y_0:y_1],[z_0:z_1]) \mapsto ([\overline{z_0}:\overline{z_1}],[\overline{y_0}:\overline{y_1}])$. Then 
\[\Pic(\fR_{\C/\R}(\P_\C^1)) \simeq \Pic(\P_\C^1 \times \P_\C^1)^\Gamma \simeq \Z[\ell_1+\ell_2],\]
where $\ell_1$ resp. $\ell_2$, is the class of the curve $y_1=0$ resp. of $z_1=0$, in $\Pic(\P_\C^1 \times \P_\C^1) \simeq \Z[\ell_1,\ell_2]$.
Moreover, for every $b \in \Z$, the rank $2$ vector bundle $\O_{\P_\C^1 \times \P_\C^1}(b \ell_1) \oplus \O_{\P_\C^1 \times \P_\C^1}(b \ell_2)$ is $\Gamma$-linearized $(\Gamma$ exchanges the two factors$)$ and corresponds therefore to a rank $2$ vector bundle on $\fR_{\C/\R}(\P_\C^1)$ that we will denote by $\widetilde{\mathcal{H}_b} \to \fR_{\C/\R}(\P_\C^1)$.
\end{lemma}

\begin{proof}
The first isomorphism is \cite[\href{https://stacks.math.columbia.edu/tag/0CDT}{Tag 0CDT}]{SP24}. The second isomorphism and the second statement follow from the fact that the nontrivial element of $\Gamma$ swaps $\ell_1$ and $\ell_2$.
\end{proof}

\begin{definition}\label{def: Gb and Hb}
Let $b \geq 0$ and keep the notation of Lemma \ref{lem: two real forms for F0bc}.
We denote
\[\mathcal{G}_b:=\P\left(\O_{\fR_{\C/\R}(\P_\C^1)} \oplus \O_{\fR_{\C/\R}(\P_\C^1)}(-b (\ell_1+ \ell_2))\right)\ \ \text{and}\ \  \mathcal{H}_b:=\P\left(\widetilde{\mathcal{H}_b}\right)\] which are real forms of $\FF_{0,\C}^{b,-b}$ and $\FF_{0,\C}^{b,b}$ respectively.
Let us note that $\mathcal G_0\simeq \mathcal H_0\simeq\p^1_\R\times\fR_{\C/\R}(\p_\R^1)$.
\end{definition}

\begin{proposition}\label{prop: k-forms of Fabc with a=0}
Let $a=0$ and let $b,c \in \Z$ such that $b \geq 0$ and $(b,c) \neq (0,0)$.
\begin{itemize}
\item If $b \neq |c|$, then $\FF_{0,\C}^{b,c}$ has no nontrivial real form with a real point.
\item If $b=|c|$, then $\FF_{0,\C}^{b,c}$ has exactly two non-isomorphic real forms with a real point, namely the trivial real form $\FF_{0,\R}^{b,c}$  and the real form $\mathcal{G}_b$ if $b=-c$ resp. $\mathcal{H}_b$ if $b=c$ $($with the notation of Definition \ref{def: Gb and Hb}$)$, corresponding to the real structure
\begin{equation}\label{eq: real structures for Gb and Hb}
\theta\colon[x_0:x_1;y_0:y_1;z_0:z_1] \mapsto 
\left\{
    \begin{array}{ll}
           \left[\overline{x_0}:\overline{x_1};\overline{z_0}:\overline{z_1};\overline{y_0}:\overline{y_1}\right]  &  \text{ if } b=-c ; \text{ and} \\
         \left[\overline{x_1}: \overline{x_0};\overline{z_0}:\overline{z_1};\overline{y_0}:\overline{y_1}\right] & \text{ if } b=c.
    \end{array}
\right.
\end{equation}
\end{itemize}
Moreover, if $X_0$ is a real form of $\FF_{0,\C}^{b,c}$, then either $X_0$ is rational or $X_0$ has no real points.
\end{proposition}

\begin{proof}
Let us denote $X=\FF_{0,\C}^{b,c}$.
For any descent datum $\Gamma \times X \to X$, it follows from Lemma \ref{lem: XabPic} and Proposition \ref{prop: aut group of Fabc} that we have a $\Gamma$-equivariant short exact sequence
\[
1 \to \Aut(X)_{\F_0} \to \Aut(X) \to H \to 1,
\]
where $H=\Aut(\F_{0,\C})$ if $b=|c| \neq 0$ and $H=\Autz(\F_{0,\C})$ otherwise, which
yields an exact sequence in Galois cohomology 
\[
\H^1\left(\Gamma,\Aut(X)_{\F_0}\right) \to \H^1\left(\Gamma,\Aut(X)\right) \to \H^1(\Gamma,H).
\] 
Let us note that $\H^1(\Gamma,\Aut(\F_{0,\C}))$ parametrizes the equivalence classes of all real structures on $\F_0$ while $\H^1(\Gamma,\Autz(\F_{0,\C}))$ parametrizes the equivalence classes of real structures on $\F_{0,\C}$ that do not exchange the two rulings. 
By Lemma \ref{lem: real forms of (P1)n}, a rational real form of $X$ is therefore mapped to $\P_\R^1 \times \P_\R^1$ if $b \neq |c|$, and to $\P_\R^1 \times \P_\R^1$ or $\fR_{\C/\R}(\P_\C^1)$ if $b = |c|$.
Also, by \cite[Remark 3.1.6]{BFT23}, $\Aut(X)_{\F_0} \simeq U \rtimes \G_{m,\C}$, where $U$ is a unipotent subgroup stabilized by the $\Gamma$-action and $\Gamma$ acts on $\G_{m,\C}$ via $t \mapsto \overline{t}$ or $t \mapsto \overline{t}^{-1}$. 
In the first case, $\H^1(\Gamma,\Aut(X)_{\F_0})$ is a singleton, while in the second case it has two elements. 

If $b \neq |c|$ and $\Gamma$ acts canonically on $X$ (i.e.~via complex conjugation), then $\Gamma$ acts on $\G_{m,\C} \subseteq \Aut(X)_{\F_0})$ via $t \mapsto \overline{t}$, and so $\H^1(\Gamma,\Aut(X)_{\F_0})$ is a singleton, which implies that $\H^1(\Gamma,\Aut(X))$ is also a singleton.

Assume now that $b=|c|$. The canonical $\Gamma$-action on $X$ induces the canonical $\Gamma$-action on $\F_{0,\C}=\P_\C^1 \times \P_\C^1$, and then again $\H^1(\Gamma,\Aut(X)_{\F_0})$ is a singleton.
On the other hand, if we let $\Gamma$ act on $X$ via the real structure $\theta$ defined in the statement, then the induced $\Gamma$-action on $\F_{0,\C}$ corresponds to the real form $\fR_{\C/\R}(\P_\C^1)$ and one can check that $\H^1(\Gamma,\Aut(X)_{\F_0})$ has one element if $b=-c$ and two elements if $b=c$ (because in the first case $\Gamma$ acts on $\G_{m,\C} \subseteq \Aut(X)_{\F_0}$ via $t \mapsto \overline{t}$ while in the second case $\Gamma$ acts via $t \mapsto \overline{t}^{-1}$).

In the latter case, these two elements correspond to the real structures $\theta$ and 
\[
\theta'\colon [x_0:x_1;y_0:y_1;z_0:z_1] \mapsto [-\overline{x_1}: \overline{x_0};\overline{z_0}:\overline{z_1};\overline{y_0}:\overline{y_1}].
\]
The real locus of $\theta'$ is empty, and so $\theta'$ corresponds to a real form of $X$ with no real points. 
On the other hand, $\theta$ corresponds to a real toric threefold with a dense open orbit isomorphic to $\G_{m,\R} \times \fR_{\C/\R}(\G_{m,\C})$ if $b=-c$ resp. to $\SSS \times \fR_{\C/\R}(\G_{m,\C})$ if $b=c$, and is therefore rational. This finishes the proof.
\end{proof}

\begin{remark}\label{rk: auto group of Gb and Hb}
Using \eqref{eq: real structures for Gb and Hb} in Proposition \ref{prop: k-forms of Fabc with a=0} and the results from Section  \ref{subsubsec: decomposable P1 bundles}, we can easily describe the real algebraic groups $\Autz(\mathcal{G}_b)$ and $\Autz(\mathcal{H}_b)$.
\begin{itemize}
\item We have $\mathcal{G}_0=\P_\R^1 \times \fR_{\C/\R}(\P_\C^1)$, and so $\Autz(\mathcal{G}_0)=\PGL_{2,\R} \times \fR_{\C/\R}(\PGL_{2,\C})$.
\item Let $b \geq 1$. Then $\Autz(\mathcal{G}_b)$ fits into an exact sequence
\[
1 \to \left(\G_{a,\R}^{(b+1)^2} \rtimes \G_{m,\R}\right) \to \Autz(\mathcal{G}_b) \to \fR_{\C/\R}(\PGL_{2,\C}) \to 1.
\]
\item Let $b \geq 1$. Then $\Autz(\mathcal{H}_b)$ fits into an exact sequence
\[
1 \to \SSS \to \Autz(\mathcal{H}_b) \to \fR_{\C/\R}(\PGL_{2,\C}) \to 1.
\] 
\end{itemize}
\end{remark}

\section{Forms of singular \texorpdfstring{$\P^1$}{P1}-fibrations over \texorpdfstring{$\P(1,1,2)$}{P(1,1,2)}}

\begin{proposition}\label{prop: k-forms of Wb}
Let $b \geq 2$.
The trivial $\k$-form $\W_{b,\k}$ is the only rational $\k$-form of the $\bk$-variety $\W_{b,\bk}$.
\end{proposition}

\begin{proof}
Given a descent datum $\Gamma \times \F_{2,\bk} \to \F_{2,\bk}$, there exists a unique descent datum $\Gamma \times \P(1,1,2)_{\bk} \to \P(1,1,2)_{\bk}$ that makes $\Gamma$-equivariant the contraction morphism $\F_{2,\bk} \to \P(1,1,2)_{\bk}$, and the converse is also true since the singular point of $\P(1,1,2)_{\bk}$ is necessarily fixed for any $\Gamma$-action. This implies that the trivial $\k$-form $\P(1,1,2)_\k$ is the unique (up to isomorphism) rational $\k$-form of $\P(1,1,2)_{\bk}$ since the trivial $\k$-form $\F_{2,\k}$ is the unique (up to isomorphism) rational $\k$-form of $\F_{2,\bk}$ (see Lemma \ref{lem: k-forms of Fa}). 
Moreover, Lemma~\ref{lem: k-forms of Fa} also implies that, if $Y$ is a nontrivial $\k$-form of $\P(1,1,2)_{\bk}$, then $Y(\k)$ is a singleton (corresponding to the singular point of $\P(1,1,2)_{\bk}$).

The proof of the lemma is then analogous to that of Proposition \ref{prop: k-forms of PPb}. Indeed, denoting $X=\W_{b,\bk}$, Lemma \ref{lem:WbPicN} and Remark \ref{rk: description auto of P1-fibration of Wb} imply that we have a $\Gamma$-equivariant short exact sequence
\[
1 \to \Aut(X)_{\P(1,1,2)} \to \Aut(X) \to \Aut\left(\P(1,1,2)_{\bk}\right) \to 1
\]
which induces an exact sequence in Galois cohomology
\[
\H^1\left(\Gamma,\Aut(X)_{\P(1,1,2)}\right) \to \H^1(\Gamma,\Aut(X)) \to \H^1\left(\Gamma,\Aut\left(\P(1,1,2)_{\bk}\right)\right).
\]
Hence, since the trivial $\k$-form $\P(1,1,2)_\k$ is the unique $\k$-form of $\P(1,1,2)_{\bk}$ with more than one $\k$-point (up to isomorphism), it suffices to show that $\H^1(\Gamma,\Aut(X)_{\P(1,1,2)})$ is a singleton to prove the proposition. 
By Remark \ref{rk: description auto of P1-fibration of Wb}, we have that $\Aut(X)_{\P(1,1,2)} \simeq U \rtimes \G_{m,\bk}$, where $U$ is a $\Gamma$-stable unipotent subgroup and $\Gamma$ acts naturally on $\G_{m,\bk}$, so $\H^1(\Gamma,\Aut(X)_{\P(1,1,2)})$ is a singleton, which concludes the proof.
\end{proof}

\begin{remark}
Any $\k$-form of $\W_{b,\k}$ as $\k$-points; indeed, the two singular points of $\W_{b,\k}$ are both fixed by any $\Gamma$-action.
\end{remark}

\section{Forms of the \texorpdfstring{$\P^1$}{P1}-bundles \texorpdfstring{$\U_{a}^{b,c} \to \F_a$}{Uabc} and \texorpdfstring{$\V_b \to \P^2$}{Vb}}\label{sec: real forms of Uab}
In this section we determine the rational  $\k$-forms of the $\P^1$-bundles $\U_{a,\bk}^{b,c} \to \F_{a,\bk}$ (case \hyperlink{th:D_c}{(c)}) and $\V_{b,\bk} \to \P_{\bk}^2$ (case \hyperlink{th:D_e}{(e)}).

\begin{proposition}\label{prop: k-forms of Umemura bundles}
The $\bk$-variety $\U_{a,\bk}^{b,c}$ has no nontrivial $\k$-form with a $\k$-point. 
In particular, the trivial $\k$-form $\U_{a,\k}^{b,c}$ is the only rational $\k$-form of $\U_{a,\bk}^{b,c}$.
\end{proposition}

\begin{proof}
Let us denote $X=\U_{a,\bk}^{b,c}$.
The proof is analogous to that of Proposition \ref{prop: k-forms of Fabc with a>1}.
The fact that $\H^1(\Gamma,\Aut(X)_{\F_a})$ is a singleton is an immediate consequence of the fact that the group $\Aut(X)_{\F_a}$ is unipotent (see \cite[Remark 3.6.4]{BFT23}).
\end{proof}

\begin{corollary}\label{cor: k-forms of Vb}
The $\bk$-variety $\V_{b,\bk}$, with $b \geq 2$, has no nontrivial $\k$-form with a $\k$-point.
In particular, the trivial $\k$-form $\V_{b,\k}$ is the only rational $\k$-form of $\V_{b,\bk}$.
\end{corollary}

\begin{proof}
Let us denote $Y=\V_{b,\bk}$ and $X=\U_{1,\bk}^{b,2} $.
Let $\psi\colon X \to Y$ be the blow-up of the smooth rational curve ${q'}^{-1}([0:1:0])$ in $Y$. The cone of effective curves of $Y$ is generated by two curves $f'$ and $s'$ that satisfy $K_{Y} \cdot f'=-2$ and $K_{Y} \cdot s'=b-3 \geq -1$ (see \cite[Lemma 6.4.2]{BFT22}), and so the structure morphism $Y \to \P_{\bk}^2$ is $\Gamma$-equivariant for any $\Gamma$-action on $Y$. Similarly, it follows from  \cite[Lemma 6.4.1]{BFT22} that, for any $\Gamma$-action on $X$, the structure morphism $X \to \F_{1,\bk}$ and the divisorial contraction $\psi$ are $\Gamma$-equivariant.

Consider the canonical $\Gamma$-actions on $X$ and $Y$.
Since $\psi \Aut(X) \psi^{-1}=\Aut(Y)$ (by Corollary \ref{cor:Aut-Vb}), we have an isomorphism of $\Gamma$-groups $\Aut(Y) \simeq \Aut(X),\ \varphi \mapsto \varphi':= \psi^{-1} \circ \varphi \circ \psi$. This isomorphism induces a natural bijection of pointed sets between the $\Gamma$-actions on $X$ and on $Y$, resp. between the $\k$-forms of $X$ and of $Y$. 

Consider now a $\Gamma$-action on $Y$ corresponding to a $\k$-form $Y_0$ of $Y$ with a $\k$-point. This $\Gamma$-action lifts uniquely to $X$  and we have a $\Gamma$-equivariant commutative diagram
\[\xymatrix@R=4mm@C=2cm{
X \ar[r]^{\psi} \ar[d]  & Y \ar[d] \\
   \F_{1,\bk} \ar[r] & \P_{\bk}^2
  }\]  
which corresponds, in the category of $\k$-varieties, to the commutative diagram
\[\xymatrix@R=4mm@C=2cm{
X_0 \ar[r] \ar[d]  & Y_0 \ar[d] \\
   Z_0 \ar[r] & W_0
  }\]      
By assumption, $Y_0(\k) \neq \varnothing$, and thus $W_0(\k) \neq \varnothing$. In fact $W_0 \simeq \P_\k^2$ by Ch\^atelet's theorem, and so $W_0(\k)$ is Zariski dense in $W_0$. Since $Z_0 \to W_0$ is birational, $Z_0(\k) \neq \varnothing$, and so necessarily, we have that $Z_0 \simeq \F_{1,\k}$ (by Lemma \ref{lem: k-forms of Fa}) and thus $X_0$ is isomorphic to $\U_{1,\k}^{b,2}$ (see the proof of Proposition \ref{prop: k-forms of Umemura bundles}). Hence, the trivial $\k$-form $\V_{b,\k}$ is the only $\k$-form of $Y$ with a $\k$-point. 
\end{proof}

\section{Real forms of Schwarzenberger \texorpdfstring{$\P^1$}{P1}-bundles}\label{sec: Sb P1-bundles}
In this section we determine all the real forms of the Schwarzenberger $\P^1$-bundles $\SS_{b,\C} \to \P_\C^2$ (case \hyperlink{th:D_d}{(d)}).

\begin{proposition}\label{prop: real forms of S1}
The complex flag variety $\SS_{1,\C} \simeq \PGL_{3,\C}/B$, where $B$ is a Borel subgroup of $\PGL_{3,\C}$, has three non-isomorphic real forms:
\begin{itemize}
\item The trivial real form $\SS_{1,\R}$, which is rational, and such that $\Autz(\SS_{1,\R}) \simeq \PGL_{3,\R}$.
\item A nontrivial real form $\widetilde{\SS}_{1,\R}$, which is also rational, and such that  $\Autz(\widetilde{\SS}_{1,\R}) \simeq \PSU(1,2)$.
\item A nontrivial real form $\widehat{\SS}_{1,\R}$, which has no real points, and such that  $\Autz(\widehat{\SS}_{1,\R}) \simeq \PSU(3)$.
\end{itemize}
\end{proposition}

\begin{proof}
Since $\SS_{1,\C} \simeq \PGL_{3,\C}/B$ is a flag variety, \cite[Theorem 1]{Dem77} implies that 
\[
\Aut_\C(\SS_{1,\C}) \simeq \Aut_{\mathrm{gr}}(\PGL_{3,\C})  \simeq \PGL_{3,\C} \rtimes \Z/2\Z
\] 
as $\Gamma$-groups (for the canonical $\Gamma$-action on $\SS_{1,\C}$ and the natural $\Gamma$-action on $\PGL_{3,\C} \rtimes \Z/2\Z$ induced by the complex conjugation coefficientwise on $\PGL_{3,\C}$), from which we deduce that the set $\H^1(\Gamma,\Aut_\C(\SS_{1,\C})) \simeq \H^1(\Gamma,\Aut_{\mathrm{gr}}(\PGL_{3,\C}))$ has three elements (corresponding to the three non-isomorphic real forms of the complex algebraic group $\PGL_{3,\C}$ in the category of real algebraic groups; see e.g.~\cite[Section V\!I.10]{Kna02}). An explicit computation (e.g.~using \cite[Examples 1.12 and 2.13]{MJT18}) yields that $\SS_{1,\C}$ has two non-isomorphic real forms with real points, say $\SS_{1,\R}$ and $\widetilde{\SS}_{1,\R}$, such that $\Autz(\SS_{1,\R}) \simeq \PGL_{3,\R}$ and  $\Autz(\widetilde{\SS}_{1,\R}) \simeq \PSU(1,2)$, and another real form with no real points, say $\widehat{\SS}_{1,\R}$, such that $\Autz(\widehat{\SS}_{1,\R}) \simeq \PSU(3)$. Moreover, $\Pic(\widetilde{\SS}_{1,\R}) \simeq \Pic(\SS_{1,\C})^\Gamma \simeq \Z$ since the corresponding $\Gamma$-action on $\SS_{1,\C}$ swaps the two extremal rays of $\NE(\SS_{1,\C})$. Thus \cite[Theorem 1.2~(ii)]{KP23b} implies that $\widetilde{\SS}_{1,\R}$ is rational, which finishes the proof. 
\end{proof}

\begin{example}\label{ex: real structures for S1}
Identifying 
\[
\SS_{1,\C} \simeq \left\{\left([x_0:x_1:x_2],[y_0:y_1:y_2]\right)\in \P_\C^2 \times \P_\C^2 \  \middle| \ \sum_i x_iy_i=0\right\},
\]
we can write down the real structures on $\SS_{1,\C}$ corresponding to the three non-isomorphic real forms given in Proposition \ref{prop: real forms of S1}. 
\begin{itemize}
\item The restriction of the canonical real structure on $\P_\C^2 \times \P_\C^2$ corresponds to the trivial real form $\SS_{1,\R}$.
\item The restriction of the real structure
$([x_0:x_1:x_2],[y_0:y_1:y_2]) \mapsto ([-\overline{y_0}:\overline{y_1}:\overline{y_2}],[-\overline{x_0}:\overline{x_1}:\overline{x_2}])$ corresponds to the nontrivial rational real form $\widetilde{\SS}_{1,\R}$. Indeed, the real locus is non-empty and the $\Gamma$-action swaps the two extremal rays of $\NE(\SS_{1,\C})$ corresponding to the two projections onto each $\P_\C^2$ factor.
\item The restriction of the real structure
$([x_0:x_1:x_2],[y_0:y_1:y_2]) \mapsto ([\overline{y_0}:\overline{y_1}:\overline{y_2}],[\overline{x_0}:\overline{x_1}:\overline{x_2}])$ corresponds to the real form with no real points $\widehat{\SS}_{1,\R}$. Indeed, the real locus is empty.\end{itemize}
\end{example}

\begin{lemma}{\emph{(see \cite[Proposition 8.1(i)]{FJSVV22})}}\label{lem:FJSVV22} 
Let $\pi\colon X\to \P_\k^2$ be a conic fibration such that $X(\k)\neq\varnothing$ and $\pi_{\bk}\colon X_{\bk}\to \P_{\bk}^2$ is a $\P^1$-bundle. Then $X$ is rational.
\end{lemma}

\begin{proposition}\label{prop: real forms of Sb}
Let $b\geq 2$. The complex threefold $\SS_{b,\C}$ has two non-isomorphic real forms: 
\begin{itemize}
\item The trivial real form $\SS_{b,\R}$, which is rational, and such that $\Aut(\SS_{b,\R}) \simeq \PGL_{2,\R}$.
\item A nontrivial real form $\widetilde{\SS}_{b,\R}$, which is rational if $b$ is odd and has no real points if $b$ is even, and such that  $\Aut(\widetilde{\SS}_{b,\R}) \simeq \SO_{3,\R}$.
\end{itemize}
\end{proposition}

\begin{proof}
It follows from \cite[Lemma 5.3.9]{BFT22}, which describes the cone of effective curves and the intersection form on $X=\SS_{b,\C}$, and Proposition \ref{prop: aut group of Sb} that we have the following $\Gamma$-equivariant isomorphisms
\[
\Aut_\C(X) \simeq  \Aut(\P^2_\C,C_0) \simeq \Aut(C_0).
\]
These isomorphisms induce a bijection $\H^1(\Gamma,\Aut(X)) \simeq \H^1(\Gamma,\Aut(C_0))$, and so $X$ has exactly two non-isomorphic real forms: the trivial real form $\SS_{b,\R}$, which is rational and satisfies $\Autz(\SS_{b,\R}) \simeq \PGL_{2,\R}$, and a nontrivial real form that we denote by $\widetilde{\SS}_{b,\R}$. 

Recall that $X$ is an almost homogeneous $\SL_{2,\C}$-threefold (see \cite[Remark 4.2.7]{BFT23}). It follows from \cite[Corollary 5.17 and Appendix C]{MJT20} that $\Aut(\widetilde{\SS}_{b,\R}) \simeq \SO_{3,\R}$ and that $\widetilde{\SS}_{b,\R}(\R) \neq \varnothing$ if and only if $b$ is odd. Moreover, since any $\Gamma$-action on $X$ fixes the extremal rays of $\NE(X)$, there is a conic fibration structure $\widetilde{\SS}_{b,\R} \to \P_\R^2$ whose complexification yields the $\P^1$-bundle structure $X \to \P_\C^2$, and so $\widetilde{\SS}_{b,\R}$ is rational if and only if $\widetilde{\SS}_{b,\R}(\R) \neq \varnothing$ by Lemma \ref{lem:FJSVV22}. This concludes the proof.
\end{proof}

An explicit description of the threefold $\widetilde{\SS}_{b,\R}$ is given in the following example. It is also shown that, when $b$ is odd, $\widetilde{\SS}_{b,\R} \to \P_\R^2$ is a $\P^1$-bundle.

\begin{example}\label{ex: real structures for Sb}
Let $b \geq 2$.
Denote by $U_0,U_1\subset \p_\C^2$ the two open subsets 
\[U_0= \{[X:Y:Z] | X \neq 0\} \simeq \A_\C^2,\quad U_1= \{[X:Y:Z] | Z \neq 0\} \simeq \A_\C^2.\] 
By \cite[Lemma 4.2.1]{BFT22}, the restriction of $\SS_{b,\C}$ on $\P_\C^2 \setminus \{ [0:1:0] \}$ is obtained by gluing $\p_\C^1 \times U_0$ and $\p_\C^1 \times U_1$ along $\p_\C^1 \times (U_0 \cap U_1)$ via the isomorphism given by
\[
\begin{array}{rcccccc}
\theta \colon & \p_\C^1 \times U_0 &\dasharrow &\p_\C^1 \times U_1 \\
&\left(\begin{bmatrix} x_0 \\ x_1 \end{bmatrix} ,[1:u:v] \right) & \mapsto & \left( \hat{A}(u,v) \begin{bmatrix} x_0 \\ x_1 \end{bmatrix},\left[\dfrac{1}{v}:\dfrac{u}{v} :1\right]\right),\end{array}
\]
where $A(u,v) \in \GL_2(\C[u,v,\frac{1}{v}])$ is uniquely defined by the equality
\[ 
A(s+t,st)=\dfrac{1}{s-t}\begin{bmatrix} s^{b}-t^{b} & st(s^{b-1}-t^{b-1}) \\
 s^{b+1}-t^{b+1}& st(s^{b}-t^{b})\end{bmatrix}
 \]
and $\hat{A}(u,v)$ is the image of $A(u,v)$ in $\PGL_2(\C[u,v,\frac{1}{v}])$. 
One can check that 
\begin{equation}\label{eq: relation for A}
A(s+t,st)\,A\left(\frac{-(s+t)}{st},\frac{1}{st}\right)=(-1)^{b-1}I_2\  \text{ in }\GL_2(\C[u,v,\frac{1}{v}]),
\end{equation} 
and thus
$\hat{A}(s+t,st)\hat{A}(\frac{-(s+t)}{st},\frac{1}{st})=\hat{I_2}$ in $\PGL_2(\C[u,v,\frac{1}{v}])$.

The antiregular map
\[
\begin{array}{rcccccc}
\gamma \colon & \p^1 \times U_1 &\rightarrow &\p^1 \times U_0 \\
&\left(\begin{bmatrix} x_0 \\ x_1 \end{bmatrix} ,[p:q:1] \right) & \mapsto & \left( \begin{bmatrix} \overline{x_0} \\ \overline{x_1} \end{bmatrix},\left[1:-\overline{q}:\overline{p}\right]\right),\end{array}
\]
gives rise to a real structure $\widetilde{\gamma}$ on $\SS_{b,\C}$. Indeed, one can check (using \eqref{eq: relation for A}) that $(\theta \circ \gamma)_{|\p^1 \times (U_0 \cap U_1)}$ is an involution  and then notice that any antiregular involution defined on the inverse image of $\P_\C^2 \setminus \{[0:1:0]\}$ via the structure morphism $\SS_{b,\C} \to \P_\C^2$ extends uniquely to $\SS_{b,\C}$. 
Let us note that the real structure defined by  $\widetilde{\gamma}$ is not equivalent to the canonical real structure since the $\Gamma$-action on $C_0 \subset \P_\C^2$ induced by $\widetilde{\gamma}$, which is given by $[X:Y:Z] \mapsto [\overline{Z}:-\overline{Y}:\overline{X}]$, corresponds to the conic with no real points. 

If moreover $b$ is odd, then \eqref{eq: relation for A} implies that the real structure $\widetilde{\gamma}$ on $\SS_{b,\C}$ is induced by a $\Gamma$-linearization of a rank $2$ vector bundle $\mathcal{E}_b \to \P_\C^2$, whose transition function is given by 
\[
\begin{array}{rcccccc}
\widetilde{\theta} \colon & \A^2_\C \times U_0 &\dasharrow &\A_\C^2 \times U_1 \\
&\left(\begin{pmatrix} x_0 \\ x_1 \end{pmatrix} ,[1:u:v] \right) & \mapsto & \left( A(u,v) \begin{pmatrix} x_0 \\ x_1 \end{pmatrix},\left[\dfrac{1}{v}:\dfrac{u}{v} :1\right]\right),\end{array}
\]
such that $\P(\mathcal{E}_b) \simeq \SS_{b,\C}$ as $\P^1$-bundles over $\P_\C^2$. In particular, $\widetilde{\SS}_{b,\R} \to \P_\R^2$ is a $\P^1$-bundle and we recover the fact that $\widetilde{\SS}_{b,\R}$ is rational. 
\end{example}

\section{Real forms of Umemura quadric fibrations}\label{sec: real forms of Umemura quadric fibrations}
In this section we determine all the real forms of the Umemura quadric fibrations $\QQ_{g,\C} \to \P_\C^1$ (case \hyperlink{th:D_h}{(h)}), where $g \in \C[u_0,u_1]$ is a homogeneous polynomial of degree $2n \geq 2$. The two main results of this section are Proposition \ref{prop: existence real forms Qg} and Theorem \ref{th: real forms of Qg}.

\subsection{Existence of real forms for $\QQ_{g,\C}$} 
Let us start by determining for which $g \in \C[u_0,u_1]$ the complex threefold $\QQ_{g,\C}$ admits real forms.

\begin{lemma}\label{lem: iso classes for Qg}
Let $g,g' \in \C[u_0,u_1]$ be two homogeneous polynomials of degree $2n\geq 2$. 
Then $\QQ_{g,\C} \simeq \QQ_{g',\C}$ if and only if there exist $\lambda \in \C^*$ and $\varphi=\begin{bmatrix}
a & b \\ c & d
\end{bmatrix} \in \SL_{2}(\C)$ such that $g'=\lambda (g \circ \varphi)$.
\end{lemma} 

\begin{proof}
The ``if'' part is clear; indeed, the automorphism of $\P(\O_{\P^1}^{\oplus 3}\oplus \O_{\P^1}(n))$ given by
\[
[x_0:x_1:x_2:x_3;u_0:u_1] \mapsto [x_0:x_1:x_2:\sqrt{\lambda}x_3;au_0+bu_1:cu_0+du_1]
\]
induces an isomorphism $\QQ_{g',\C} \simeq \QQ_{g,\C}$.
 
Let us now assume that $\QQ_{g,\C} \simeq \QQ_{g',\C}$. Then the hyperelliptic curves $C$ and $C'$, defined respectively by the equations $g(u_0,u_1)-u_2^2=0$ and $g'(u_0,u_1)-u_2^2=0$ in the weighted projective space $\P(1,1,n)_\C$, are isomorphic (see \cite[Lemma 4.4.5]{BFT22} and its proof). 

If $n \geq 2$, any automorphism of $\P(1,1,n)_\C$ can be written
\[ [u_0:u_1:u_2] \mapsto [au_0+bu_1:cu_0+du_1: \frac{1}{\sqrt{\lambda}} u_2+p(u_0,u_1)],\]
where $\varphi:=\begin{bmatrix}
a &b \\ c & d
\end{bmatrix} \in \SL_2(\C)$, $\lambda \in \C^*$, and $p \in \C[u_0,u_1]$ is a homogeneous polynomial of degree $n$. But such an automorphism of $\P(1,1,n)_\C$ maps $C'$ to $C$ if and only if $p=0$ and $g'=\lambda(g \circ \varphi)$, which proves the claim for $n \geq 2$.

If $n=1$, then $C$ and $C'$ are plane conics. As their anticanonical divisor induces an embedding into $\p^2_\C$, any isomorphism $C'\simeq C$ extends to an automorphism $\alpha$ of $\p_\C^2$. 
Since the equations of $C$ and $C'$ do not contain any $u_0u_2$ and $u_1u_2$ terms, $\alpha$ must be of the form 
\[
 \mathrm{PGL}_3(\C)\ni\alpha=\begin{bmatrix}a&b&0\\ c&d&0\\ 0&0&\mu\end{bmatrix}
\quad\text{with}\quad
\varphi:=\begin{bmatrix}
a & b \\ c & d
\end{bmatrix}\in\SL_2(\C).
\]
Then
\[
\left(g'(u_0,u_1)-u_2^2=0\right)=C'=\alpha^{-1}(C)=\left(g(\varphi(u_0,u_1))-(\mu u_2)^2=0\right)
=\left(\frac{1}{\mu^2}g(\varphi(u_0,u_1))-u_2^2=0\right),
\]
which proves the claim for $n=1$.
\end{proof}

The proof of the next result is inspired by the proof of \cite[Proposition 3.11]{BBP23}, where Blanc-Bot-Poloni determined when the so-called \textit{Danielewski surfaces} admit real forms.

\begin{proposition}\label{prop: existence real forms Qg}
Let $g \in \C[u_0,u_1]$ be a homogeneous polynomial of degree $2n \geq 2$.
The following conditions are equivalent:
\begin{enumerate}
\item\label{item a 4.42} The complex threefold $\QQ_{g,\C}$ admits a real form.
\item\label{item b 4.42} There exist $\lambda \in \C^*$ and $\varphi=\begin{bmatrix}
a & b \\ c & d
\end{bmatrix} \in \SL_{2}(\C)$ such that $\lambda (g \circ \varphi) \in \R[u_0,u_1]$.
\item\label{item c 4.42} There exists $g' \in \R[u_0,u_1]$ such that $\QQ_{g,\C} \simeq \QQ_{g',\C}$.
\end{enumerate}
Moreover, for $g \in \R[u_0,u_1]$, a canonical real structure on $\QQ_{g,\C}$ is given by 
\begin{equation}\label{eq: canonical real structure on Qg}
 \mu_1\colon  [x_0:x_1:x_2:x_3;u_0:u_1] \mapsto [\overline{x_0}:\overline{x_1}:\overline{x_2}:\overline{x_3};\overline{u_0}:\overline{u_1}].
 \end{equation}
\end{proposition}

\begin{proof}
If $g \in \R[u_0,u_1]$, then it is clear that $\mu_1$ defines a real structure on $\QQ_{g,\C}$.
Hence, \ref{item c 4.42} implies \ref{item a 4.42}. The equivalence \ref{item b 4.42} $\Leftrightarrow$ \ref{item c 4.42} is Lemma \ref{lem: iso classes for Qg}.
Let us prove \ref{item a 4.42} $\Rightarrow$ \ref{item b 4.42}.

Since \ref{item b 4.42} is easily verified when $g$ is a monomial, we may assume that $g$ is not a monomial. 
Replacing $g$ by some $\lambda (g \circ \varphi)$ (see Lemma \ref{lem: iso classes for Qg}), we may further assume that 
\begin{equation}\label{eq: reduced form for g}
g(u_0,u_1)=u_0^r u_1^{2n-r} + \sum_{s \leq r-2} c_s u_0^s u_1^{2n-s},\ \ \text{for some $r \in \{2,\ldots,2n\}$, and with $c_s \in \C$.}
\end{equation}
Assume that $\QQ_{g,\C}$ admits a real form, i.e.~that there exists a real structure $\mu\colon \QQ_{g,\C} \to \QQ_{g,\C}$. 
Then $\mu$ induces an isomorphism of complex threefolds $\theta\colon \QQ_{g,\C} \simeq \QQ_{\overline{g},\C}=\QQ_{g,\C} \times_{\Spec(\C)} \Spec(\C)$, where $\Spec(\C) \to \Spec(\C)$ is the morphism $\Spec(z \mapsto \overline{z})$  and $\overline{g} \in \C[u_0,u_1]$ is the homogeneous polynomial of degree $2n$ obtained from $g$ by taking the complex conjugate of each coefficient. Indeed, we have the following commutative diagram (whose square is Cartesian):
\[\xymatrix@R=4mm@C=2cm{
   \QQ_{g,\C} \ar@/^/[rrd]^{\mu} \ar@/_/[rdd] \ar@{.>}[rd]^{\theta} \\
    &\QQ_{\overline{g},\C}  \ar[d] \ar[r]  & \QQ_{g,\C}  \ar[d] \\
  &{\Spec(\C)} \ar[r]^{\Spec(z \mapsto \overline{z})} & {\Spec(\C)} }
\]
According to Lemma \ref{lem: iso classes for Qg}, there exist $\lambda \in \C^*$ and $\varphi=\begin{bmatrix}
a & b \\ c & d
\end{bmatrix} \in \SL_{2}(\C)$ such that $\overline{g}=\lambda (g \circ \varphi)$.
Using \eqref{eq: reduced form for g}, we see that necessarily $b=c=0$,  $\lambda=(a^rd^{2n-r})^{-1}=a^{2(n-r)}$ and $|a|=1$ (because if $c_{s_0} \neq 0$, then we must have $\overline{c_{s_0}}=\lambda a^{2(s_0-n)} c_s=a^{2(s_0-r)}c_{s_0}$, hence $|a|=1$ by taking the modulus on both sides of the equality).
Let $\alpha \in \C^*$ be such that $\alpha^2=a$; in particular, $|\alpha|=1$. Let $g'(u_0,u_1):=\alpha^{2(n-r)}g(\alpha u_0,\alpha^{-1}u_1)$. 
Let us check that $g' \in \R[u_0,u_1]$, which will conclude the proof:
{\small
\[
\overline{g'}(u_0,u_1)=\alpha^{2(r-n)}\overline{g}(\alpha^{-1} u_0,\alpha u_1)=\alpha^{2(r-n)}\lambda {(g \circ \varphi)}(\alpha^{-1} u_0,\alpha u_1)=\alpha^{2(n-r)}g(\alpha u_0,\alpha^{-1}u_1)=g'(u_0,u_1).
\]}
\end{proof}

\subsection{Classification of the real forms of $\QQ_{g,\C}$ through Galois cohomology}\label{subsec: real forms Qg via Galois cohomology}

From now on, and without loss of generality (see Proposition \ref{prop: existence real forms Qg}), we consider real forms of $\QQ_{g,\C}$ only when $g \in \R[u_0,u_1]$. As before, we will denote by $\QQ_{g,\R}$ the trivial real form of $\QQ_{g,\C}$. Our goal, in the rest of Section \ref{sec: real forms of Umemura quadric fibrations}, is to determine the other real forms of $\QQ_{g,\C}$ using Galois cohomology.

\begin{proposition}\label{propo: first four real forms of Qg}
Let $g \in \R[u_0,u_1]$ be a homogeneous polynomial, which is not a square and is of degree $2n$ for some $n \in \N_{\geq 1}$.
Let $\Gamma$ act on $\Aut(\QQ_{g,\C})$ by $\mu_1$-conjugation, where $\mu_1\colon \QQ_{g,\C} \to \QQ_{g,\C}$ is the canonical real structure defined by \eqref{eq: canonical real structure on Qg}. Then the short exact sequence \eqref{eq: ES for Aut(Qg)} of Proposition \ref{prop: aut group of Qg} induces an exact sequence in Galois cohomology 
\[
\H^1(\Gamma,\Aut(\QQ_{g,\C})_{\P^1}) \to \H^1(\Gamma,\Aut(\QQ_{g,\C})) \to \H^1(\Gamma,F).
\] 
The set $\H^1(\Gamma,\Aut(\QQ_{g,\C})_{\P^1})$ parametrizes the four real forms of $\QQ_{g,\C}$ defined as the hypersurfaces in $\P(\O_{\P^1}^{\oplus 3}\oplus \O_{\P^1}(n))$ given by the equations
\begin{itemize}
\item $\QQ_{g,\R}\colon\ x_0^2-x_1x_2-g(u_0,u_1)x_3^2=0$ (the trivial real form, always rational);
\item $\QQ'_{g,\R} \colon\ x_0^2-x_1x_2+g(u_0,u_1)x_3^2=0$ (always rational); 
\item $\TT_{g,\R}\colon\ x_0^2+x_1^2+x_2^2-g(u_0,u_1)x_3^2=0$; and
\item $\TT'_{g,\R} \colon\ x_0^2+x_1^2+x_2^2+g(u_0,u_1)x_3^2=0$.
\end{itemize}
Moreover, the following hold:
\begin{itemize}
\item Assume that $g$ has at least three (complex) roots. If $X \in \{\QQ_{g,\R},\QQ'_{g,\R}\}$, then $\Autz(X) \simeq \PGL_{2,\R}$, and if $X \in \{\TT_{g,\R},\TT'_{g,\R}\}$, then $\Autz(X) \simeq \SO_{3,\R}$. In particular, $\QQ_{g,\R}$ and $\TT_{g,\R}$ are two non-isomorphic real forms of $\QQ_{g,\C}$.
\item Assume that $g$ has exactly two roots. Up to a complex-linear change of coordinates, we can assume that $g=u_0^au_1^b$ for some $a,b \geq 1$.
If $X\in \{\QQ_{g,\R},\QQ'_{g,\R}\}$, then $\Autz(X)\simeq \PGL_{2,\R}\times \G_{m,\R}$ and if $X \in \{\TT_{g,\R},\TT'_{g,\R}\}$, then $\Autz(X) \simeq \SO_{3,\R}\times \G_{m,\R}$, where $\G_{m,\R} \subset\Aut(\p^1_\R)$ is the real multiplicative group preserving the two real roots of $g$.
\end{itemize}
\end{proposition}

\begin{proof}
Recall that the cone of effective curves on $\QQ_{g,\C}$ is generated by two curves $f$ and $h$ that satisfy 
$K_{Q_g}\cdot f=-2$ and $K_{Q_g}\cdot h=n-2\geq -1$ (see \cite[Lemma 4.4.3]{BFT22} and its proof). 
Hence $\Gamma$ acts on $\p_\C^1$ and the structure morphism $\pi_g\colon \QQ_{g,\C}\to\p_\C^1$ is $\Gamma$-equivariant. 
By Proposition \ref{prop: aut group of Qg}, we have therefore a short exact sequence of $\Gamma$-groups
\[
1 \to \Aut(\QQ_{g,\C})_{\P^1} \to \Aut(\QQ_{g,\C}) \to F \to 1
\] 
which induces an exact sequence in Galois cohomology
\[
\H^1(\Gamma,\Aut(\QQ_{g,\C})_{\P^1}) \to \H^1(\Gamma,\Aut(\QQ_{g,\C})) \to \H^1(\Gamma,F).
\] 

Let us now determine the real forms of $\QQ_{g,\C}$ parametrized by $\H^1(\Gamma,\Aut(\QQ_{g,\C})_{\P^1})$. 
We have $\Aut(\QQ_{g,\C})_{\P^1} \simeq \PGL_{2}(\C) \times \Z/2\Z$ as $\Gamma$-groups, where $\Gamma$ acts on $\PGL_2(\C)$ via complex conjugation on each coefficient and trivially on $\Z/2\Z$. Thus 
\[
\H^1(\Gamma,\Aut(\QQ_{g,\C})_{\P^1}) \simeq \H^1(\Gamma,\PGL_{2}(\C)) \times \H^1(\Gamma,\Z/2\Z) \simeq \Z/2\Z \times \Z/2\Z,
\]
where the first copy of $\Z/2\Z$ is generated by the class of the involution $\tau\colon [x_0:x_1:x_2:x_3;u_0:u_1]\mapsto [-x_0:x_2:x_1:x_3;\ u_0:u_1]$ while the second copy of $\Z/2\Z$ is generated by the class of the involution $\sigma\colon [x_0:x_1:x_2:x_3;\ u_0:u_1]\mapsto [x_0:x_1:x_2:-x_3;\ u_0:u_1]$. 
We deduce that $\H^1(\Gamma,\Aut(\QQ_{g,\C})_{\P^1})$ parametrizes the equivalence classes of the four real structures 
\begin{align*}
&\mu_1&\colon &[x_0:x_1:x_2:x_3;\ u_0:u_1] &\mapsto &\ \ [\overline{x_0}:\overline{x_1}:\overline{x_2}:\overline{x_3};\ \overline{u_0}:\overline{u_1}];&\\
&\mu_2:=\sigma \circ \mu_1&\colon &[x_0:x_1:x_2:x_3;\ u_0:u_1] &\mapsto &\ \ [\overline{x_0}:\overline{x_1}:\overline{x_2}:-\overline{x_3};\ \overline{u_0}:\overline{u_1}]; &\\
&\mu_3:=\tau \circ \sigma \circ \mu_1&\colon &[x_0:x_1:x_2:x_3;\ u_0:u_1] &\mapsto &\ \ [-\overline{x_0}:\overline{x_2}:\overline{x_1}:-\overline{x_3};\ \overline{u_0}:\overline{u_1}]; &\\
&\mu_4:=\tau \circ \mu_1&\colon &[x_0:x_1:x_2:x_3;\ u_0:u_1] &\mapsto &\ \ [-\overline{x_0}:\overline{x_2}:\overline{x_1}:\overline{x_3};\ \overline{u_0}:\overline{u_1}].&
\end{align*}
These four real structures correspond to the four real forms listed in the statement (in the same order). Also, the maps 
\begin{equation}\label{eq: embeddings of A3 in Qg}
(x,y,z) \mapsto [x:x^2 - g(1,z)y^2:1:y;z:1]\ \ \text{and}\ \ 
(x,y,z) \mapsto [x:x^2 + g(1,z)y^2:1:y;z:1]
\end{equation}
yield open immersions of $\A_\R^3$ into $\QQ_{g,\R}$ and $\QQ'_{g,\R}$ respectively (see the proof of \cite[Lemma 4.4.3]{BFT22}), and so $\QQ_{g,\R}$ and $\QQ'_{g,\R}$ are both rational. 

It remains to compute the identity component of the automorphism groups of these four real threefolds. For $\QQ_{g,\R}$ and $\QQ'_{g,\R}$, the same formula as in Proposition \ref{prop: aut group of Qg} yields a faithful $\PGL_{2,\R}$-action. For $\TT_{g,\R}$ and $\TT'_{g,\R}$,  a faithful $\SO_{3,\R}$-action is obtained by letting $\SO_{3,\R}$ act naturally on $(x_0,x_1,x_2) \in \R^3$ and trivially on the other coordinates. 
If $g$ has exactly two roots, we furthermore have a faithful $\G_{m,\R}$-action on the basis of the fibration.
The result follows then from Proposition \ref{prop: aut group of Qg} since $\PGL_{2,\R}$ and $\SO_{3,\R}$ are the two real forms of $\PGL_{2,\C}$ in the category of real algebraic groups.
\end{proof}

\begin{remark}
Depending on the choice of $g \in \R[u_0,u_1]$, it is possible for the real forms $\QQ_{g,\R}$ and $\QQ'{g,\R}$, as well as for the real forms $\TT_{g,\R}$ and $\TT'_{g,\R}$, to be isomorphic. This occurs when the corresponding cohomology classes in $\H^1(\Gamma,\Aut(\QQ_{g,\C})_{\P^1})$ are mapped to the same cohomology class in $\H^1(\Gamma,\Aut(\QQ_{g,\C}))$, which is the case precisely when $n$ is odd and $F \neq A_l$ with $l$ odd (see the proof of Theorem \ref{th: real forms of Qg} for details).
\end{remark}

We now focus on determining $\H^1(\Gamma,F)$ when $F=\mathrm{Im}(\Aut(\QQ_{g,\C}) \to \PGL_{2,\C})$ is a finite subgroup of $\PGL_{2,\C}$. We first recall the classification of the finite subgroups of $\PGL_{2,\C}$.

\begin{lemma} \label{lem:finite subgroups of SL2}\emph{(Finite subgroups of $\PGL_{2,\C}$; see \cite{Klein93})}.
If $F$ is a finite subgroup of $\PGL_{2,\C}$, then it is conjugate to one of the following subgroups:
\begin{itemize}[leftmargin=4mm]
\item $(A_l$, $l \geq 1)$ The cyclic group of cardinal $l$, generated by $\omega_{2l}={\scriptsize\begin{bmatrix}
\zeta_{2l} & 0 \\ 0 & \zeta_{2l}^{-1}
\end{bmatrix}}$, where $\zeta_{2l}$ is a primitive $2l$-th root of unity.
\item $(D_l$, $l \geq 2)$ The dihedral group of cardinal $2l$, generated by $A_l$ and $f={\scriptsize\begin{bmatrix}
0 & i \\ i & 0
\end{bmatrix}}$.
\item $(E_6)$ The tetrahedral group of cardinal $12$, generated by $D_2$ and $\alpha= {\scriptsize\begin{bmatrix}
1-i & 1-i \\ -1-i & 1+i
\end{bmatrix}}$.
\item $(E_7)$ The octahedral group, of cardinal $24$, generated by $E_6$ and $A_4$.
\item $(E_8)$ The icosahedral group, of cardinal $60$, generated by $A_5$, $h={\scriptsize\begin{bmatrix}
0 & 1 \\ -1 & 0
\end{bmatrix}}$, and $\beta= {\scriptsize \begin{bmatrix}
\zeta_5+\zeta_5^{-1} & 1 \\ 1 & -\zeta_5-\zeta_5^{-1}
\end{bmatrix}}$.
\end{itemize}
\end{lemma}

\newpage

\begin{remark}\label{rk: F is Gamma-stable}\item
\begin{itemize}
\item Let $F:=\mathrm{Im}(\Aut(\QQ_{g,\C}) \to \PGL_{2,\C})$. Then, performing a real-linear change of coordinates with $u_0$ and $u_1$ if necessary, we may and will always assume that $F$ is one of the finite subgroups of $\PGL_{2,\C}$ listed in Lemma \ref{lem:finite subgroups of SL2}.
\item Let $\Gamma$ act on $\PGL_{2}(\C)$ coefficientwise by complex conjugation.
Then each finite subgroup of $\PGL_{2,\C}$ listed in Lemma~\ref{lem:finite subgroups of SL2} is $\Gamma$-stable.
\end{itemize} 
\end{remark}

\begin{lemma}\label{lem: computation of H1(F)}
Let $F$ be one of the finite subgroups of $\PGL_{2,\C}$ listed in Lemma~\ref{lem:finite subgroups of SL2}, and let 
$\Gamma$ act on $F$ coefficientwise by complex conjugation.
Then an exhaustive list of class representatives of the finite set $\H^1(\Gamma,F)$ is given in the table below.
\smallskip

\[
\begin{array}{|l|l|}
\hline 
\text{Subgroup } F \subseteq \PGL_{2,\C} & \H^1(\Gamma,F)\\
\hline 
\hline 
A_l,\ l\ge 1,\ l \text{ odd}  &[I_2]\\
  \hline
 A_l,\ l\ge 2,\ l \text{ even} &[I_2],\ \ [\omega_{2l}] \\
   \hline
D_l,\ l\ge 3,\ l \text{ odd}  &[I_2],\ \ [f]\\
  \hline
 D_l,\ l\ge 2,\ l \text{ even} &[I_2],\ \ [\omega_{2l}],\ \ [f], \ \ [h] \\
   \hline
        E_6      &  [I_2], \ \ [h]  \\
          \hline
              E_7   &  [I_2],\ \  [\omega_8],\ \  [h]  \\
                \hline
                    E_8     &  [I_2], \ \ [h]  \\
                    \hline
\end{array} 
\]
\end{lemma}

\begin{proof}
Recall that, if $\Gamma \simeq \Z/2\Z$ and $A$ is a $\Gamma$-group, then
\[ 
\H^1(\Gamma,A)=Z^1(\Gamma,A)/\sim, \text{\ where \ } Z^1(\Gamma,A)=\{ a \in A \ | \   a^{-1}= \gamma \cdot a \}
\] 
and $a_1$, $a _2 \in Z^1(\Gamma,A)$ satisfy $a_1 \sim a_2$ if $a_2=b^{-1} a_1 (\gamma \cdot  b)$ for some $b \in A$. 
Determining each set $\H^1(\Gamma,F)$ is then done by an explicit calculation on a case-by-case basis.
\end{proof}

\begin{remark}\label{rk: no real points over h}
With the notation of Lemma \ref{lem: computation of H1(F)}, the class $[h]$ in $\H^1(\Gamma,F)$ corresponds to the equivalence class of the real structure $[u_0:u_1] \mapsto [\overline{u_1}: -\overline{u_0}]$ on $\P_\C^1$; the latter corresponds to the real form of $\P_\C^1$ with no real points. All the other classes given in the table of Lemma \ref{lem: computation of H1(F)} correspond to the trivial real form $\P_\R^1$. 
\end{remark}

\begin{theorem}\label{th: real forms of Qg}
Let $g \in \R[u_0,u_1]$ be a homogeneous polynomial with at least three distinct roots, which is not a square and is of degree $2n$ for some $n \in \N_{\geq 2}$.
Assume that $F=\mathrm{Im}(\Aut(\QQ_{g,\C}) \to \PGL_{2,\C})$ is one of the finite subgroups listed in Lemma \ref{lem:finite subgroups of SL2}. Then the number of real forms of $\QQ_{g,\C}$, up to isomorphism, is given in the table in Theorem \ref{th: third cases Qg, over R, dim 3} \ref{QQgR:3}.
\end{theorem}

\begin{proof}
Let us start with the case $n$ even ($n \geq 2$).
According to Proposition \ref{prop: aut group of Qg}, $\Aut(\QQ_g) \simeq \Aut(\QQ_g)_{\P^1} \times F$ as $\Gamma$-groups, where $\Gamma$ acts on $\Aut(\QQ_g)$ and on $\Aut(\QQ_g)_{\P^1}$  by $\mu_1$-conjugation (with $\mu_1\colon \QQ_{g,\C} \to \QQ_{g,\C}$ the canonical real structure defined by \eqref{eq: canonical real structure on Qg}), and coefficientwise by complex conjugation on $F$ (see Remark \ref{rk: F is Gamma-stable}).
Hence $\H^1(\Gamma,\Aut(\QQ_{g,\C})) \simeq \H^1(\Gamma,\Aut(\QQ_{g,\C})_{\P^1}) \times \H^1(\Gamma,F)$, and so the real forms of $\QQ_{g,\C}$ are parametrized by the pairs of cohomology classes $([*]_1,[*]_2) \in \H^1(\Gamma,\Aut(\QQ_{g,\C})_{\P^1}) \times \H^1(\Gamma,F)$. 
The total number of real forms of $\QQ_{g,\C}$, up to isomorphism, follows then 
from Proposition \ref{propo: first four real forms of Qg} and Lemma \ref{lem: computation of H1(F)}. Moreover, if $[*]_2=[h]_2$, then the corresponding real forms of $\QQ_{g,\C}$ have no real points (see Remark \ref{rk: no real points over h}). 
Furthermore, denoting \[
r=\left\{
    \begin{array}{ll}
        1 & \text{ if } F \in \{A_l \text{ with $l$ odd}, E_6,E_8 \} \\
         2 & \text{ if } F \in \{A_l \text{ with $l$ even}, D_l \text{ with $l$ odd}, E_7 \} \\
        3 & \text{ if } F=D_l \text{ with $l$ even} 
    \end{array}
\right.,
\]
the $4r$ real forms of $\QQ_{g,\C}$ for which $[*]_2 \neq [h]_2$ can be described as the hypersurfaces
\[
    \begin{array}{lcl}
       W_i \colon \ x_0^2-x_1x_2-g_i(u_0,u_1)x_3^2 &  = &0 \\
       X_i \colon  \ x_0^2-x_1x_2+g_i(u_0,u_1)x_3^2 &  = &0 \\
        Y_i\colon  \ x_0^2+x_1^2+x_2^2-g_i(u_0,u_1)x_3^2 &  = &0 \\
         Z_i \colon \   x_0^2+x_1^2+x_2^2+g_i(u_0,u_1)x_3^2 &  = &0 
    \end{array}
\]
in $\P(\O_{\P^1}^{\oplus 3}\oplus \O_{\P^1}(n))$, 
where $i \in \{1,\ldots,r\}$ and the $g_i \in \R[u_0,u_1]$ are some homogeneous polynomials of degree $2n$ that can be written explicitly using the classical invariant theory (see Section \ref{sec: gi with invariant theory} for details). 
Let us note that each $W_i$ and $X_i$ contains a dense open subset isomorphic to $\A_\R^3$ (consider the maps defined by \eqref{eq: embeddings of A3 in Qg}), and so they are rational. 
We are then able to fill in the first three columns of the table in the statement of Theorem \ref{th: third cases Qg, over R, dim 3} (corresponding to the case $n$ even).

\smallskip

Let us now consider the case $n$ odd ($n \geq 3$). 
If $F=A_l$ with $l$ odd, then we still have that $\Aut(\QQ_{g,\C}) \simeq \Aut(\QQ_{g,\C})_{\P^1} \times F$ as $\Gamma$-groups; indeed, the natural surjection $\SL_{2,\C} \to \PGL_{2,\C}$ restricted to the subgroup generated by $\omega_l$ induces an isomorphism when $l$ is odd, and so the $F$-action on $\P_\C^1$ lifts to an $F$-action on $\QQ_{g,\C}$ that commutes with the $\Aut(\QQ_{g,\C})_{\P^1}$-action. Arguing as in the case $n$ even yields therefore the existence of four non-isomorphic real forms for $\QQ_{g,\C}$, two of which are rational.
We now assume that $F\neq A_l$ with $l$ odd.
Then $\Aut(\QQ_g)$ is no more isomorphic to the direct product $\Aut(\QQ_g)_{\P^1} \times F$, but we still have an exact sequence in Galois cohomology 
\[
\H^1(\Gamma,\Aut(\QQ_{g,\C})_{\P^1}) \to \H^1(\Gamma,\Aut(\QQ_{g,\C})) \to \H^1(\Gamma,F).
\] 
The real structures on $\QQ_{g,\C}$ defined by
\begin{align*}
&\mu_5 \ \colon [x_0:x_1:x_2:x_3;u_0:u_1] \ \ \mapsto \ \ [\overline{x_0}:\overline{x_1}:\overline{x_2}:\overline{x_3}\ ;\ \zeta_{2l} \overline{u_0}:\zeta_{2l}^{-1}\overline{u_1}];&\\
&\mu_6\ \colon [x_0:x_1:x_2:x_3;u_0:u_1] \ \ \mapsto \ \ [\overline{x_0}:\overline{x_1}:\overline{x_2}:\overline{x_3}\ ;\ i\overline{u_1}:i\overline{u_0}]; \text{ \  and}\\
&\mu_7 \ \colon [x_0:x_1:x_2:x_3;u_0:u_1] \ \ \mapsto \ \ [\overline{x_0}:\overline{x_1}:\overline{x_2}:\overline{x_3}\ ;\ \overline{u_1}:-\overline{u_0}]&
\end{align*}
correspond to cohomology classes of $\H^1(\Gamma,\Aut(\QQ_{g,\C}))$ that are respectively mapped to the cohomology classes $[\omega_{2l}]$, $[f]$, and $[h]$ in $\H^1(\Gamma,F)$.
Moreover, each $\mu_j$ with $j \in \{5,6,7\}$ stabilizes $\Aut(\QQ_{g,\C})_{\P^1}$ and satisfies $\mu_j \circ \varphi \circ \mu_j=\mu_1 \circ \varphi \circ \mu_1$, for every $ \varphi \in \Aut(\QQ_{g,\C})_{\P^1}$. 
It then follows from Proposition \ref{propo: first four real forms of Qg} that, for each $\Gamma$-action, $\H^1(\Gamma,\Aut(\QQ_{g,\C})_{\P^1})$ has four elements (corresponding to hypersurfaces $W_i$, $X_i$, $Y_i$, and $Z_i$ as above when $j \in \{1,5,6\}$), and it remains to determine in each case which elements from $\H^1(\Gamma,\Aut(\QQ_{g,\C})_{\P^1})$ are mapped to the same elements in $\H^1(\Gamma,\Aut(\QQ_{g,\C}))$ to conclude concerning the number of real forms of $\QQ_{g,\C}$ (up to isomorphism).

Denote $\delta:=\frac{1}{5}(\zeta_{10}^{3} + 3 \zeta_{10}^{2} - 2 \zeta_{10} + 1)$ and 
$\epsilon:=\frac{1}{5}(3\zeta_{10}^{3} - \zeta_{10}^{2} +4 \zeta_{10} -2)$.
We consider the automorphism $\psi \in \Aut(\QQ_{g,\C})$ defined by
\[
\psi\colon [x_0:x_1:x_2:x_3;u_0:u_1] \ \ \mapsto \ \
\begin{cases}
    \begin{array}{l}
        [x_0:x_1:x_2:x_3;i u_0:-i u_1] \text{ if } F\in \{A_l \text{ with $l \geq 2$ even},E_6,E_7\}; \\
        [x_0:x_1:x_2:x_3;i u_1:i u_0] \text{ if } F=D_l \text{ with }l \geq 2; \\
        [x_0:x_1:x_2:x_3;\delta u_0+ \epsilon u_1:\epsilon u_0-\delta u_1] \text{ if } F=E_8. \\
    \end{array}
\end{cases}
\]
Then $\psi$ satisfies $\psi \circ \mu_1 \circ \psi^{-1}=\mu_2$ and $\psi \circ \mu_3 \circ \psi^{-1}=\mu_4$. Hence, $\mu_1, \mu_2 \in \H^1(\Gamma,\Aut(\QQ_{g,\C})_{\P^1})$ resp. $\mu_3, \mu_4 \in \H^1(\Gamma,\Aut(\QQ_{g,\C})_{\P^1})$, are mapped to the same element in $\H^1(\Gamma,\Aut(\QQ_{g,\C}))$. 
We conclude that the fiber of the map $\H^1(\Gamma,\Aut(\QQ_{g,\C})) \to \H^1(\Gamma,F)$ over $[I_2]$ contains exactly two elements. The same holds over $[\omega_{2l}]$ (replacing $\mu_1$ by $\mu_5$ and taking $\psi([x_0:x_1:x_2:x_3;u_0:u_1])=[x_0:x_1:x_2:x_3;i u_0:-i u_1])$ and over $[f]$ (replacing $\mu_1$ by $\mu_6$ and taking $\psi([x_0:x_1:x_2:x_3;u_0:u_1])=[x_0:x_1:x_2:x_3;i u_1:i u_0])$. 
However, the same does not hold over $[h]$; indeed, the equality $\sigma \circ \mu_7 =\psi \circ \mu_7 \circ \psi^{-1}$ (see the proof of Proposition \ref{propo: first four real forms of Qg} for the definition of the regular involution $\sigma$) is equivalent to the existence of an element $M \in \widetilde F$, where $\widetilde F$ is the inverse image of $F$ via the natural cover $\SL_{2,\C} \to \PGL_{2,\C}$, such that 
$\begin{bmatrix}
0&1\\-1&0
\end{bmatrix} \overline{M}=-M \begin{bmatrix}
0&1\\-1&0
\end{bmatrix}$, and this last equation has no solution when $F \in \{D_l \text{ with $l \geq 2$ even}, E_6, E_7,E_8\}$.

We therefore obtain the following results regarding the cardinality of the fiber of the map $\H^1(\Gamma,\Aut(\QQ_{g,\C})) \to \H^1(\Gamma,F)$ over a given cohomology class $[*]$ in $\H^1(\Gamma,F)$:
\smallskip
\[
    \begin{array}{|l|c|c|c|c|}
        \hline
       \text{Subgroup } F \subseteq \PGL_{2,\C}  & [I_2] &  [\omega_{2l}] & [f] & [h] \\
        \hline
        A_l,\ l\ge 2,\ l \text{ even}   & 2 &  2 & \times & \times \\
        \hline
        D_l,\ l\ge 3,\ l \text{ odd}   &  2 & \times & 2 & \times \\
        \hline
        D_l,\ l\ge 2,\ l \text{ even}   & 2 & 2 & 2 & 4 \\
        \hline
        E_6 & 2 & \times & \times &4 \\
        \hline
        E_7 &  2 & 2 & \times & 4 \\
        \hline
        E_8 & 2 & \times & \times & 4 \\
        \hline
\end{array}
\]
\smallskip

\noindent Arguing as in the case $n$ even, we observe that the real forms of $\QQ_{g,\C}$ over the class $[h] \in \H^1(\Gamma,F)$ have no real points, and that half of the real forms over $[*] \neq [h]$, the ones given by $W_i \simeq X_i$ (since $\mu_1$ is equivalent to $\mu_2$ as explained above), contain a dense open subset isomorphic to $\A_\R^3$. We are then able to fill in the last three columns of the table in the statement of Theorem \ref{th: third cases Qg, over R, dim 3} (corresponding to the case $n$ odd).
\end{proof}

\begin{remark}\label{rk: about the real forms of Qg}\item
\begin{itemize}
\item With the notation of Proposition \ref{propo: first four real forms of Qg}, we always have 
\[W_{1} \simeq \QQ_{g,\R},\ X_{1} \simeq \QQ'_{g,\R},\ Y_{1} \simeq \TT_{g,\R},\ \text{ and}\  Z_{1} \simeq \TT'_{g,\R}.\]
\item As shown in the proof of Theorem \ref{th: real forms of Qg}, if $n$ is odd and $F \neq A_l$ with $l$ odd, then we have $W_i \simeq X_i$ and $Y_i \simeq Z_i$. Otherwise, the four real forms $W_i$, $X_i$, $Y_i$, and $Z_i$ are pairwise non-isomorphic.
Indeed, we precisely show that the real structures $\mu_1$ and $\mu_2$ resp. $\mu_3$ and $\mu_4$ are mapped onto the same element in $\H^1(\Gamma,\Aut(\QQ_{g,\C}))$, which means exactly that $W_1 \simeq X_1$ et $Y_1 \simeq Z_1$. The same holds for the other $W_i, X_i, Y_i, Z_i$.
\item Arguing as in the proof of Proposition \ref{propo: first four real forms of Qg}, we check that if $X \in \{ W_i,X_i\}$, then $\Autz(X) \simeq \PGL_{2,\R}$, and if $X \in \{Y_i,Z_i\}$, then $\Autz(X) \simeq \SO_{3,\R}$. 
\item As mentioned earlier, each $W_i$ and $X_i$ contains a dense open subset isomorphic to $\A_\R^3$
(consider the maps defined by \eqref{eq: embeddings of A3 in Qg}), and so they are rational.
However we do not know in general for which $g \in \R[u_0,u_1]$ the threefolds $Y_i$ and $Z_i$ are rational.
\end{itemize}
\end{remark}

\begin{remark}
In some cases, one can easily determine the rationality or irrationality of a real form $X$ of $\QQ_{g,\C}$. For instance, if $X(\R)$ is empty or disconnected, then $X$ is obviously irrational.

Another criterion, communicated to us by Lena Ji and Isabel Vogt, involves the following result: \cite[Satz 22]{Witt} states that if $X \to \mathbb{P}^1_\mathbb{R}$ is a real quadric fibration of positive relative dimension with smooth generic fibre, and if the map induced on real points is surjective, then the quadric fibration has a section defined over $\mathbb{R}$. Consequently, the generic fibre is rational, and hence $X$ is rational.
For example, consider the fibration 
\[\pi_{(u_0^2 + u_1^2)^a} \colon \mathcal{T}_{(u_0^2 + u_1^2)^a,\mathbb{R}} = (x_0^2 + x_1^2 + x_2^2 - (u_0^2 + u_1^2)^a x_3^2 = 0) \to \mathbb{P}^1_\mathbb{R}\ \ \text{for $a \geq 1$.}\] 
This fibration maps the real points of $\mathcal{T}_{(u_0^2 + u_1^2)^a,\mathbb{R}}$ surjectively onto $\mathbb{P}^1_\mathbb{R}(\mathbb{R})$. By the aforementioned criterion, $\mathcal{T}_{(u_0^2 + u_1^2)^a,\mathbb{R}}$ is therefore rational.
\end{remark}

\subsection{Description of the real forms of $\QQ_{g,\C}$ through invariant theory}\label{sec: gi with invariant theory}
In this section we resort to the classical invariant theory of the finite subgroups of $\SL_{2,\C}$ (see Remark \ref{rk: every F occurs}) to give explicit equations for the real forms of $\QQ_{g,\C}$ corresponding to cohomology classes in $\H^1(\Gamma,\Aut(\QQ_{g,\C}))$ whose image in $\H^1(\Gamma,F)$ is not the cohomology class $[h]$ via the exact sequence in Galois cohomology that appears in Proposition \ref{propo: first four real forms of Qg}. 
Recall that the real forms of $\QQ_{g,\C}$ corresponding to cohomology classes in $\H^1(\Gamma,\Aut(\QQ_{g,\C}))$ that are mapped to $[h]$ in $\H^1(\Gamma,F)$ have no real points; they are therefore not of immediate interest to us since we are primarily interested in the rational real forms of $\QQ_{g,\C}$.
Generators of the invariant algebra for each finite subgroup of $\SL_{2,\C}$ acting naturally on $\C[u_0,u_1]$ can be found for instance in \cite[Section  1.3]{Dol}.

\smallskip

Let $g=g_1 \in \R[u_0,u_1]$ be a homogeneous polynomial, which is not a square and is of degree $2n$ for some $n \in \N_{\geq 2}$.
Let $F=\mathrm{Im}(\Aut(\QQ_{g,\C}) \to \PGL_{2,\C})$ be one of the finite subgroups listed in Lemma \ref{lem:finite subgroups of SL2}. Let \[
r=\left\{
    \begin{array}{ll}
        1 & \text{ if } F \in \{A_l \text{ with $l$ odd}, E_6,E_8 \} \\
         2 & \text{ if } F \in \{A_l \text{ with $l$ even}, D_l \text{ with $l$ odd}, E_7 \} \\
        3 & \text{ if } F=D_l \text{ with $l$ even} 
    \end{array}
\right..
\]
Then (see proof of Theorem~\ref{th: real forms of Qg}) the following equations define $4r$ real forms of $\QQ_{g,\C}$, seen as real hypersurfaces in $\P(\O_{\P^1}^{\oplus 3}\oplus \O_{\P^1}(n))$, corresponding to cohomology classes in $\H^1(\Gamma,\Aut(\QQ_{g,\C}))$ whose image in $\H^1(\Gamma,F)$ is not the cohomology class $[h]$:
\[
    \begin{array}{lcl}
       W_i \colon \ x_0^2-x_1x_2-g_i(u_0,u_1)x_3^2 &  = &0 \\
       X_i \colon  \ x_0^2-x_1x_2+g_i(u_0,u_1)x_3^2 &  = &0 \\
        Y_i\colon  \ x_0^2+x_1^2+x_2^2-g_i(u_0,u_1)x_3^2 &  = &0 \\
         Z_i \colon \   x_0^2+x_1^2+x_2^2+g_i(u_0,u_1)x_3^2 &  = &0 
    \end{array},
\]
where $i \in \{1,\ldots,r\}$ and the $g_i$ are defined as follows:  
\begin{itemize}
\item If $F=A_l$, with $l \geq 1$, then there exists $p_1 \in \R[T_1, T_2,T_3]$ such that 
\[
g_1=p_1(f_1,f_2,f_3) \text{ with } \left\{
    \begin{array}{l}
       f_1(u_0,u_1) =u_{0}^{2l} \\
        f_2(u_0,u_1) =u_0 u_1 \\
        f_3(u_0,u_1) =u_{1}^{2l} 
    \end{array}
\right.
\]
and, when $l$ is even, we denote $g_2=p_1(-f_1,f_2,-f_3)$.
Indeed, consider the complex isomorphism defined by
\[
\varphi\colon W_{2,\C} \to \QQ_{g,\C},\ [x_0:x_1:x_2:x_3; u_0:u_1] \mapsto [x_0:x_1:x_2:x_3;\zeta_{4l} u_0: \zeta_{4l}^{-1} u_1].\] Then $\varphi \circ \mu_1 \circ \varphi^{-1}=\mu_5$, which means precisely that $\mu_5$ is the real structure associated to the real form $W_{2,\R}$ of $\QQ_{g,\C}$ defined above. The verification is similar in the other cases ($X_2,Y_2,Z_2$).

\smallskip

\item If $F=D_l$, with $l \geq 3$ odd, then there exists $p_2 \in \R[T_1, T_2,T_3]$ such that 
\[
g_1=p_2(f_1,f_2,f_3) \text{ with } \left\{
    \begin{array}{l}
       f_1(u_0,u_1) =u_0^2 u_1^2 \\
        f_2(u_0,u_1) =u_0^{2l}-u_1^{2l} \\
        f_3(u_0,u_1) =u_0 u_1(u_0^l-u_1^l)^2 
    \end{array}
\right.
\]
and we denote $g_2=p_2(\tilde{f}_1,\tilde{f}_2,\tilde{f}_3)$ with 
\[
 \left\{
    \begin{array}{l}
       \tilde{f}_1(u_0,u_1) 
       =-(u_0^2+u_1^2)^2 \\
        \tilde{f}_2(u_0,u_1) 
        =2\sum_{k=0}^{l} (-1)^k\binom{2l}{2k}u_0^{2(l-k)}u_1^{2k} \\
        \tilde{f}_3(u_0,u_1) 
        = -2i^{l+1}(u_0^2+u_1^2)^{l+1}-2(u_0^2+u_1^2)\sum_{k=0}^{l-1}(-1)^k\binom{2l}{2k+1}u_1^{2k+1}u_0^{2(l-k)-1}
    \end{array}
\right.
\]

\smallskip

\item If $F=D_l$, with $l \geq 2$ even, then there exists $p_3 \in \R[T_1, T_2,T_3]$ such that 
\[
g_1=p_3(f_1,f_2,f_3) \text{ with } \left\{
    \begin{array}{l}
       f_1(u_0,u_1) =u_0^2 u_1^2 \\
        f_2(u_0,u_1) =(u_0^l-u_1^l)^2 \\
        f_3(u_0,u_1) =u_0 u_1(u_0^{2l}-u_1^{2l})
    \end{array}
\right.
\]
and we denote $g_2=p_3(f_1,-(u_0^l+u_1^l)^2,-f_3)$ and
$g_3=p_3(\tilde{f}_1,\tilde{f}_2,\tilde{f}_3)$ with 
\[
 \left\{
    \begin{array}{l}
       \tilde{f}_1(u_0,u_1) 
       =-(u_0^2+u_1^2)^2 \\
        \tilde{f}_2(u_0,u_1) 
        =-2i^{l}(u_0^2+u_1^2)^{l}+2\sum_{k=0}^{l}(-1)^k\binom{2l}{2k}u_1^{2k}u_0^{2(l-k)} \\
        \tilde{f}_3(u_0,u_1) 
        =-2(u_0^2+u_1^2)\sum_{k=0}^{l-1} (-1)^k\binom{2l}{2k+1}u_0^{2(l-k)-1}u_1^{2k+1}
    \end{array}
\right.
.\]

\smallskip

\item If $F=E_6$, then there exists $p_4 \in \R[T_1, T_2,T_3]$ such that 
\[
g_1=p_4(f_1,f_2,f_3) \text{ with } \left\{
    \begin{array}{l}
       f_1(u_0,u_1) =u_0^5 u_1 - u_0 u_1^5 \\
        f_2(u_0,u_1) = u_0^8 + 14 u_0^4 u_1^4 + u_1^8 \\
        f_3(u_0,u_1) =u_0^{12} - 33 u_0^8 u_1^4 - 33 u_0^4 u_1^8 + u_1^{12}
    \end{array}
\right.
.\]

\smallskip

 \item If $F=E_7$, then there exists $p_5 \in \R[T_1, T_2,T_3]$ such that 
\[
g_1=p_5(f_1,f_2,f_3) \text{ with } \left\{
    \begin{array}{l}
       f_1(u_0,u_1) =u_0^8 + 14 u_0^4 u_1^4 + u_1^8 \\
        f_2(u_0,u_1) = u_0^{10} u_1 ^2 - 2 u_0 ^6 u_1^6 + u_0^2 u_1^{10} \\
        f_3(u_0,u_1) = u_0^{17} u_1 - 34 u_0^{13} u_1^5 + 34 u_0^5 u_1^{13} - u_0 u_1^{17}
    \end{array}
\right.
,\]
and we denote 
\[
g_2=p_5(\tilde{f}_1,\tilde{f}_2,\tilde{f}_3) \text{ with }
 \left\{
    \begin{array}{l}
       \tilde{f}_1(u_0,u_1) 
       =-u_0^8 + 14 u_0^4 u_1^4 - u_1^8 \\
        \tilde{f}_2(u_0,u_1) 
        =-u_0^{10} u_1 ^2 - 2 u_0 ^6 u_1^6 - u_0^2 u_1^{10} \\
        \tilde{f}_3(u_0,u_1) 
        = u_0^{17} u_1 + 34 u_0^{13} u_1^5 - 34 u_0^5 u_1^{13} - u_0 u_1^{17}
    \end{array}
\right.
.\]

\smallskip

 \item If $F=E_8$, then there exists $p_6 \in \R[T_1, T_2,T_3]$ such that 
\[
g_1=p_6(f_1,f_2,f_3) \text{ with } \left\{
    \begin{array}{l}
       f_1(u_0,u_1) =u_0^{11} u_1 + 11 u_0^6 u_1^6 - u_0 u_1^{11} \\
        f_2(u_0,u_1) = u_0^{20} - 228 u_0^{15} u_1^5 + 494 u_0^{10} u_1^{10} + 228 u_0^5 u_1^{15} + u_1^{20} \\
        f_3(u_0,u_1) =u_0^{30} + 522 u_0^{25} u_1^5 - 10005 u_0^{20} u_1^{10} - 10005 u_0^{10} u_1^{20} - 522 u_0^5 u_1^{25} + u_1^{30}
    \end{array}
\right.
.\]
\end{itemize}

\subsection{The case when $g$ has exactly two distinct roots}\label{ss:g has two roots}
We now suppose that $g \in \R[u_0,u_1]$ is of degree $\deg(g)=2n\geq2$ and has exactly two distinct roots. 
Up to a complex-linear change of coordinates, we can assume that $g = u_0^a u_1^b$ for some odd $a, b \geq 1$ (because $g$ is not a square).

First, suppose that the multiplicities of the two roots of $g$ are distinct.
By Proposition~\ref{prop: aut group of Qg}, we have a short exact sequence
\[
1\to\Aut(\QQ_{g,\C})_{\p^1_{\C}}\to\Aut(\QQ_{g,\C})\to\G_{m,\C}\to 1
\]
which induces an exact sequence in Galois cohomology
\[
\H^1(\Gamma,\Aut(\QQ_{g,\C})_{\p^1_{\C}})\stackrel{\Psi}\to \H^1(\Gamma,\Aut(\QQ_{g,\C}))\to \H^1(\Gamma,\G_{m,\C})=\{*\} \ \text{(by Hilbert's theorem 90)}
\]
Then $\Psi$ is a surjection. By Proposition~\ref{propo: first four real forms of Qg}, there are at most four real forms for $\QQ_{g,\C}$, namely
\[
    \begin{array}{lcl}
       \QQ_{g,\R} \colon \ x_0^2-x_1x_2-u_0^au_1^bx_3^2 &  = &0 \\
       \QQ'_{g,\R} \colon  \ x_0^2-x_1x_2+u_0^au_1^bx_3^2 &  = &0 \\
        \TT_{g,\R} \colon  \ x_0^2+x_1^2+x_2^2-u_0^au_1^bx_3^2 &  = &0 \\
         \TT'_{g,\R} \colon \   x_0^2+x_1^2+x_2^2+u_0^au_1^bx_3^2 &  = &0 
    \end{array}
\]
with $a+b=2n$, $a,b \geq 1$ odd, corresponding to the four real structures $\mu_1$, $\mu_2$, $\mu_3$ and $\mu_4$ respectively. But $\Psi$ maps $\mu_1$ and $\mu_2$ resp. $\mu_3$ and $\mu_4$, seen as elements of $\H^1(\Gamma,\Aut(\QQ_{g,\C})_{\p^1_{\C}})$, to the same  element in  
$\H^1(\Gamma,\Aut(\QQ_{g,\C}))$. Indeed, the automorphism of $\QQ_{g,\C}$ defined by  $\varphi \colon [x_0:x_1:x_2:x_3;u_0:u_1] \mapsto  [x_0:x_1:x_2:i^{-a} x_3;-u_0:u_1]$ (see \cite[Example 4.4.6(1)]{BFT22}) satisfies $\varphi \circ \mu_1 \circ \varphi^{-1}=\mu_2$ and $\varphi \circ \mu_3 \circ \varphi^{-1}=\mu_4$. 
Therefore, $\QQ_{g,\C}$ admits exactly two non-isomorphic real forms which are $\QQ_{g,\R}$ (always rational) and $\TT_{g,\R}$.

\smallskip

Now let us consider the case where the multiplicities of the two roots of $g$ are the same, i.e.~$a=b\geq 1$.
By Proposition~\ref{prop: aut group of Qg}, we have a short exact sequence
\[
1\to\Aut(\QQ_{g,\C})_{\p^1_{\C}}\to\Aut(\QQ_{g,\C})\to\G_{m,\C} \rtimes\Z/2\Z\to 1,
\]
where   $\Z/2\Z$ is the subgroup of $\Aut(\QQ_{g,\C})$ generated by the involution $\upsilon\colon [x_0:x_1:x_2:x_3;u_0:u_1] \mapsto [x_0:x_1:x_2:x_3;u_1:u_0]$. 
It induces an exact sequences in Galois cohomology
\[
\H^1(\Gamma,\Aut(\QQ_{g,\C})_{\p^1_{\C}})\stackrel{\Psi}\to \H^1(\Gamma,\Aut(\QQ_{g,\C}))\stackrel{\Phi}\to \H^1(\Gamma,\G_{m,\C} \rtimes\Z/2\Z) \simeq \Z/2\Z.
\]

If $\Gamma$ acts on $\Aut(\QQ_{g,\C})$ via $\mu_1$-conjugation, then Proposition~\ref{propo: first four real forms of Qg} yields that $\H^1(\Gamma,\Aut(\QQ_{g,\C})_{\p^1_{\C}})$ parametrizes the equivalence classes of the four real structures $\mu_1$, $\mu_2$, $\mu_3$ and $\mu_4$. 
Moreover, as in the previous case ($a \neq b$), the automorphism $\varphi \in \Aut(\QQ_{g,\C})$ satisfies
$\varphi \circ \mu_1 \circ \varphi^{-1}=\mu_2$ and $\varphi \circ \mu_3 \circ \varphi^{-1}=\mu_4$.
Therefore, the fiber $\Phi^{-1}(\Phi([\mu_1]))$ parametrizes the two non-isomorphic real forms of $\QQ_{g,\C}$ defined as the hypersurfaces in $\P(\O_{\P^1}^{\oplus 3}\oplus \O_{\P^1}(n))$ given by the equations
\begin{itemize}
\item  $\QQ_{g,\R}\ \colon\  x_0^2-x_1x_2-(u_0u_1)^a x_3^2=0$ (always rational); and
\item $ \TT_{g,\R} \ \colon \ x_0^2+x_1^2+x_2^2-(u_0u_1)^a  x_3^2=0$.
\end{itemize}

On the other hand, writing $\mu_8:=\upsilon \circ \mu_1$, we see that $\mu_8$ is a real structure on $\QQ_{g,\C}$ such that $\Phi([\mu_8]) \neq \Phi([\mu_1])$. 
Letting $\Gamma$ act on $\Aut(\QQ_{g,\C})$ via $\mu_8$-conjugation, we see that $\H^1(\Gamma,\Aut(\QQ_{g,\C})_{\p^1_{\C}})$ parametrizes the equivalence classes of the four real structures $\mu_8$, $\mu_9:=\sigma \circ \mu_8$, $\mu_{10}:=\tau \circ \sigma \circ\mu_8$ and $\mu_{11}:=\tau \circ \mu_8$ (see the proof of Proposition~\ref{propo: first four real forms of Qg} for details). 
These four real structures, seen as elements of $\H^1(\Gamma,\Aut(\QQ_{g,\C})_{\p^1_{\C}})$, are mapped via $\Psi$ to four different elements of $\H^1(\Gamma,\Aut(\QQ_{g,\C}))$. 
Therefore, the fiber $\Phi^{-1}(\Phi([\mu_8]))$ parametrizes the four non-isomorphic real forms of $\QQ_{g,\C}$ defined as the hypersurfaces in $\P(\O_{\P^1}^{\oplus 3}\oplus \O_{\P^1}(n))$ given by the equations
\begin{itemize}
\item $ x_0^2-x_1x_2-(u_0^2+u_1^2)^a x_3^2=0$ (always rational);
\item $x_0^2-x_1x_2+(u_0^2+u_1^2)^ax_3^2=0$ (always rational); 
\item $ x_0^2+x_1^2+x_2^2-(u_0^2+u_1^2)^a x_3^2=0$; and
\item $ x_0^2+x_1^2+x_2^2+(u_0^2+u_1^2)^a x_3^2=0$ (no real points).
\end{itemize}
By Proposition~\ref{propo: first four real forms of Qg}, the identity components of the automorphism groups of the first two real forms are isomorphic to $\PGL_{2,\R}\times\mathbb{G}_{m,\R}$ while the identity components  of the automorphism groups of the last two real forms are isomorphic to $\SO_{3,\R}\times\mathbb{G}_{m,\R}$.

\smallskip

The following proposition summarizes our study.

\begin{proposition}\label{prop: real forms of Qg with two roots}
Let $g \in \R[u_0,u_1]$ be a homogeneous polynomial with exactly two distinct roots, which is not a square and is of degree $2n$ for some $n \in \N_{\geq 1}$. We use the notation of Proposition \ref{propo: first four real forms of Qg}.
\begin{enumerate}
\item If $F \simeq \G_{m,\C}$, then $\QQ_{g,\C}$ admits exactly two non-isomorphic real forms which are $\QQ_{g,\R}$ (trivial real form, always rational) and $\TT_{g,\R}$.
\item If $F\simeq \mathbb{G}_m\rtimes\Z/2\Z$, then $\QQ_{g,\C}$ admits exactly two non-isomorphic real forms which are
$\QQ_{g,\R}$ (trivial real form, always rational), $\QQ'_{g,\R}$ (always rational), $\TT_{g,\R}$ and $\TT'_{g,\R}$.
\end{enumerate}
\end{proposition}

\section{Real forms of rank 1 Fano threefolds}\label{sec: Case of rank 1 Fano threefolds}
In this section we determine all the real forms of the Fano threefolds \hyperlink{th:D_j}{(j)}-\hyperlink{th:D_l}{(l)} of Theorem \ref{th: list of MFS corresponding to max alg subgroups of Cr3} and \hyperlink{th:D_Fano_mbis}{(n)} of Remark \ref{rk: case of Y5 and X12}. 

\smallskip

\begin{definition}\label{def: Qrs}
Let $r,s \in \N$. We denote by $Q^{r,s}$ the real quadric hypersurface of equation $(x_1^2+\cdots+x_r^2-x_{r+1}^2-\cdots -x_{r+s}^2) \subseteq \P_{\R}^{r+s-1}$. 
\end{definition}

\begin{proposition}\label{prop: real forms of Q3}
Let $X=Q_3$ be the complex smooth quadric in $\P_\C^4$. Then $X$ has three non-isomorphic real forms: $Q^{3,2}$ and $Q^{4,1}$, which are both rational and with automorphism groups the indefinite special orthogonal groups $\SO(3,2)$ and $\SO(4,1)$ respectively, and $Q^{5,0}$, which has no real points. 
\end{proposition}

\begin{proof}
The automorphism group of $Q_3$ is $\SO_{5,\C}$, hence the set of real forms (up to isomorphism) of $Q_3$ is in bijection with $\H^1(\Gamma,\SO_{5,\C})$. 
But, since $\Aut_{\mathrm{gr}}(\SO_{5,\C}) \simeq \SO_{5,\C}$, the set $\H^1(\Gamma,\SO_{5,\C})$ is also in bijection with the set of real forms (up to isomorphism) of $\SO_{5,\C}$ in the category of algebraic groups. There are three such real forms (see e.g.~\cite[Section V\!I.10]{Kna02}), and so $Q_3$ has three non-isomorphic real forms. On the other hand, $Q^{3,2}$, $Q^{4,1}$ and $Q^{5,0}$ are three non-isomorphic real forms of $Q_3$, so there is no other real form (up to isomorphism). Finally, only $Q^{3,2}$ and $Q^{4,1}$ have real points, and the result follows then from the fact that a smooth quadric over a field $\k$ is rational if and only if it has a $\k$-point (see e.g.~\cite[Proposition 2.6]{KP23a}). 
\end{proof}

\begin{proposition} \label{prop: real forms of WPS}
The trivial real form $\P(1,1,1,2)_\R$ resp. $\P(1,1,2,3)_\R$ is the unique real form (up to isomorphism) of the complex weighted projective space $\P(1,1,1,2)_\C$ resp. $\P(1,1,2,3)_\C$.
\end{proposition}

\begin{proof}
The description of the automorphism groups of weighted projective spaces given in \cite[Section 8]{AA89} implies that $\Aut(\P(1,1,1,2))$ and $\Aut(\P(1,1,2,3))$ are both connected (over any field).
Then \cite[Proposition 6.2.1]{BFT22} implies that $\Aut(\P(1,1,1,2)_\C) \simeq \Aut(\PP_{2,\C})$ as $\Gamma$-groups (endowed with the canonical $\Gamma$-action on both sides). 
Hence, $\H^1(\Gamma,\Aut(\P(1,1,1,2)_\C)) \simeq \H^1(\Gamma,\Aut(\PP_{2,\C}))$, and the result follows from Remark \ref{rk: real forms of Pb}.
Likewise, \cite[Lemma 6.3.4]{BFT22} implies that $\Aut(\RR_{(3,1),\C}) \simeq \Aut(\P(1,1,2,3)_\C)$ as $\Gamma$-groups. Hence, $\H^1(\Gamma,\Aut(\RR_{(3,1),\C})) \simeq \H^1(\Gamma,\Aut(\P(1,1,2,3)_\C))$, and the result follows from Remark \ref{rk: real forms of Rmn}.
\end{proof}

\begin{proposition}\label{prop: real forms of Y5 and X12}
Let $X$ be one of the two complex Fano threefolds $Y_{5,\C}$ or $X_{12,\C}^{\mathrm{MU}}$ of Remark \ref{rk: case of Y5 and X12}.
Then $X$ has exactly two non-isomorphic real forms and both are rational with automorphism groups isomorphic to $\PGL_{2,\R}$ $($for the trivial real form $)$ and $\SO_{3,\R}$ $($for the nontrivial real form$)$.
We denote by $\widetilde{Y}_{5,\R}$ the nontrivial real form of $Y_{5,\C}$ and by $\widetilde{X}_{12,\R}^{\mathrm{MU}}$ the nontrivial real form of $X_{12,\C}^{\mathrm{MU}}$.
\end{proposition}

\begin{proof}
We have $\Aut(X) \simeq \PGL_{2,\C}$ (see e.g.~\cite[Theorem 1.1.2]{KPS18}) from which we deduce that $\H^1(\Gamma,\Aut(X))$ is a set with two elements, i.e.~that $X$ has exactly two non-isomorphic real forms.
By \cite[Theorem 1.1]{KP23a}, both real forms of $Y_{5,\C}$ are rational and a real form of $X_{12,\C}^{\mathrm{MU}}$ is rational if and only if it has a real point. 
On the other hand, \cite[Proposition 5.13, Corollary 5.14 and Appendix C]{MJT20} implies that the two non-isomorphic real forms of $X_{12,\C}^{\mathrm{MU}}$ have real points, so they are both rational.
The statement concerning the automorphism groups of the two non-isomorphic real forms of $Y_{5,\C}$ and $X_{12,\C}^{\mathrm{MU}}$ also follows from \textit{loc.~cit.}
\end{proof}

\begin{remark}
The $\k$-forms of the quintic del Pezzo $\bk$-threefold $Y_{5,\bk}$ are studied in \cite[Section 2]{DK19} when $\k$ is an arbitrary field of characteristic zero.
\end{remark}

\section{Description of the equivariant Sarkisov links}\label{sec: equiva Sark links in dim 3}
In this section, we determine all the equivariant Sarkisov links starting from one of the nontrivial rational real Mori fiber spaces listed in Theorems \ref{th: second cases, over R, dim 3} and \ref{th: third cases Qg, over R, dim 3}.
We will proceed case by case.

\subsection{Sarkisov links from the Fano threefolds}\label{ss:links Fano}
We show that there are no $\Autz(X)$-equivariant Sariksov links starting from a Fano threefold $X$ as in Theorem \ref{th: second cases, over R, dim 3}~\ref{item 0 th2}-\ref{item 2 th2}-\ref{item 4 th2}-\ref{item 6 th2}.

\begin{lemma}\label{lem:link homogeneous}
Let $X$ be a real projective variety such that $\Autz(X_\C)$ acts transitively on $X_\C$.
Then any $\Autz(X)$-equivariant birational map starting from $X$ is an isomorphism.
In particular, there are no $\Aut(X)$-equivariant Sarkisov links starting from $X \in \{\widetilde{\SS}_{1,\R}, Q^{3,2}, Q^{4,1} \}$, there is one $\Autz(X)$-equivariant Sarkisov link of type IV starting from $X=\fR_{\C/\R}(\P_\C^1) \times \P_\R^1$ $($induced by the projections on both factors$)$, and there are two $\Autz(X)$-equivariant Sarkisov links of type IV starting from $X=\P_\R^1 \times \P_\R^1 \times \P_\R^1$ $($induced by the three natural projections on $\P_\R^1 \times \P_\R^1)$.
\end{lemma}

\begin{proof}
Let $\varphi\colon X \dashrightarrow X'$ be an $\Autz(X)$-equivariant birational map starting from $X$. Then the birational map $\varphi_\C \colon X_\C \dashrightarrow X_\C'$ (which is both $\Gamma$-equivariant, for the natural $\Gamma$-actions on $X_\C$ and $X_\C'$, and $\Autz(X_\C)$-equivariant) is an isomorphism. Indeed, since $X_\C$ is $\Autz(X_\C)$-homogeneous, the birational map $\varphi_\C$ has no base-locus and does not contract any subvariety; it is therefore an isomorphism between $X_\C$ and $\varphi_\C(X_\C)$, and since $X$ is projective we must have $\varphi_\C(X_\C)=X'_\C$.
Hence, $X \simeq X'$ and $\varphi$ is an $\Autz(X)$-equivariant isomorphism. 
The second part of the lemma follows from the fact that $\widetilde{\SS}_{1,\R}$, $Q^{3,2}$, $Q^{4,1}$ are Fano threefolds of Picard rank $1$, and by considering the cone of effective curves for $\fR_{\C/\R}(\P_\C^1) \times \P_\R^1$ (two contractions) and for $\P_\R^1 \times \P_\R^1 \times \P_\R^1$ (three contractions).
\end{proof}

\begin{proposition}\label{prop:Y5-X12-no-links}
Assume that the base field $\k$ is a characteristic zero field.
Let $X$ be a $\k$-form of one the Fano $\bk$-threefolds $Y_{5,\bk}$ or $X_{12,\bk}^{\rm{MU}}$.
Then there are no $\Aut(X)$-equivariant Sarkisov links starting from $X$.
\end{proposition}

\begin{proof}
Recall that $\Aut(Y_{5,\k})\simeq\PGL_{2,\k}\simeq\Aut(X_{12,\k}^{\rm{MU}})$ are connected. 
We will prove that there are no $\Aut(X_{\bk})$-equivariant Sarkisov links starting from $X_{\bk}$ when $\k=\bk$. The claim will then follow from Lemma~\ref{lem:no links over bk} and the fact that there are no (equivariant) Sarkisov links of type IV starting from a Fano variety. 

So, suppose that $\k=\bk$.
First, notice that by the description of Sarkisov links (see Figure~\ref{fig:Sarkisov links} in Section  \ref{sec: Mori fibrations and Sarkisov links}), any $\Aut(X)$-equivariant Sarkisov link $\chi\colon X\rat X'$ starting from the Fano threefold $X$ is a link of type I or II (because $\rho(X)=1$). This means it is of the following form
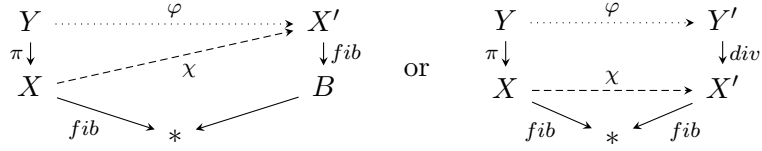
\begin{figure}[h!]
\[
{
\def\arraystretch{2.2}
\begin{array}{cc}
\begin{tikzcd}[ampersand replacement=\&,column sep=1.3cm,row sep=0.16cm]
Y\ar[dd,"\pi",swap]\ar[rr,dotted,"\varphi"] \&\& X' \ar[dd,"fib"] \\ \\
X\ar[uurr,"\chi",dashed,swap]  \ar[dr,"fib",swap] \&  \& B \ar[dl] \\
\& \ast \&
\end{tikzcd}
& \ \text{or}\ \ \ \ \ 
\begin{tikzcd}[ampersand replacement=\&,column sep=.8cm,row sep=0.16cm]
 Y\ar[dd,"\pi",swap]\ar[rr,dotted,"\varphi"]\& \& Y'\ar[dd,"div"]\\ \\
X \ar[rr,"\chi",dashed] \ar[dr,"fib",swap] \&  \& X' \ar[dl,"fib"] \\
\&  \ast \&
\end{tikzcd}
\end{array}
}
\]\caption{The possible forms of $\chi$.}\label{fig:links for Fanos}
\end{figure}

\noindent where $\pi$ and $div$ are divisorial contractions, $\ast$ is a point, $fib$ is a Mori fibration and $\varphi$ is an isomorphism in co-dimension $1$.

Let us recall the construction of $X$ given in \cite{MU83}. 
Denote by $R_n$ the vector space of homogeneous polynomials in $\k[x,y]$ of degree $n$. The group $\SL_2=\mathrm{SL}_{2,\k}$ acts naturally on $R_n$ and hence on $\p(R_n)$. Then $X$ is defined as the closure of the orbit $\SL_{2} \cdot f_n$ for some $f_n\in R_n$ with $n=6$ if $X=Y_5$ and $n=12$ if $X=X_{12}^{\mathrm{MU}}$ \cite[Lemma 3.3]{MU83}. The threefold $X$ contains the surface $S:=\overline{\SL_{2}\cdot x^{n-1}y}=\SL_{2} \cdot x^{n-1}y\sqcup \SL_{2}\cdot x^n$, which is singular along the rational curve $\ell:=\SL_{2} \cdot x^n$ \cite[Lemma 1.6]{MU83}. 
Furthermore, $X=\SL_{2}\cdot f_n\sqcup S$  \cite[Lemma 1.5]{MU83}. 
Moreover, $S$ is the image of an $\SL_{2}$-equivariant morphism $\phi\colon \p^1_{\k}\times\p^1_{\k}\to S$ given by a linear system of bidegree $(n-1,1)$ on $\p^1_{\k}\times\p^1_{\k}$, where $\SL_{2}$ acts diagonally on $\p^1_{\k}\times\p^1_{\k}$ \cite[Lemma 1.6]{MU83}. 
In particular, $\ell \simeq \P_{\k}^1$ is the image of the diagonal in $\p^1_{\k}\times\p^1_{\k}$. 

Notice that the $\Aut(X)$-orbits of $X$ are $X \setminus S$ and $S\setminus \ell$ and $\ell$, of dimension $3$, $2$ and $1$ respectively, and so each of them is preserved by the $\Gamma$-action on $X$ (corresponding to the $\k$-form $X$) by Remark~\ref{rk:Gamma respects aut-action}. 
In particular, the link $\chi$ has indeterminacy locus $\ell$. 
Since $\k=\bk$ and $X$ and $\ell$ are smooth, $\pi$ is a blow-up along $\ell$ on an open subset of $X$ \cite[Lemma 2.2.8]{BFT22}. Since $\ell$ is an $\Aut(X)$-orbit and $\pi$ is $\Aut(X)$-equivariant, $\pi$ is the blow-up along $\ell$.
It follows that $Y$ is smooth and $\rho(Y)=2$. 
In particular, the cone of effective curves has dimension two and hence $\NS(Y)$ has to extremal rays; one induces the contraction $\pi$. Suppose that the other extremal ray induces a contraction $\eta\colon Y\to Z$. If $\eta$ is a fibration, then $\chi$ is of type I,  
but according to \cite[Theorem E]{BFT22} there is no $\Autz(Y_{\bk})$-equivariant birational map from $Y$ to a conic fibration or to a del Pezzo fibration. Hence $\chi$ is of type II, and $\eta$ is a divisorial contraction or a small contraction. 
By \cite[Theorem 3.3]{Mor82} (see also \cite[Theorem 1.2.5]{Kol91}), and since $Y$ is smooth, the exceptional locus of $\eta$ is a divisor. In particular, $\eta$ is a divisorial contraction. By Remark~\ref{rmk:sarkisov-2-rays-game}, $\varphi$ is an isomorphism and so we can set $Y'=Y$, $Z=X'$, and $\eta=div\circ\varphi$ in the right-hand side of Figure~\ref{fig:links for Fanos}.

Denote by $\tilde S$ the strict transform of $S$ in $Y$, and by $E$ the exceptional divisor of $\pi$. These are the only two $\Autz(Y)$-invariant divisors in $Y$, and so $\eta$ must contract $\tilde S$. 
Consider the curve $f:=\phi(\{p\}\times\p^1_{\k})\subset S$. 
Then the morphism $\phi_{|\{p\}\times\p^1_{\k}}\colon \{p\}\times\p^1_{\k}\to f$ is of degree $1$ and $\mathcal O_{\p(R_n)(1)}\cdot f=1$ \cite[Corollary 1.7]{MU83}. 
Moreover, the proof of \cite[Lemma 3.4]{MU83} implies that $\tilde S$ is a smooth ruled surface over $\eta(\tilde \ell) \simeq \P_{\k}^{1}$ with general rational fiber the $\tilde f$ which are the inverse images of $f$ by the morphism $\tilde S \to S$.

The embedding $X_{\bk}\hookrightarrow\p(R_n)$ is given by a very ample divisor $D$ and $D\cdot f=1$. 
 Since $\Pic(X)\simeq\Z$ \cite[Corollary 1.9]{MU83} and since $-K_{X}$ is ample, we can write $a(-K_X)=D$ for some $a\in\Q_{>0}$. 
 Then $1=D\cdot f=a(-K_X)\cdot f$, which implies that $\frac{1}{-a}=K_X\cdot f\in\Z_{<0}$ (since $X_{\bk}$ is Fano by \cite[Lemma 3.3.]{MU83}). It follows that $a=1$ and thus $-K_X\cdot f=1$. Also, $E\cdot\tilde f=2$ \cite[Lemma 3.4]{MU83}. 
We have $K_Y=\pi^*K_X+E$, since $X$ and $\ell$ are smooth, and thus
\[
K_Y\cdot \tilde f=(\pi^*K_X+E)\cdot\tilde f=K_X\cdot f+E\cdot\tilde f=-1+2=1>0.
\]
This contradicts the assumption that $\eta$ contracts $\tilde S$, i.e. that $[\tilde f]$ is a $K_Y$-negative extremal ray in $\NE(Y)$. 
Hence, there are no $\Aut(X)$-equivariant Sarkisov links starting from $X$ when $\k=\bk$, which concludes the proof.
\end{proof}

\subsection{Sarkisov links from $\mathcal{G}_b$ and $\mathcal{H}_b$}
In this section we work over $\k=\R$. 
We list the $\Autz(X)$-equivariant Sarkisov links starting from $X$ when $X \simeq \mathcal G_b$ or $\mathcal H_b$, which are the nontrivial rational real forms of $\FF_{0,\C}^{b,c}$ that appear in Theorem \ref{th: second cases, over R, dim 3}~\ref{item 1 th2}.
If $b=0$, then $\mathcal G_0\simeq \mathcal H_0\simeq\p^1_\R\times\fR_{\C/\R}(\p_\R^1)$, which is already treated in Lemma \ref{lem:link homogeneous}, and so we will assume that $b \geq 1$.

\begin{lemma}\label{lem:links G_b}\item
\begin{enumerate}
\item\label{links G_b:2} If $b\geq2$, then there are no $\Autz(\mathcal G_b)$-equivariant Sarkisov links starting from $\mathcal G_b$.
\item\label{links G_b:1} If $b=1$, the only $\Aut(\mathcal G_1)$-equivariant link starting from $\mathcal G_1$ is a divisorial contraction $\psi \colon \mathcal G_1\to  Z$, where $ Z$ is isomorphic to the singular quadric hypersurface $Q^{1,3}$ in $\P_\R^4$.  
\item \label{links G_b:3} The only Sarkisov link starting from $Z \simeq Q^{1,3}$ is the blow-up of the singular point $\psi\colon \mathcal G_1\to Z$. Moreover, $\psi  \Aut(\mathcal G_1) \psi^{-1} =\Aut(Q^{1,3})$.
\end{enumerate} 
\end{lemma}

\begin{proof}
Recall that $\mathcal{G}_b$ is the real form of $\FF_{0,\C}^{b,-b}$ corresponding to the real structure
\[\theta\colon \ \FF_{0,\C}^{b,-b} \to \FF_{0,\C}^{b,-b},\  \left[x_0:x_1 ;y_0:y_1;z_0:z_1\right]\mapsto \left[\overline{x_0}: \overline{x_1};\overline{z_0}:\overline{z_1};\overline{y_0}:\overline{y_1}\right].
\]
With the notation of Lemma \ref{lem: XabPic}, the cone $\NE(\mathcal G_b)$ of effective curves on $\mathcal G_b$ is generated by the two extremal rays $[\ell_1+\ell_2]$ and $[\ell_3]$, and we have $K_{\mathcal{G}_b} \cdot [\ell_1+\ell_2]=2(b-2)$ and $K_{\mathcal{G}_b} \cdot [\ell_3]=-2$.

We now assume that $b \geq 1$.
By \cite[Section 3.1]{BFT23}, the group $\Autz(\FF_{0,\C}^{b,-b})$ acts on $\FF_{0,\C}^{b,-b}$ with two orbits, namely the surface $H_0:=(x_0=0)$ and the open orbit $\FF_{0,\C}^{b,-b} \setminus H_0$. 
We deduce that $\mathcal{G}_b$ is the union of two $\Autz(\mathcal{G}_b)$-orbits, namely the divisor $H_0$ and its complement. 
It follows that the only possible equivariant Sarkisov links starting from $\mathcal G_b$ are links of type III and IV.
But if $b \geq 2$, we have seen that we can only contract the ray $[\ell_3]$ which yields the $\P^1$-fibration $\mathcal G_b\to\fR_{\C/\R}(\p^1_\C)$. Hence, for $b \geq 2$, there are no $\Autz(\mathcal G_b)$-equivariant Sarkisov links starting from $\mathcal G_b$. This proves \ref{links G_b:2}.

It remains to study the case $b=1$. The contraction of the extremal ray $[\ell_1+\ell_2]$ corresponds to the $\Autz(\mathcal G_1)$-equivariant birational morphism $\psi\colon \mathcal G_1 \to Z$ given by
\[
\mathcal G_{1,\C} \simeq \FF_{0,\C}^{1,-1} \to Z_\C \subset \P_\C^4, \ [x_0:x_1;y_0:y_1;z_0:z_1] \mapsto [x_0 y_0z_0:x_0 y_0z_1:x_0 y_1z_0:x_0 y_1z_1:x_1],
\]
where the real structure on $\P_\C^4$ is defined by $\mu([a:b:c:d:e])=[\overline{a}:\overline{c}:\overline{b}:\overline{d}:\overline{e}]$ 
and $Z_\C$ is the singular quadric hypersurface of equation $ad-bc=0$ in $\P_\C^4$. The pair $(Z_\C, \mu_{|Z_\C})$ corresponds to the real quadric hypersurface in $\P_\R^4$ defined as the zero-locus of the real quadratic form $4ad-b^2-c^2$ of signature $(1,3)$. This proves \ref{links G_b:1}.

Finally, since the blow-up of the singular point of $Z \simeq Q^{1,3}$ is $\Aut(Q^{1,3})$-equivariant, we obtain that $\psi  \Aut(\mathcal G_1) \psi^{-1} =\Aut(Q^{1,3})$. The fact that $\psi$ is the unique Sarkisov link starting from $Q^{1,3}$ is a consequence of the fact that $Q_\C^{1,3}$ is the union of two $\Aut(Q_\C^{1,3})$-orbits: the singular point and its open complement. 
This proves \ref{links G_b:3} and concludes the proof of the lemma.
\end{proof}

\begin{remark}\label{rk: case of Q13}
Let us note that the real Fano threefold $Z \simeq Q^{1,3}$ is $\Q$-factorial while the complex Fano threefold $Z_\C$ is not (indeed, $Q^{1,3}=(w^2-x^2-y^2-z^2=0)$ and for instance $(w-x=y+iz=0)$ is not a $\Q$-Cartier divisor, see also \cite[Definition 3.1]{Kol99})
and that is why $Z_\C$ does not appear in the list of complex Mori fiber spaces given by Theorem \ref{th: list of MFS corresponding to max alg subgroups of Cr3}.
\end{remark}

\begin{lemma}\label{lem:links H_b}
Let $b \geq 2$. There are no $\Autz(\mathcal{H}_b)$-equivariant Sarkisov links starting from $\mathcal{H}_b$.
\end{lemma}

\begin{proof}
Recall that $\mathcal{H}_b$ is the real form of $\FF_{0,\C}^{b,b}$ corresponding to the real structure
\[\theta\colon \ \FF_{0,\C}^{b,b} \to \FF_{0,\C}^{b,b},\  \left[x_0:x_1 ;y_0:y_1;z_0:z_1\right]\mapsto \left[\overline{x_1}: \overline{x_0};\overline{z_0}:\overline{z_1};\overline{y_0}:\overline{y_1}\right].
\]
By \cite[Section 3.1]{BFT23}, the group $\Autz(\FF_{0,\C}^{b,b})$ acts on $\FF_{0,\C}^{b,b}$ with three orbits, namely the surfaces $H_0:=(x_0=0)$ and $H_1=(x_1=0)$, and the open orbit $\FF_{0,\C}^{b,b} \setminus(H_0\cup H_1)$. Since $\theta$ swaps $H_0$ and $H_1$, we deduce that $\mathcal{H}_b$ is the union of two $\Autz(\mathcal{H}_b)$-orbits, namely the divisor $H_0+H_1$ and its complement. 
It follows that the only possible equivariant Sarkisov links starting from $\mathcal H_b$ are links of type III and IV. 
With the notation of Lemma \ref{lem: XabPic}, the cone $\NE(\mathcal H_b)$ of effective curves on $\mathcal H_b$ is generated by the two extremal rays $[\ell_2+\ell_4]$ and $[\ell_3]$.
Since $K_{\mathcal H_b} \cdot (\ell_2+\ell_4)=2(b-2) \geq 0$ (because $b \geq 2$), we can only contract the ray $[\ell_3]$ which yields the $\P^1$-fibration $\mathcal H_b\to\fR_{\C/\R}(\p^1_\C)$.
\end{proof}

\begin{remark}\label{rk: cas of H1}
The morphism 
\[
\FF_{0,\C}^{1,1} \to \P_{\C}^{3},\  [x_0:x_1;y_0:y_1;z_0:z_1] \mapsto [x_0y_0:x_0y_1:x_1z_0:x_1 z_1]
\] 
induces an $\Autz(\mathcal{H}_1)$-equivariant birational morphism $\delta$ from $\mathcal{H}_1$ to $\P_\R^3$. It follows that $\Autz(\mathcal{H}_1)$ is not a maximal connected algebraic subgroup of $\Bir(\P_\R^3)$ since $\delta \Autz(\mathcal{H}_1) \delta^{-1} \subsetneq \Aut(\P_\R^3)=\PGL_{4,\R}$ (indeed, by Remark \ref{rk: auto group of Gb and Hb}, we have $\dim(\Autz(\mathcal{H}_1))=7$).
\end{remark}

 \subsection{Sarkisov links from $\tilde{\mathcal S}_{b,\R}$. }
In this section we work over $\k=\R$ and consider the $\Autz(\tilde{\mathcal S}_{b,\R})$-equivariant Sarkisov links starting from $\tilde{\mathcal S}_{b,\R}$.
We recall that an explicit description of the real threefold $\tilde{\mathcal S}_{b,\R}$ is given at the end of Section \ref{sec: Sb P1-bundles}; in particular, it is shown that $\tilde{\mathcal S}_{b,\R}$ is rational if and only if $b$ is odd, in which case the natural conic fibration $\pi\colon\tilde{\mathcal S}_{b,\R} \to \P_\R^2$ is a $\P^1$-bundle.

\begin{proposition}\label{prop:links Sb}
Let $b\geq3$ be an odd integer. 
There is a birational involution $\varphi\colon \tilde{\mathcal S}_{b,\R} \dashrightarrow \tilde{\mathcal S}_{b,\R}$, which is a type II equivariant link such that $\varphi \Autz(\tilde{\mathcal S}_{b,\R}) \varphi^{-1}=\Autz(\tilde{\mathcal S}_{b,\R})$. Moreover, $\varphi$ is the unique $($up to automorphism$)$ equivariant Sarkisov 
link starting from $\tilde{\mathcal S}_{b,\R}$.
\end{proposition}

\begin{proof}
This is proved similarly to \cite[Proposition 6.2.2]{BFT22}. 
We work mostly over $\C$ and keep track at each step of the $\Gamma$-action.
Consider the Cartesian square
\[\xymatrix@R=4mm@C=2cm{
    \hat{\mathcal S}_{b,\C} \ar[r]^{\varepsilon} \ar[d]_q  &  \tilde{\mathcal S}_{b,\C} \ar[d]^\pi  \\
   \p^1_\C\times\p^1_\C \ar[r]^{\kappa} & \p^2_\C 
  }\] 
where $\kappa$ is the double cover defined in Section \ref{sec: Family (d)}.
We recall from \cite[Lemma 4.2.4 (2)]{BFT23} that  $\tilde{\mathcal S}_{b,\R}$ contains a unique curve invariant by $\Autz(\tilde{\mathcal S}_{b,\R})$, which is isomorphic to the real form of $\P_\C^1$ with no real points and that we denote by $\gamma_0$.
Let $D:=\pi^{-1}(C_{0,\C})$. For $i=1,2$, let $pr_i\colon\p^1_\C\times\p^1_\C\to\p^1_\C$ be the $i$-th projection. Then $pr_i\circ q\colon \hat{\mathcal S}_{b,\C}\to\p^1$ is a $\F_{b,\C}$-bundle and we set $S_i\subset \hat{\mathcal S}_{b,\C}$ to be the union of the $(-b)$-sections. The intersection $T=S_1\cap S_2$ is a curve isomorphic to the diagonal $\Delta$ of $\p^1_\C\times\p^1_\C$. 
Set $E:=\varepsilon(S_1)=\varepsilon(S_2)$. According to \cite[Lemma 5.3.9(1)]{BFT22}, the group $\Autz(\tilde S_{b,\C}) \simeq \PGL_{2,\C}$ acts on $\tilde S_{b,\C}$ with four orbits, namely the curve $\gamma_{0,\C}=\varepsilon(T)\subset D$, the two surfaces $D \setminus \gamma_{0,\C}$ and $E\setminus \gamma_{0,\C}$, and $\tilde S_{b,\C}\setminus(D \cup E)$. 
Moreover, by \cite[Lemma 5.3.9(3)\&(6)]{BFT22}, the cone of effective curves of $\tilde{\mathcal S}_{b,\C}$ is generated by the class of a general fiber $f$ of $\pi$ and by the class of a curve $s_1\subset E$ such that $K \cdot s_1=b-3\geq0$, and the two extremal rays $[f]$ and $[s_1]$ are $\Gamma$-invariant.

It follows that an $\Autz(\tilde{\mathcal S}_{b,\R})$-equivariant link $\chi\colon \tilde{\mathcal S}_{b,\R}\rat Y$ is not of type III or IV. So, it is of type I or II, and hence it starts with the inverse of a birational contraction $\eta\colon Z\to \tilde{\mathcal S}_{b,\R}$. Since $\tilde{\mathcal S}_{b,\R}$ has no invariant points, $\eta$ contracts a surface $H$ onto the curve $\gamma_0$. Since $\rho(Z/\tilde{\mathcal S}_{b,\R})=1$, the surface $H$ is irreducible. So, $H_{\C}$ has either one or two irreducible components that are exchanged by the $\Gamma$-action on $\tilde{\mathcal S}_{b,\C}$. 

Suppose that $H$ is geometrically irreducible. Then $\eta_\C$ is the blow-up of $\gamma_{0,\C}$ (see \cite[Lemma 2.2.8]{BFT22}). 
We can then contract $\Autz(\tilde{\mathcal S}_{b,\C})$-equivariantly the strict transform of $D$ and obtain a birational involution $ \tilde{\mathcal S}_{b,\C}\rat \tilde{\mathcal S}_{b,\C}$, which is a link of type II, as explained in \cite[Lemma 5.6.2]{BFT23}. 
Since both $\gamma_{0,\C}$ and $D$ are $\Gamma$-invariant, this defines an $\Autz(\tilde{\mathcal S}_{b,\R})$-equivariant link $ \tilde{\mathcal S}_{b,\R}\rat \tilde{\mathcal S}_{b,\R}$ of type II.

Suppose now that $H_{\C}$ has two irreducible components $H_1$ and $H_2$. 
Let $\tilde D\subset Z$ be the strict transform of the surface $D\subset\tilde S_b$.
 Since $\tilde D_{\C}$ is $\Gamma$-stable and $\Gamma$ switches $H_1,H_2$, both components $H_1,H_2$ have to intersect $\tilde D_{\C}$. 
By Lemma~\ref{lem:contractions}, we can write $\eta_{\C}=\eta_2\circ\eta_1$, where $\eta_1\colon Z_\C\to U$ contracts $H_1$ and $\eta_2\colon U\to\tilde S_b$ contracts $\eta_1(H_2)$. They are both $\Autz(\tilde S_{b,\C})$-equivariant by Blanchard's lemma (see \cite[Proposition~4.2.1]{BSU13}). 
Since both $H_1$ and $\tilde D_\C$ intersect, the surfaces $\eta_1(H_2)$ and $\eta_1(\tilde D_{\C})$ must be tangent. 
However, $\eta_2$ is the blow-up of $\tilde S_b$ in $\gamma_{0,\C}$. Indeed, since $\gamma_{0,\C}$ and $\tilde S_{b,\C}$ are smooth, $\eta_2$ is the blow-up of $\gamma_{0,\C}$ on an open subset of $\tilde S_{b,\C}$ \cite[Lemma 2.2.8]{BFT22}. Since $\eta_2$ is $\Autz(\tilde S_{b,\C})$-equivariant, it is the blow-up of $\gamma_{0,\C}$. 
Then $\eta_1(\tilde D_\C)$ is the strict transform of $D_\C$ by $\eta_2$. We can compute that it is not tangent to the exceptional divisor $\eta_1(H_2)$ of $\eta_2$. 
\end{proof}

\subsection{Sarkisov links from quadric bundles over $\p^1_\R$}\label{sec: Sarkisov links for Qg}
In this section we work over the field $\k=\R$ and determine the $\Autz(X)$-equivariant Sarkisov links starting from $X$ when $X$ is a real form of the Umemura quadric bundle $\QQ_{g,\C} \to \P_\C^1$.

Let $g\in\R[u_0,u_1]$ be a homogeneous polynomial, which is not a square and is of degree $2n$ for some $n \in \N_{\geq 1}$. 
Recall that the real forms of $\pi_\C\colon \QQ_{g,\C} \to \P_\C^1$ were determined in Sections \ref{subsec: real forms Qg via Galois cohomology} and \ref{ss:g has two roots}. The real forms of $\pi_\C\colon \QQ_{g,\C} \to \P_\C^1$ considered here are exclusively those of the forms $\pi\colon U_g \to \P_\R^1$. (This limitation is not overly restrictive, as our primary focus lies on the rational real forms of $\QQ_{g,\C}$.)
When we talk about the roots of $g$, we always mean the \emph{complex} roots of $g$, even if $g\in\R[u_0,u_1]$.

\begin{proposition}[{see \cite[Lemma 6.6.2]{BFT22}} when $\k=\C$]\label{prop: first links from Qg}
Let $h\in\R[u_0,u_1]$ be a polynomial, irreducible over $\R$.
Let $\pi\colon U_g \to \P_\R^1$ be a real form of $\pi_\C\colon \QQ_{g,\C} \to \P_\C^1$.
Then the birational map defined by
\[
\psi_h\colon U_g\rat U_{gh^2},\ [x_0:x_1:x_2:x_3;u_0:u_1]\mapsto[hx_0:hx_1:hx_2:x_3;u_0:u_1]
\] 
is a Sarkisov link of type II.
Moreover, $\psi_h^{-1} \Autz(U_{gh^2}) \psi_h \subseteq \Autz(U_g)$ with equality if and only if either $g$ has at least three roots or $gh^2$ has exactly two roots.
If furthermore $g$ has at least three distinct roots, then there are no other equivariant Sarkisov links starting from $U_g$.
\end{proposition}

\begin{proof}
A description of the real forms of $\QQ_{g,\C}$ together with the identity components of their automorphism groups 
was obtained in Section \ref{sec: real forms of Umemura quadric fibrations}. 
Using this description, a case-by-case verification yields that $\psi_h^{-1}$ is $\Autz(U_{gh^2})$-equivariant, i.e.~that $\psi_h^{-1} \Autz(U_{gh^2}) \psi_h \subseteq \Autz(U_g)$ with equality if and only if either $g$ has at least three roots or $gh^2$ has exactly two roots (see also \cite[Lemma 6.6.2]{BFT22}). 
As $\psi_h$ furthermore decomposes as a  blow-up of some invariant curve followed by the contraction of an invariant divisor (see \cite[Lemma 6.6.2]{BFT22}), it is a Sarkisov link of type II. 

It remains to check that if $g$ has at least three distinct roots, then there are no other equivariant Sarkisov links starting from $U_g$.
Let us show first that there are not $\Autz(U_g)$-equivariant links of type III and IV starting from $U_g$. 
A link of type III or IV corresponds to the contraction of the two extremal rays of $U_g$ (see Remark~\ref{rmk:sarkisov-2-rays-game}). 
Denote respectively by $H,F \subset \QQ_{g,\C}$ the hypersurfaces $(x_3=0)$ and $(u_1=0)$.
Since $g$ is not a square, \cite[Lemma 4.4.3]{BFT22} gives that the cone of curves $\NE(\QQ_{g,\C})$ is generated by $[f]$ and $[h]$, where $f=F \cap H$ and $h=(x_0=x_1=x_3=0)$, and we have $K_{\QQ_{g,\C}} \cdot f=-2$ and $K_{\QQ_{g,\C}} \cdot h=n-2 \geq 0$ because $n\geq2$ and $g$ is not a square (see \cite[Lemma 4.4.3]{BFT22} and its proof). 
In particular, both extremal rays are $\Gamma$-stable, but only the one generated by $[f]$ yields a contraction. Therefore, there are no $\Autz(U_g)$-equivariant links of type III and IV starting from $U_g$. 
 
Before continuing, let us recall the orbit decomposition of $\QQ_{g,\C}$ for the action of $\Autz(\QQ_{g,\C})\simeq\PGL_{2,\C}$ (see \cite[Lemma 6.6.1]{BFT22}): 
\begin{itemize}
\item If $p\in\p^1_\C$ is a root of $g$, then the singular quadric $\pi^{-1}_{\C}(p)$ consists of three orbits: a fixed point $q=(x_0=x_1=x_2=0)\cap\pi^{-1}_{\C}(p)$, the curve $\Delta_p:=(x_3=0)\cap \pi^{-1}_{\C}(p)$, and the surface $\pi^{-1}_{\C}(p)\setminus(\Delta_p\cup\{q\})$. 
\item If $p\in\p^1_\C$ is not a root of $g$, then the smooth quadric $\pi^{-1}_{\C}(p)\simeq \p^1_\C\times\p^1_\C$ consists of two orbits, namely the diagonal curve $\Delta_p\subset\pi^{-1}_{\C}(p)$ and the surface $\pi_{\C}^{-1}(p)\setminus\Delta_p$. 
\end{itemize} 
 
\smallskip 
 
Suppose now that there is a link $\chi\colon U_g\rat X'$ of type I or II; see Figure~\ref{fig:links for Ug}.
\begin{figure}[h!]
\[
{
\def\arraystretch{2.2}
\begin{array}{cc}
\begin{tikzcd}[ampersand replacement=\&,column sep=1.3cm,row sep=0.16cm]
Y\ar[dd,"\eta",swap]\ar[rr,dotted,"\varphi"] \&\& X' \ar[dd,"fib"] \\ \\
U_g\ar[uurr,"\chi",dashed,swap]  \ar[dr,"fib",swap] \&  \& B \ar[dl] \\
\& \ast \&
\end{tikzcd}
& \ \text{or}\ \ \ \ \ 
\begin{tikzcd}[ampersand replacement=\&,column sep=.8cm,row sep=0.16cm]
 Y\ar[dd,"\eta",swap]\ar[rr,dotted,"\varphi"]\& \& Y'\ar[dd,"div"]\\ \\
U_g \ar[rr,"\chi",dashed] \ar[dr,"fib",swap] \&  \& X' \ar[dl,"fib"] \\
\&  \ast \&
\end{tikzcd}
\end{array}
}
\]\caption{The possible forms for a link $\chi$ of type I or II.}\label{fig:links for Ug}
\end{figure}
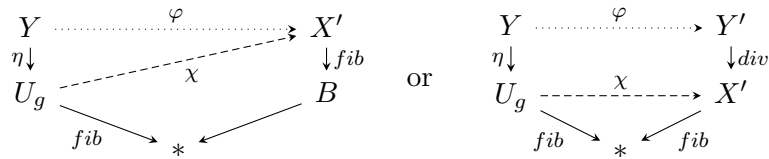 
Since $g$ has at least three distinct roots, the group $\Autz(U_g)$ acts trivially on $\p^1_\R$ and, by the description of the $\Autz(Q_{g,\C})$-orbits above combined with Remark~\ref{rk:Gamma respects aut-action}, the center of $\eta$ is contained in a fiber $F$ of the structure morphism $\pi\colon U_g\to\p^1_\R$.
Hence, we have the following possibilities for $\eta$:
\begin{itemize}
\item If $F_\C$ is irreducible, then the center of $\eta$ is $\Delta_{\pi(F)}$ if $F$ is smooth and it is $\Delta_{\pi(F)}$ or the fixed point $F_{sing}$ if $F$ is singular. 
We then obtain the link of type II from the statement with $\deg(h)=1$ analogously to \cite[Proposition 6.6.5]{BFT22}.
\item If $F_\C$ is reducible, we have $F_\C=F_1\cup F_2$ and $F_2=\overline{F_1}$. The center of $\eta$ has two geometrical components $C_i\subset F_i$, for $i \in \{1,2\}$, and $C_2=\overline{C_1}$. Then $C_i$ is $\Delta_{\pi_{\C}(F_i)}$ if $F$ is smooth  and it is $\Delta_{\pi_\C(F_i)}$ or $(F_i)_{sing}$ if $F$ is singular. We then obtain the link of type II from the statement with $\deg(h)=2$ analogously to \cite[Proposition 6.6.5]{BFT22}.
\end{itemize}
\end{proof}

\begin{corollary}\label{cor: exactly two roots}
Assume that $g$ has exactly two distinct roots. 
Let $\pi\colon U_g \to \P_\R^1$ be a real form of $\pi_\C\colon \QQ_{g,\C} \to \P_\C^1$.
There exists a birational map $\varphi\colon U_g \dashrightarrow Q_0$, where $Q_0 \subset \P_\R^4$ is a smooth real quadric hypersurface, such that $\varphi \Autz(U_g) \varphi^{-1} \subsetneq \Autz(Q_0)$.
\end{corollary}

\begin{proof}
Up to a complex-linear change of coordinates, we can assume that $g=u_0^au_1^b$ with $a,b \geq 1$ odd and $a+b=2n$.
By Proposition \ref{prop: first links from Qg}, there exists a birational map $\psi\colon U_g \dashrightarrow U_{u_0u_1}$ such that $\psi \Autz(U_g) \psi^{-1}=\Autz(U_{u_0u_1})$. 

We now assume that $g=u_0u_1$. According to \cite[Example 4.4.6~(2)]{BFT22},
the morphism $\QQ_{g,\C}\to \p_\C^{4}$ given by $[x_0:x_1:x_2:x_3;u_0:u_1]\mapsto [x_0:x_1:x_2:x_3u_{0}:x_{3}u_{1}]$ is the blow-up of the plane $P\subseteq \p_\C^{4}$ given by $P=\{[x_0:x_1:x_2:x_3:x_{4}]\in \p_\C^{4}\mid x_{3}=x_{4}=0\}$, so  $\QQ_{g,\C}$  is the blow-up of the smooth complex quadric hypersurface $Q=\{[x_0:x_1:x_2:x_3:x_{4}]\in \p_\C^{4}\mid x_0^2-x_1x_2-x_{3}x_{4}=0\}$ along the smooth conic $C=P\cap Q$. This
yields a birational map $\psi\colon \QQ_{g,\C} \dashrightarrow Q$ such that $\psi \Autz(\QQ_{g,\C}) \psi^{-1}=\Aut(Q,C):=\{f\in \Aut(Q)\mid f(C)=C\} \subsetneq \Autz(Q)$.

On the other hand, the real structures on $\QQ_{g,\C}$ when $g=u_0u_1$ were determined in Section \ref{ss:g has two roots}: up to equivalence, there are six real structures on $\QQ_{g,\C}$, namely $\mu_1$, $\mu_3$, $\mu_8$, $\mu_9$, $\mu_{10}$ and $\mu_{11}$. Letting $\Gamma$ act on $\QQ_{g,\C}$ via one of these six real structures, we see that there is in each case an induced $\Gamma$-action on $Q$ making the birational map $\psi\colon \QQ_{g,\C} \dashrightarrow Q$ $\Gamma$-equivariant. 
Therefore, there exists a real birational map $\psi_0 \colon U_g \dashrightarrow Q_0$, where $Q_0 \subset \P_\R^4$ is a smooth real quadric hypersurface, such that $\psi_0 \Autz(U_g) \psi_0^{-1} \subsetneq \Autz(Q_0)$.
\end{proof}

\section{Proofs of our main results in the \texorpdfstring{$3$}{3}-dimensional case}\label{sec: proof of Theorem list of some real MFS}

\begin{proof}[Proof of Theorem \ref{th: first cases, over k, dim 3}]
The base field $\k$ is assumed to be arbitrary of characteristic zero.
\begin{itemize}
\item If $X$ belongs to the family $\hypertarget{tth:D_a}{(a)}$ with parameter $a \geq 1$, then the result is Proposition \ref{prop: k-forms of Fabc with a>1}.
\item If $X$ belongs to the family $\hypertarget{tth:D_b}{(b)}$, then the result is Proposition \ref{prop: k-forms of PPb}.
\item If $X$ belongs to the family $\hypertarget{tth:D_c}{(c)}$, then the result is Proposition \ref{prop: k-forms of Umemura bundles}.
\item If $X$ belongs to the family $\hypertarget{tth:D_e}{(e)}$, then the result is Corollary \ref{cor: k-forms of Vb}.
\item If $X$ belongs to the family $\hypertarget{tth:D_f}{(f)}$, then the result is Proposition \ref{prop: k-forms of Wb}.
\item If $X$ belongs to the family $\hypertarget{tth:D_g}{(g)}$, then the result is Proposition \ref{prop: k forms of Rmn}.
\item If $X$ belongs to the family $\hypertarget{tth:D_i}{(i)}$, then the result is Ch\^atelet's theorem (Proposition \ref{prop:chatelet}).
\end{itemize}
\end{proof}

\vspace{-2mm}

\begin{proof}[Proof of Theorem \ref{th: second cases, over R, dim 3}]\item
\begin{itemize}
\item If $X=(\P_\C^1)^3$, then the result is Lemma \ref{lem: real forms of (P1)n}.
\item If $X=\FF_0^{b,c}$ with $b \geq 0$ and $(b,c)\neq(0,0)$, then the result is Proposition \ref{prop: k-forms of Fabc with a=0}.
\item If $X=\SS_{1,\C}$, then the result is Proposition \ref{prop: real forms of S1}.
\item If $X=\SS_{b,\C}$ with $b \geq 2$, then the result is Proposition \ref{prop: real forms of Sb}.
\item If $X =Q_3$, then the result is Proposition \ref{prop: real forms of Q3}.
\item If $X=\P(1,1,1,2)_\C$ or $\P(1,1,2,3)_\C$, then the result is Proposition \ref{prop: real forms of WPS}.
\item If $X$ is one of the complex Fano threefolds $Y_{5,\C}$ or $X_{12,\C}^{\mathrm{MU}}$, then the result is Proposition \ref{prop: real forms of Y5 and X12}.
\end{itemize}
\end{proof}

\vspace{-3mm}

\begin{proof}[Proof of Theorem \ref{th: third cases Qg, over R, dim 3}]
Part \ref{QQgR:1} is Proposition \ref{prop: existence real forms Qg}, part \ref{QQgR:2} is Proposition \ref{prop: aut group of Qg},  part \ref{QQgR:3} is Theorem \ref{th: real forms of Qg} and Proposition \ref{prop: real forms of Qg with two roots}, and part \ref{QQgR:4} follows from Proposition \ref{propo: first four real forms of Qg}. 
\end{proof}

\begin{proof}[Proof of Corollary \ref{cor: max subrgoups Cr3 real}]
The first part of the statement is Corollary \ref{cor:maximality-extension} with $\k_1=\R$ and $\k_2=\C$. The second part of the statement is a direct consequence of Theorems \ref{th: first cases, over k, dim 3}, \ref{th: second cases, over R, dim 3}, and \ref{th: third cases Qg, over R, dim 3}.
\end{proof}

\vspace{-3mm}

\begin{proof}[Proof of Theorem \ref{th: eq Sarkisov links between real MFS}]\item
\begin{itemize}
\item \ref{th: eq Sarkisov links between real MFS i} is Lemma \ref{lem:link homogeneous} and Proposition \ref{prop:Y5-X12-no-links}.
\item \ref{th: eq Sarkisov links between real MFS ii} is Lemma \ref{lem:link homogeneous}.
\item \ref{th: eq Sarkisov links between real MFS iii} is Lemma \ref{lem:links G_b} \ref{links G_b:1}-\ref{links G_b:3}.
\item \ref{th: eq Sarkisov links between real MFS iv} is Lemma \ref{lem:links G_b} \ref{links G_b:2}.
\item \ref{th: eq Sarkisov links between real MFS v} is Remark \ref{rk: cas of H1}.
\item \ref{th: eq Sarkisov links between real MFS vi} is Lemma \ref{lem:links H_b}.
\item \ref{th: eq Sarkisov links between real MFS vii} is Proposition \ref{prop:links Sb}.
\item \ref{th: eq Sarkisov links between real MFS viii} is Proposition \ref{prop: first links from Qg}.
\end{itemize}
\end{proof}


\bibliographystyle{abbrv}
\bibliography{biblio}

\end{document}